\documentclass{article}
\usepackage[margin=1in]{geometry}
\usepackage[T1]{fontenc}
\usepackage[utf8]{inputenc}
\usepackage{lmodern}
\usepackage{amsmath,amssymb,amsthm,mathtools,mathrsfs}
\usepackage[cal=euler]{mathalpha}
\usepackage{cochineal}
\usepackage{microtype}
\usepackage{hyperref}
\usepackage[numbers,sort&compress]{natbib}
\usepackage{enumitem}
\usepackage{booktabs}
\usepackage{array}
\usepackage{graphicx}
\usepackage{xcolor}

\numberwithin{equation}{section}

\theoremstyle{plain}
\newtheorem{theorem}{Theorem}[section]
\newtheorem{proposition}[theorem]{Proposition}
\newtheorem{lemma}[theorem]{Lemma}
\newtheorem{corollary}[theorem]{Corollary}

\theoremstyle{definition}
\newtheorem{conjecture}[theorem]{Conjecture}

\DeclareMathOperator{\KL}{D}
\DeclareMathOperator{\Var}{Var}
\DeclareMathOperator{\Cov}{Cov}

\DeclareMathOperator*{\esssup}{ess\,sup}
\newcommand{\E}{\mathbb{E}}
\newcommand{\PP}{\mathbb{P}}
\newcommand{\R}{\mathbb{R}}
\newcommand{\N}{\mathbb{N}}
\newcommand{\cF}{\mathcal{F}}

\newcommand{\cG}{\mathcal{G}}
\newcommand{\cX}{\mathcal{X}}
\newcommand{\cB}{\mathcal{B}}
\newcommand{\cT}{\mathcal{T}}
\newcommand{\cC}{\mathcal{C}}
\newcommand{\cP}{\mathcal{P}}

\newcommand{\ind}{\mathbf{1}}

\newcommand{\dd}{\,\mathrm{d}}
\newcommand{\Amart}[1]{A^{#1}}
\newcommand{\clock}[1]{F^{#1}}

\title{Information on trajectories:\\
martingales and random times}
\author{Akshay Balsubramani \\ {\small \texttt{akshay@vac.bio}}}
\date{}

\providecommand{\authcmt}[2]{\textcolor{#1}{#2}}
\providecommand{\akshay}[1]{\authcmt{red}{[AB: #1]}}
\providecommand{\vac}[1]{\authcmt{blue}{[VAC: #1]}}
\renewcommand{\akshay}[1]{}
\renewcommand{\vac}[1]{}
\makeatletter
\@ifundefined{diam}{}{}
\@ifundefined{conjecture}{}{}
\@ifundefined{problem}{}{}
\makeatother
\providecommand{\cA}{\mathcal{A}}
\providecommand{\cE}{\mathcal{E}}

\providecommand{\cR}{\mathcal{R}}

\providecommand{\cV}{\mathcal{V}}

\makeatletter
\def\bvapx@title{}
\AtBeginDocument{\providecommand{\phantomsection}{}}
\let\bvapx@maketitle\maketitle
\def\maketitle{\global\let\bvapx@title\@title\bvapx@maketitle}
\newcommand{\appendixtitle}{%
  \clearpage
  \suppressfloats[t]%
  \phantomsection
  \begin{center}
    {\LARGE\bfseries Appendices\ifx\bvapx@title\@empty\else\space to ``\bvapx@title''\fi\par}%
  \end{center}
  \vspace{1.5\baselineskip}%
}
\makeatother

\allowdisplaybreaks[2]
\AtBeginDocument{%
  \setlength{\abovedisplayskip}{6pt plus 2pt minus 2pt}%
  \setlength{\belowdisplayskip}{6pt plus 2pt minus 2pt}%
  \setlength{\abovedisplayshortskip}{3pt plus 1pt minus 1pt}%
  \setlength{\belowdisplayshortskip}{3pt plus 1pt minus 1pt}%
}

\providecommand{\akshay}[1]{}\renewcommand{\akshay}[1]{}%
\providecommand{\vac}[1]{}\renewcommand{\vac}[1]{}%
\begin{document}
\maketitle

% ======================================================================
\begin{abstract}
Accounting for information flow on the path space of trajectories of a nonnegative martingale yields exact variational identities for it, even at arbitrary random times. 
This recovers the widely used classical concentration inequalities, from Ville to PAC-Bayes, and measures what each one discards. 
The tail a bound controls is itself a relative entropy, resolved by the chain rule into per-step conditional divergences. 
The discarded slack has a closed form in each of three geometries: the crossing itself for Ville's inequality and for pooled tests, a Gibbs tilt for the Azuma--Hoeffding and PAC-Bayes bounds, and a dominating certificate for the $L^p$ maximal bound.
On a path-time space, the same identity gains one factor that measures anticipation: an arbitrary random time incurs an e-process ``peeking penalty.''
The partition function can be read as a coalescent---a prefix-sharing probability of independent copies---and geometric mixtures of test martingales gain a pooling benefit for multi-model safe testing.
\end{abstract}

% ======================================================================
\section{Introduction}
\label{sec:intro}
% ======================================================================

A martingale concentration inequality is proved by dropping nonnegative terms,
thereby discarding information.
This paper accounts for that information.
An identity of relative entropy, placed on the path space of a nonnegative martingale, names what each of the classical bounds treated here throws away.
The tail a bound controls is itself a relative entropy, of which the bound captures one part and discards the rest (Figure~\ref{fig:decompositions}); the discarded part is written in closed form for the bounds from Ville's inequality through the Azuma--Hoeffding and PAC-Bayes families to Doob's $L^p$ maximal bound.

The \emph{mixed coincidence identity}
\cite{balsubramani2026information} is an elementary manipulation of relative entropy:
the logarithm of a mixed partition function, built from products of powers of nonnegative factors, equals a variational optimum over distributions, and the gap at any candidate distribution is its relative entropy to the Gibbs optimizer.
Every nonnegative martingale of unit mean is a product of one-step conditional likelihood ratios.
Placing the mixed coincidence identity on the path space of such a martingale therefore turns it into a \emph{sequential} variational equality for the martingale's R\'enyi moments.
The log of an arbitrary product of powers of the one-step ratios equals an explicit supremum over path measures, with residual the relative entropy to a Gibbs path-tilt, and the chain rule resolves that residual into per-step conditional divergences.  Every concentration bound,
line-crossing inequality, PAC-Bayes statement, and random-time peeking penalty below is Theorem~\ref{thm:pathspace} read at a particular choice of factors: the path-space calculus states equalities where the classical theory gives one-sided bounds.

The bounds this covers are standard tools of concentration and statistical inference, anytime-valid and fixed-sample alike---Ville's maximal inequality, Azuma--Hoeffding, the variance-sensitive Bennett--Bernstein--Freedman family, parameter-mixture boundaries, PAC-Bayes---each classically proved by its own ad hoc argument, and all recovered here by the same route \cite{balsubramani2026information}.
Doob's $L^p$ maximal inequality is the exception, and gets its own residual in the geometry of power means $\|X\|_p=\E[|X|^p]^{1/p}$, where the extremal object is a dominating certificate function in place of a Gibbs tilt.
The same residual then measures the two phenomena a static bound cannot see: the anticipation carried by a random time, and the disagreement within a pool of competing tests.

\begin{figure}[!ht]
\centering
\includegraphics[width=0.98\textwidth]{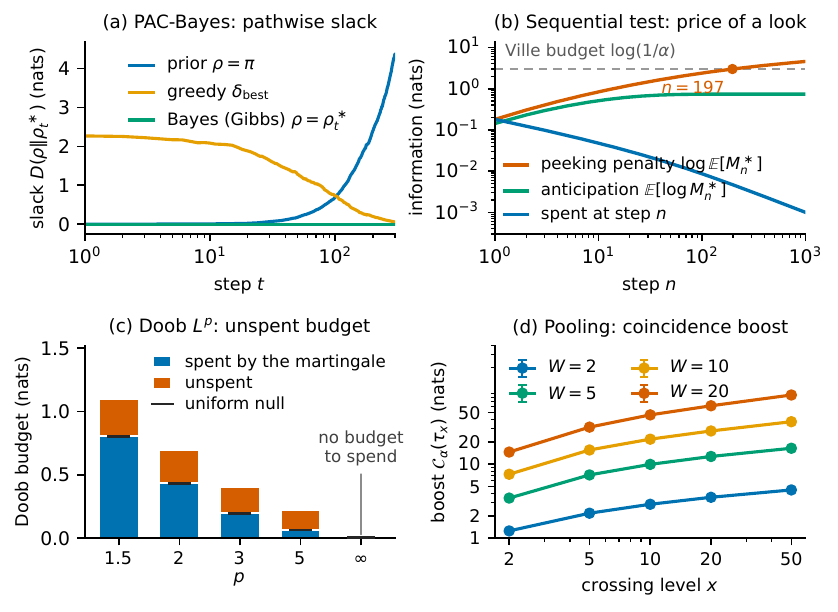}
\caption{\textbf{Each classical inequality leaves a residual the path-space calculus names and measures.}
All four residuals are drawn in nats on a common axis.
The pathwise PAC-Bayes slack \textbf{(a)} is the divergence $\KL(\rho\|\rho^\ast_t)$ to the running Gibbs tilt: identically zero at the Bayes posterior at every step, growing for a prior-anchored posterior as the tilt departs the prior, and decaying for a greedy one as the tilt concentrates.
A Gaussian test martingale evaluated at the time of its running maximum $M^\ast_n:=\max_{t\le n}M_t$ incurs the peeking penalty $\log\E[M^\ast_n]$ \textbf{(b)}.
It passes the entire Ville budget $\log(1/\alpha)=2.996$ nats at $\alpha=0.05$ by step $197$ and keeps growing, while the anticipation it measures saturates at $0.739$ nats---below the one nat a continuous path would incur, short by the overshoot of a discrete walk.
Taking logarithms through Doob's $L^p$ inequality \textbf{(c)} turns its sharp constant into an additive budget of $\log\tfrac{p}{p-1}$ nats.
The martingale spends the logarithm of its maximal-to-terminal norm ratio and leaves the rest; the budget vanishes as $p\to\infty$.
This is the one residual here that is \emph{not} a relative entropy, and Section~\ref{sec:doob-bellman} resolves it.
Pooling $W$ test martingales \textbf{(d)} tightens Ville's bound by the crossing form of the multi-way coincidence divergence, $\cC^{\mathrm{cross}}_\alpha(x)=\log[(1/x)/\PP(\tau^{(\alpha)}_x<\infty)]$, with $\tau^{(\alpha)}_x$ the first time the pooled process reaches $x$.
It grows in both the crossing level $x$ and the bettor count $W$ (logarithmic axes), reaching tens of nats as independent disagreement accumulates.
Bootstrap $95\%$ intervals lie within the markers except in the deepest tail.}
\label{fig:decompositions}
\end{figure}

\subsection{Main contributions}
\label{sec:contributions}

Placed on filtered path spaces, the mixed coincidence identity becomes a sequential variational equality for the R\'enyi moments of likelihood-ratio martingales (Theorem~\ref{thm:seqMCI}).
It is a reading of the finite-measure Donsker--Varadhan formula on path space (Theorem~\ref{thm:pathspace}), and every inequality below descends from it.
For integer exponents its partition function has a literal probabilistic meaning, the probability that independent copies share a common path prefix; that free energy decomposes into one-step predictive-coincidence costs, and for finite-state Markov families its growth rate is a Perron--Frobenius spectral radius.

The classical tails then become equalities.
Ville's and the Azuma--Hoeffding tails are relative-entropy equalities, the latter resolved by the chain rule into a sum of per-step conditional divergences of the conditioned law.
The gap to each classical bound is then attributed: the mean overshoot at first passage (Theorem~\ref{thm:firstpassage-exact}) for Ville, and for Azuma--Hoeffding the event relaxation, the sub-optimality of the exponential tilt, and the cumulant majorant.
One master identity stands behind them.
The exponential martingale of the running conditional log-CGF obeys an exact entropic equality (Theorem~\ref{thm:master}).
Replacing the exact cumulant by a predictable majorant---the supermartingale relaxation (Proposition~\ref{prop:supermart-relax})---accounts for the replacement exactly, and yields the entropic inequality behind Ville, Azuma--Hoeffding, Bennett--Bernstein--Freedman, parameter-mixture boundaries and PAC-Bayes.
The PAC-Bayes bound is itself an identity, discarding exactly the divergence from the posterior to the Gibbs tilt the loss induces \cite{balsubramani2026information}, and the remaining families follow as the entropic inequality specialized to their cumulant bounds.
The same accounting crosses a boundary of arbitrary form.
The discrete-time crossing law (Theorem~\ref{thm:curved-crossing-discrete}) leaves a quadrature residual and an overshoot that continuous time sets to zero, and the classical continuous statement is that identity with both vanishing.

Doob's $L^p$ inequality is the exception treated here, its residual a power-mean quantity sharpened in that geometry.
The role the Gibbs tilt plays there is played by the Az\'ema--Yor family of certificates indexed by an arbitrary concave function.
Each member is affine in the running value and is flat to first order when the running maximum advances, and its optional-stopping deficit resolves step by step into Bregman divergences between consecutive records (Proposition~\ref{prop:ay-certificate}), the second-order remainder the smooth fit leaves.
Specialized to the conjugate power and combined with H\"older's inequality, the family writes Doob's slack exactly as three pieces: a H\"older deficit, an optional-stopping deficit, and the initial value $M_0^p/(p-1)$ (Corollary~\ref{cor:doob-residual}).
The per-step resolution the chain rule supplies in the entropic geometry is therefore available in the power-mean geometry too, with a Bregman divergence in place of a relative entropy.
Each classical inequality thus belongs to whichever of the three geometries---Gibbs, power-mean, optional stopping---makes its extremizer the optimizer.

Disagreement among several tests is itself evidence.
A geometric mixture of test martingales is a supermartingale whose Ville bound is tightened by the multi-way coincidence divergence, a dynamical benefit of model disagreement.
Set that benefit beside the mass that never reaches the level and the overshoot at the crossing, and the inequality goes.
The pooled crossing probability is then closed form (Corollary~\ref{cor:pooling-crossing-exact}).

The cost of a random time closes the account.
An arbitrary random time splits into a hazard clock the history already determines and a martingale factor that absorbs everything it does not (Theorem~\ref{thm:surv-factor}).
That factor is trivial exactly for stopping times, and more generally for pseudo-stopping times, so it is the sole carrier of anticipation; expectations at the random time are expectations against the hazard clock, reweighted by it.
On path-time space the identity gains exactly one further factor, and that factor measures the anticipation (Theorem~\ref{thm:pathtime}).
Reading any evidence process at an arbitrary random time then costs a peeking penalty in nats (Theorem~\ref{thm:peeking}).
The penalty is at most zero at a stopping time and exactly zero at a pseudo-stopping time for uniformly integrable martingales, while a time that looks ahead can incur an unbounded one.
Between those ends the worst case over e-processes has an exact value, the supremum norm of the time's anticipation index (Theorem~\ref{thm:peeking-radius}), and it is zero precisely on the pseudo-stopping class.

% ======================================================================
\section{Information-theoretic preliminaries}
\label{sec:info-background}
% ======================================================================

Every information quantity used below is a reading of the mixed coincidence identity \cite{balsubramani2026information}; this section fixes notation and recalls the identity and the specializations the path-space sections use.

\subsection{Entropies and divergences}

Let $(\cX, \cB, \nu)$ be a $\sigma$-finite measure space.
A \emph{distribution} on $\cX$ is a probability measure $P$ absolutely continuous with respect to $\nu$; we write $p := \mathrm{d}P/\mathrm{d}\nu$ for its density.
More generally, a nonnegative \emph{factor} (or unnormalized density) $\pi$ is a measurable map $\pi \colon \cX \to [0,\infty)$ with $0 < \|\pi\|_1 := \int \pi\,\mathrm{d}\nu < \infty$.

For a distribution $P$ with density $p$ and a strictly positive factor $\pi$, the \emph{entropy} and \emph{cross-entropy} are $H(p) := -\E_{X\sim P}[\log p(X)]$ and $H(p, \pi) := -\E_{X\sim P}[\log \pi(X)]$.
The \emph{relative entropy} of $p$ relative to $\pi$, written $\KL(p\|\pi)$, is $
  \KL(p\|\pi) \;:=\; H(p,\pi) - H(p)
  \;=\; \E_{X\sim P}\!\left[\log\frac{p(X)}{\pi(X)}\right]
$.
When $\pi$ is itself a probability density, $\KL(p\|\pi) \geq 0$, with equality if and only if $p = \pi$ $\nu$-almost everywhere (Gibbs' inequality).

Relative entropy is also needed against a finite positive measure that need not be a probability.
For $\mu$ a finite positive measure on $(\cX,\cB)$ and $Q$ a probability measure on $(\cX,\cB)$ with $Q \ll \mu$,
\begin{equation}\label{eq:finite-entropy}
  \KL(Q\|\mu)
  := \int_\cX \log\!\left(\frac{\mathrm{d}Q}{\mathrm{d}\mu}\right) \mathrm{d}Q
  \;\in\; [-\log \mu(\cX),\; +\infty]
\end{equation}
and $\KL(Q\|\mu) := +\infty$ if $Q \not\ll \mu$.

This is the usual relative entropy shifted by $-\log \mu(\cX)$, reducing to the probability-measure case when $\mu(\cX) = 1$; it is the form needed against the \emph{path law} $P|_{\cF_T}$, of total mass $1$, and against the \emph{path-time measure} $\widehat P$ of Section~\ref{sec:pathtime}, of total mass at most one and in general strictly less.

For $\alpha > 0$, $\alpha \neq 1$, and distributions $P_1,P_2$ with densities $p_1,p_2$, the \emph{R\'enyi entropy} and \emph{R\'enyi divergence} are
\begin{align}
  H_\alpha(p)
  &:= \frac{1}{1-\alpha}\log\int_\cX p(x)^\alpha\,\mathrm{d}\nu(x)
     \label{eq:renyi-ent} \\
  D_\alpha(p_1\|p_2)
  &:= \frac{1}{\alpha-1}\log\int_\cX p_1(x)^\alpha\,p_2(x)^{1-\alpha}\,
     \mathrm{d}\nu(x)
     \label{eq:renyi-div}
\end{align}
In the limit $\alpha \to 1$, $H_\alpha(p) \to H(p)$ and $D_\alpha(p_1\|p_2) \to \KL(p_1\|p_2)$.
The R\'enyi entropy is nonincreasing in $\alpha$; the R\'enyi divergence is nondecreasing in $\alpha$, nonnegative, and satisfies the data-processing inequality: $D_\alpha(P_1 K \,\|\, P_2 K) \leq D_\alpha(P_1\|P_2)$ for any Markov kernel $K$.

\subsection{The mixed coincidence identity}
\label{sec:mci}

The identity below is due to~\cite{balsubramani2026information}; it is recalled in the notation the path-space sections use, and the path-space identity of the next section is this one placed on a filtered space.
The name \emph{coincidence} marks the free energy and the variational problem coinciding; a distinct sense of the word surfaces on path space, where the \emph{coincidence divergence} of the coalescent reading (Proposition~\ref{prop:coalescence}) measures the probability that independent copies share a common path prefix.
Stating the identity for a mixture of factors raised to arbitrary powers lets one statement specialize to Donsker--Varadhan, R\'enyi, and PAC-Bayes alike.

Fix nonnegative factors $\pi_1,\ldots,\pi_W$ on $(\cX,\nu)$ and an exponent vector $\alpha = (\alpha_1,\ldots,\alpha_W) \in \R^W$.
The \emph{mixed partition function} is
\begin{equation}\label{eq:Zalpha}
  Z(\alpha) := \int_\cX \prod_{i=1}^W \pi_i(x)^{\alpha_i}\,\dd\nu(x)
             = \E_{X\sim\nu}\!\left[\prod_{i=1}^W \pi_i(X)^{\alpha_i}\right],
\end{equation}
assumed to satisfy $0 < Z(\alpha) < \infty$.
The \emph{geometric-mixture density} is
\begin{equation}\label{eq:pstar}
  p^*_\alpha(x) := \frac{\prod_{i=1}^W \pi_i(x)^{\alpha_i}}{Z(\alpha)}.
\end{equation}
Write $\bar\alpha := \sum_{i=1}^W \alpha_i$ and, when $\bar\alpha > 0$,
$\widetilde\alpha := \alpha/\bar\alpha \in \Delta([W])$ for the normalized weights.

Sequential settings reuse this object with $\pi_i$ replaced by one-step likelihood ratios and the exponents by a time-indexed vector; that is the bridge to martingale theory.

\begin{theorem}[Mixed coincidence identity \cite{balsubramani2026information}]
\label{thm:mci}
Let $\pi_1,\ldots,\pi_W$ be nonnegative factors on $(\cX,\nu)$, let $\alpha \in \R^W$ satisfy $0 < Z(\alpha) < \infty$ for $Z$ of~\eqref{eq:Zalpha}, and let $p^*_\alpha$ be as in~\eqref{eq:pstar}.
For \emph{every} distribution $p$ on $(\cX,\nu)$ with $\E_p[\log \pi_i(X)]$ finite for all $i$,
\begin{equation}\label{eq:mci}
  -\log Z(\alpha) + \KL(p\|p^*_\alpha)
  = \sum_{i=1}^W \alpha_i\,\KL(p\|\pi_i)
    + (\bar\alpha - 1)\,H(p).
\end{equation}
Consequently,
\begin{equation}\label{eq:mci-var}
  \log Z(\alpha)
  = \max_{p \in \Delta(\cX)}
    \left[
      H(p) - \sum_{i=1}^W \alpha_i\,H(p,\pi_i)
    \right]
  = -\min_{p \in \Delta(\cX)}
    \left[
      \sum_{i=1}^W \alpha_i\,\KL(p\|\pi_i)
      + (\bar\alpha - 1)H(p)
    \right],
\end{equation}
with unique optimum $p = p^*_\alpha$.
\end{theorem}

Equation~\eqref{eq:mci} is an equality at \emph{every} distribution $p$,
including the non-optimal ones: there the gap $\KL(p\|p^*_\alpha)$ is the information cost of using $p$ instead of the optimal geometric mixture.

\subsection{Donsker--Varadhan, R\'enyi and multi-prior specializations}

Each classical information-theoretic object below is Theorem~\ref{thm:mci} read at a particular choice of factors and exponents, and each reappears on path space later.
The consequence used most, and the one the peeking penalty of Section~\ref{sec:pathtime} rests on, is the following.

\begin{corollary}[Donsker--Varadhan / finite-measure Gibbs formula\cite{balsubramani2026information}]
\label{cor:DV}
Let $\mu$ be a finite positive measure on $(\cX,\cB)$ and let $g \colon \cX \to [0,\infty)$ be measurable with $0 < \int g\,\dd\mu < \infty$.
Let $Q^\star$ be the Gibbs measure $\dd Q^\star/\dd\mu = g / \int g\,\dd\mu$, a probability measure supported on $\{g > 0\}$.
Then for every probability measure $Q \ll \mu$ with $Q(\{g = 0\}) = 0$,
\begin{equation}\label{eq:DV}
  \log \int g\,\dd\mu
  = \E_Q[\log g] - \KL(Q\|\mu) + \KL(Q\|Q^\star).
\end{equation}
Consequently,
\begin{equation}\label{eq:DV-var}
  \log \int g\,\dd\mu
  = \sup_{Q \ll \mu}\bigl\{\E_Q[\log g] - \KL(Q\|\mu)\bigr\},
\end{equation}
with unique maximizer $Q = Q^\star$ (a $Q$ charging $\{g = 0\}$ has $\E_Q[\log g] = -\infty$ and is never optimal, so the supremum is unchanged by restricting to $Q \ll Q^\star$).
\end{corollary}

Corollary~\ref{cor:DV} is the ``one-factor'' specialization of Theorem~\ref{thm:mci}: take $\nu = \mu$, $W = 1$, $\pi_1 = g$,
$\alpha_1 = 1$.
The strictly positive case $g \colon \cX \to (0,\infty)$ is recovered when $\{g=0\}$ is $\mu$-null; the indicator case $g = \ind_A$ (with $\mu(A) > 0$) recovers the conditional form $\log \mu(A) = \sup_{Q : Q(A) = 1}\{-\KL(Q\|\mu)\}$, optimized by $Q^\star = \mu(\cdot \mid A)$.
It is the DV principle for finite measures, so \emph{every} downstream result in this paper can be viewed as an application of the DV formula to a carefully chosen $(\mu, g)$ on path space.

Two named R\'enyi quantities are the same identity read at a different exponent.
One density at a free power gives $Z(\alpha)=\int p^\alpha\,\dd\nu$ and the R\'enyi entropy $H_\alpha(p)=(1-\alpha)^{-1}\log Z(\alpha)$, whose variational form $-(1-\alpha)H_\alpha(p)=\min_w\{\alpha\KL(w\|p)+(\alpha-1)H(w)\}$ is minimized by the \emph{escort distribution} $w^*_\alpha\propto p^\alpha$.
Two densities at complementary powers $(\alpha,1-\alpha)$ give the R\'enyi divergence as a relative-entropy \emph{barycenter} --- the distribution minimizing a weighted sum of divergences to the given ones --- $-(\alpha-1)D_\alpha(p_1\|p_2)=\min_w\{\alpha\KL(w\|p_1)+(1-\alpha)\KL(w\|p_2)\}$, minimized by $w^*\propto p_1^\alpha p_2^{1-\alpha}$~\cite{balsubramani2026information}.

\begin{proposition}[Multi-prior PAC-Bayes / multi-way coincidence\cite{balsubramani2026information}]
\label{prop:multi-PAC}
Let $\pi_1,\ldots,\pi_W$ be probability distributions and $\alpha \in \Delta([W])$.
Define the \emph{multi-way coincidence divergence}
$\cC_\alpha(\pi_{1:W})
  := -\log Z(\alpha)
  = \min_p \sum_{i=1}^W \alpha_i\,\KL(p\|\pi_i)$. 
Then for any $\rho$:
$\KL(\rho\|p^*_\alpha)
  = \sum_{i=1}^W \alpha_i\,\KL(\rho\|\pi_i) - \cC_\alpha(\pi_{1:W})
$. 
\end{proposition}

The last specialization runs the other way: a relative-entropy budget determines how far a change of measure can move a mean.
For $Z$ integrable under a reference $P$ and a convex $\phi$ dominating its centered cumulant, $\log\E_P[e^{\lambda(Z-\E_P[Z])}]\le\phi(\lambda)$ on $(0,b)$, the transportation lemma reads $\E_Q[Z]-\E_P[Z]\le\phi^{*,-1}(\KL(Q\|P))$ for every $Q\ll P$, with $\phi^*$ the convex conjugate and $\phi^{*,-1}$ its generalized inverse; in the sub-Gaussian case $\phi(\lambda)=v\lambda^2/2$ this is $\sqrt{2v\,\KL(Q\|P)}$~\cite{balsubramani2026information}.

% ======================================================================
\section{The path-space mixed coincidence identity}
\label{sec:pathspace}
% ======================================================================
The identity of the previous section is static---it lives on one measure space and knows nothing of time.
The focus of this paper, however, is dynamic: nonnegative martingales, the supermartingales that certify sequential tests, and the random times at which one might stop and look.
Placing the identity on path space is a matter of choosing the right substrate among them; no new machinery is required.
The machinery that zeros in on the time of evaluation---optional stopping, first passage, the value at the ultimate maximum---enters with the inequalities of Section~\ref{sec:master}; up to that point everything is algebra that holds at any fixed time.

We work throughout in discrete time.
Fix a filtered probability space $(\Omega,\cF,(\cF_t)_{t \in \N_0}, P)$ with $\cF_0 \subseteq \cF_1 \subseteq \cdots \subseteq \cF$.
All processes are real-valued and adapted unless stated otherwise, and a process is \emph{predictable} when its value at time $t$ is $\cF_{t-1}$-measurable, so that the history strictly before $t$ determines it.
Notation is collected in Appendix~\ref{app:glossary}.

An adapted process $(M_t)_{t \in \N_0}$ is a \emph{martingale} if $\E[|M_t|] < \infty$ and $\E[M_t \mid \cF_{t-1}] = M_{t-1}$ for all $t \geq 1$; a \emph{supermartingale} if $\E[M_t \mid \cF_{t-1}] \leq M_{t-1}$; a \emph{nonnegative supermartingale} or \emph{test supermartingale} if additionally $M_t \geq 0$ a.s.  A nonnegative martingale with $M_0 = 1$ is a \emph{test martingale} (or \emph{likelihood-ratio martingale}).

A random variable $\tau \colon \Omega \to \N \cup \{\infty\}$ is a \emph{stopping time} (for $(\cF_t)$) if $\{\tau \leq t\} \in \cF_t$ for all $t$.
A \emph{random time} is any positive-integer-valued random variable (not necessarily adapted to $(\cF_t)$).
A \emph{pseudo-stopping time}~\cite{NY05} is a random time at which every bounded martingale keeps its initial mean, $\E[B_\tau] = \E[B_0]$; the stopping times are the familiar instance, and Section~\ref{sec:survival} identifies the class exactly.

Two ingredients extend the static identity of Section~\ref{sec:mci} to path space:
\begin{itemize}[leftmargin=2em]
\item \emph{Path-space substrate.}  Take $(\cX,\nu) = (\Omega, P|_{\cF_T})$ itself, the path measure.
Now ``factors'' $\Pi_{i,T}$ are $\cF_T$-measurable random variables on the path space.
\item \emph{Likelihood-ratio factors.}  In the sequential setting the natural choice is $\Pi_{t,T} = R_t$, the one-step likelihood ratio of an alternative path measure against $P$ at time $t$.
Taking exponents $\alpha_t$ indexed by time produces, by the chain rule, the identity for the cumulative likelihood ratio raised to a time-dependent power.
\end{itemize}

\paragraph{A one-step example.}
Take $T = 1$, $W = 1$, exponent $\alpha = 1$, with binary state space $\{ \omega_+ , \omega_- \}$,
$P(\omega_+) = P(\omega_-) = 1/2$, and $Q$ with $Q(\omega_+) = 2/3$,
$Q(\omega_-) = 1/3$, so the likelihood ratio $R_1 = \dd Q/\dd P$ takes values $4/3$ and $2/3$ and $\E_P[R_1] = 1$.
The static identity (Theorem~\ref{thm:mci}) with $\pi_1 = R_1$, $\alpha_1 = 1$ reads $
  \log\E_P[R_1] = 0
  = \sup_{p \in \Delta(\cX)}\bigl\{\E_p[\log R_1] - \KL(p\|P)\bigr\}
$. 
The supremum is attained at $p^* = Q$, where $\E_Q[\log R_1] - \KL(Q\|P) = 0$ since $\E_Q[\log(\dd Q/\dd P)] = \KL(Q\|P)$.
The form of the variational problem matters: the optimal tilt is the alternative measure $Q$, and the supremum value $0$ records the fact that $Q$ is normalized.
Placed on a path space of length $T$ with i.i.d.\ binary increments and one-step likelihood ratios $R_t$
(Section~\ref{sec:pathspace}), the same identity decomposes additively in $t$ via the chain rule, producing per-step variational tradeoffs.
Each step's tradeoff is the input to one of the concentration bounds proved later.

\subsection{Path measures and likelihood ratios}

Fix a horizon $T \in \N$.  A \emph{path} is $\omega = (x_1,\ldots,x_T)$;
the law of the path under $P$ has density $p(\omega) = \prod_{t=1}^T p_t(x_t \mid x_{1:t-1})$ with respect to the product reference measure $\nu^{\otimes T} := \nu_1 \otimes \cdots \otimes \nu_T$, where $p_t(\cdot \mid x_{1:t-1})$ is the conditional density of $X_t$ given the history under $P$.

For an alternative path measure $Q \ll P$ with conditional densities $q_t(\cdot \mid x_{1:t-1})$, the \emph{one-step likelihood ratio} at time $t$ is
\begin{equation}\label{eq:one-step-LR}
  R_t(\omega) := \frac{q_t(x_t \mid x_{1:t-1})}
                       {p_t(x_t \mid x_{1:t-1})},
\end{equation}
and the \emph{cumulative likelihood ratio} is $L_T := \prod_{t=1}^T R_t = \dd Q / \dd P$ on $\cF_T$.
The process $(L_t)_{t=0}^T$ with $L_0 = 1$ is a nonnegative unit-mean $P$-martingale (the Radon--Nikodym martingale).

Conversely, any nonnegative martingale $(M_t)$ with $M_0 = 1$ and $\E_P[M_T] = 1$ defines an absolutely continuous measure $Q$ on $\cF_T$ via $\dd Q / \dd P |_{\cF_T} = M_T$, with one-step ratios $R_t = M_t/M_{t-1}$.
The path-space theory below therefore applies to \emph{all} nonnegative unit-mean martingales.
Throughout, $M$ denotes a generic such martingale; the particular martingale attached to a random time $\tau$ in Section~\ref{sec:survival} is written $\Amart{\tau}$, naming the time it is built from.

\begin{lemma}[Chain rule for relative entropy]
\label{lem:chain-rule}
For path measures $Q \ll P$ on $(\Omega, \cF_T)$,
\[
  \KL(Q\|P)
  = \sum_{t=1}^T \E_Q\!\left[
      \KL\!\bigl(Q_t(\cdot | X_{1:t-1}) \,\|\, P_t(\cdot | X_{1:t-1})\bigr)
    \right]
  = \sum_{t=1}^T \E_Q[\log R_t].
\]
\end{lemma}

\subsection{The path-space partition function and identity}

For a horizon $T$, strictly positive $\cF_T$-measurable factors $\Pi_{1,T},\dots,\Pi_{W,T}$, and an exponent vector $\alpha = (\alpha_1,\dots,\alpha_W) \in \R^W$, define
\begin{equation}\label{eq:pathZ}
  Z_T(\alpha)
  := \E_P\!\left[\prod_{i=1}^W \Pi_{i,T}^{\alpha_i}\right],
\end{equation}
assumed to satisfy $0 < Z_T(\alpha) < \infty$.
In the likelihood-ratio case, taking $\Pi_{t,T} = R_t$ and an exponent vector $\alpha \in \R^T$ indexed by time yields
\begin{equation}\label{eq:seqZ}
  Z_T(\alpha) := \E_P\!\left[\prod_{t=1}^T R_t^{\alpha_t}\right]
              = \E_P\!\bigl[L_T^{(\alpha)}\bigr],
  \qquad
  L_T^{(\alpha)} := \prod_{t=1}^T R_t^{\alpha_t}.
\end{equation}

The free energy $\log Z_T(\alpha)$ is readable from every path measure at once.
Each candidate $Q$ reads it as a tilted expectation of the factors net of the entropy cost of using $Q$ in place of $P$, and the amount by which that reading falls short is the relative entropy separating $Q$ from the Gibbs path law.

\begin{theorem}[Path-space mixed coincidence identity]
\label{thm:pathspace}
For $Z_T$ of~\eqref{eq:pathZ} with $0 < Z_T(\alpha) < \infty$, for every horizon $T$, every collection of positive $\cF_T$-measurable factors $(\Pi_{i,T})_{i=1}^W$, and every path measure $Q \ll P|_{\cF_T}$,
\begin{equation}\label{eq:pathspace}
  \log Z_T(\alpha)
  = \sum_{i=1}^W \alpha_i\,\E_Q[\log \Pi_{i,T}]
    - \KL(Q\|P|_{\cF_T})
    + \KL(Q\|Q_T^\alpha),
\end{equation}
where $Q_T^\alpha$ is the Gibbs law
\begin{equation}\label{eq:Qtstar}
  \frac{\dd Q_T^\alpha}{\dd P|_{\cF_T}}
  = \frac{\prod_{i=1}^W \Pi_{i,T}^{\alpha_i}}{Z_T(\alpha)}.
\end{equation}
The residual $\KL(Q\|Q_T^\alpha) \geq 0$ vanishes if and only if $Q = Q_T^\alpha$, so the identity is equivalent to the variational form $\log Z_T(\alpha) = \sup_{Q \ll
P|_{\cF_T}}\bigl\{\sum_i \alpha_i\,\E_Q[\log \Pi_{i,T}] - \KL(Q\|P|_{\cF_T})\bigr\}$,
attained uniquely at $Q_T^\alpha$.
\end{theorem}

\subsection{The sequential identity}

The sequential case is this theorem read at one particular choice of factors: take $W = T$ factors $\Pi_{t,T} := R_t$, the one-step likelihood ratios, with exponents $\alpha \in \R^T$ indexed by time.
Then $\prod_t \Pi_{t,T}^{\alpha_t} = L_T^{(\alpha)}$, and the Gibbs law $Q_T^\alpha$ of~\eqref{eq:Qtstar} is the tilt $R^*_\alpha$ with $\dd R^*_\alpha/\dd P = L_T^{(\alpha)}/Z_T(\alpha)$.

\begin{theorem}[Sequential mixed coincidence identity]
\label{thm:seqMCI}
Let the one-step likelihood ratios $R_1,\dots,R_T$ be strictly positive, and let $\alpha \in \R^T$ satisfy $0 < Z_T(\alpha) < \infty$ for $Z_T$ of~\eqref{eq:seqZ}.
Then for every path measure $\widetilde Q \ll P|_{\cF_T}$,
\begin{equation}\label{eq:seqMCI}
  \log Z_T(\alpha)
  = \sum_{t=1}^T \alpha_t\,\E_{\widetilde Q}[\log R_t]
    - \KL(\widetilde Q\|P|_{\cF_T})
    + \KL(\widetilde Q\|R^*_\alpha)
  \;=\; \max_{\widetilde Q \ll P}
    \Bigl\{\textstyle\sum_{t=1}^T \alpha_t\,\E_{\widetilde Q}[\log R_t]
      - \KL(\widetilde Q\|P|_{\cF_T})\Bigr\},
\end{equation}
the maximum attained uniquely at $\widetilde Q = R^*_\alpha$.
\end{theorem}

Read at $T$ factors, the identity resolves a martingale's R\'enyi moments into an exact per-step trade of likelihood-ratio gain against information cost; the concentration theory of Section~\ref{sec:master} onward reads the same identity at the single factor $\cE_t(\lambda)$.

\subsection{Conditional decomposition and one-step Gibbs tilts}

The identity of the previous subsection accounts for a whole trajectory at once, through a single divergence $\KL(\widetilde Q\|P)$ on $\cF_T$.
A sequential problem needs that accounting step by step, and the chain rule supplies it: the path-space divergence is the sum of one conditional divergence per time step.
The per-step form that results turns the path-space accounting into a sequential one.

\begin{corollary}[Conditional form]
\label{cor:seqMCI-cond}
\begin{equation}\label{eq:seqMCI-cond}
  \log Z_T(\alpha)
  = \max_{\widetilde Q \ll P}
    \left\{
      \sum_{t=1}^T \E_{\widetilde Q}\!\left[
        \alpha_t \log R_t - \KL(\widetilde Q_t \| P_t)
      \right]
    \right\}
\end{equation}
where $\widetilde Q_t = \widetilde Q_t(\cdot \mid X_{1:t-1})$ and $P_t = P_t(\cdot \mid X_{1:t-1})$.
The maximizer $R^*_\alpha$ has one-step conditional densities
\begin{equation}\label{eq:Rstar-cond}
  r^*_{\alpha,t}(x_t \mid x_{1:t-1})
  \propto p_t(x_t \mid x_{1:t-1}) \,
          R_t(x_t; x_{1:t-1})^{\alpha_t}\,
          h_t(x_{1:t}),
  \qquad
  h_t(x_{1:t}) := \E_P\!\Bigl[{\textstyle\prod_{s > t}} R_s^{\alpha_s}
                    \,\Big|\, \cF_t\Bigr],
\end{equation}
where $h_t$ is the forward factor.
\end{corollary}

The per-step summand $\alpha_t \log R_t - \KL(\widetilde Q_t\|P_t)$ of the conditional identity (Corollary~\ref{cor:seqMCI-cond}) is a sequential likelihood-ratio gain net of an information cost.
Its optimizer $R^*_\alpha$ tilts each conditional law toward the alternative in the proportion $\alpha_t$, up to the forward $h$-transform factor of Corollary~\ref{cor:seqMCI-cond}---absent when the tilted increments are conditionally independent.
That is the Neyman--Pearson tradeoff resolved one step at a time: the sequential identity supplies as an exact per-step accounting what the sequential-probability-ratio test and the likelihood-ratio martingales of Wald and Robbins state as an inequality,
identifying $R^*_\alpha$ as the sequentially Neyman--Pearson-optimal alternative for the exponent schedule $\alpha$.

The forward factor keeps the one-step conditional from being local.
Reweighting the transitions of a process by a nonnegative function of the current state, and renormalizing, yields another Markov process; when the weight is the conditional probability of some future event, the reweighted process is that process conditioned on the event.
This reweighting is the Doob $h$-transform, and $h$ is the harmonic function that folds the future back into the present transition.
Here $h_t$ is the conditional expectation of the tilt still to come, $\prod_{s>t} R_s^{\alpha_s}$, so the optimal step at time $t$ already anticipates what the later exponents will ask of the path.

When $h_t$ does not depend on $x_t$---in particular when the tilted increments are conditionally independent, so that $\E_P[\prod_{s>t} R_s^{\alpha_s}\mid\cF_t]$ is $\cF_{t-1}$-measurable---the forward factor is constant in $x_t$, and the normalization in~\eqref{eq:Rstar-cond} absorbs it, leaving $r^*_{\alpha,t}\propto p_t\,R_t^{\alpha_t}$.
Substituting the one-step likelihood ratio $R_t=q_t/p_t$ of~\eqref{eq:one-step-LR} gives $r^*_{\alpha,t} \propto p_t^{1-\alpha_t}\,q_t^{\alpha_t}$, the local $\alpha_t$-geometric mixture of the null conditional $P_t$ and the alternative conditional $Q_t$.
The law being mixed is $Q_t$; the variational $\widetilde Q_t$ of~\eqref{eq:seqMCI-cond} has already been resolved by the maximization.

When the exponents $\alpha_t$ are positive integers, the likelihood-ratio partition function $Z_T(\alpha) = \E_P[\prod_t R_t^{\alpha_t}]$ of~\eqref{eq:seqZ} is a moment functional of the step likelihood ratios (R\'enyi or $\chi^2$-type), and routinely exceeds one (already $\E_P[R_t^2] = 1 + \chi^2(Q_t\|P_t) \ge 1$ at one step).
The genuine step-by-step \emph{coincidence probability} is instead carried by the probability-factor normalization $\Pi_{t,T}=P_t^{(i)}$ of Proposition~\ref{prop:coalescence}.
There $Z_t(\alpha)$ is the probability that, at each time $t$, $\alpha_i$ independent copies drawn from $P_t^{(i)}$ realize the same prefix, a temporal coalescent that decomposes additively by time.
Under the probability-factor normalization, $- \log Z_t(\alpha)$ is a nonnegative quantity, the \emph{coalescent free energy}.

The chain rule splits the single path-space divergence into one conditional divergence per step, so the sum over $t$ here plays the part the sum over priors plays in the static identity (Theorem~\ref{thm:mci}): each step contributes a gain $\alpha_t \log R_t$ against a cost $\KL(\widetilde Q_t\|P_t)$.

\subsection{Barycentric and convex forms}
\label{sec:barycentric}

Two readings of the same identity round out its geometry.
Applied to a family of tilted path laws in place of raw factors, it takes a barycentric form; as a function of the exponents it inherits the convexity of classical log-partition functions.

Against a family of tilted path laws $P_T^{(1)},\dots,P_T^{(W)}$ on $\cF_T$ dominated by $P|_{\cF_T}$, with $\Lambda_{i,T}:=\dd P_T^{(i)}/\dd P|_{\cF_T}$ and $\bar\alpha:=\sum_i\alpha_i$, the identity is a barycenter:
$\log Z_T(\alpha)=-\inf_{Q\ll P|_{\cF_T}}\{\sum_i\alpha_i\KL(Q\|P_T^{(i)})-(\bar\alpha-1)\KL(Q\|P|_{\cF_T})\}$, which at $\alpha\in\Delta([W])$ reduces to $-\log Z_T(\alpha)=\inf_Q\sum_i\alpha_i\KL(Q\|P_T^{(i)})$ with minimizer $Q_T^\alpha$ of~\eqref{eq:Qtstar}, the path-space reading of the multi-way coincidence divergence.
It follows from Theorem~\ref{thm:pathspace} on expanding each $\KL(Q\|P_T^{(i)})=\KL(Q\|P|_{\cF_T})-\E_Q[\log\Lambda_{i,T}]$ and collecting terms.

\begin{proposition}[Gradient and Hessian]
\label{prop:grad}
Assume differentiation under the integral is justified.
Set $\Phi_T(\alpha) := \log Z_T(\alpha)$.
Then $\Phi_T$ is convex on its domain, and
\begin{equation}\label{eq:grad}
  \frac{\partial \Phi_T}{\partial \alpha_i}(\alpha)
  = \E_{Q_T^\alpha}[\log \Pi_{i,T}],
  \qquad
  \frac{\partial^2 \Phi_T}{\partial \alpha_i \partial \alpha_j}(\alpha)
  = \Cov_{Q_T^\alpha}(\log \Pi_{i,T}, \log \Pi_{j,T}).
\end{equation}
\end{proposition}

\subsection{Coalescent interpretation and predictive decomposition}
\label{sec:coalescent-pathspace}

The path-space partition function has a direct coincidence meaning.
Let $S$ be a countable state space and $P_t^{(1)},\dots,P_t^{(W)}$ path laws on $S^{t+1}$ with integer exponents.

\begin{proposition}[Mixed prefix-coalescence probability]
\label{prop:coalescence}
Let $\alpha_1,\dots,\alpha_W \in \N$.
Then $\displaystyle Z_t(\alpha) := \sum_{\gamma \in S^{t+1}} \prod_{i=1}^W P_t^{(i)}(\gamma)^{\alpha_i}$ is the probability that, for each $i$, $\alpha_i$ independent copies sampled from $P_t^{(i)}$ all realize the same common prefix $\gamma$ up to time $t$.
\end{proposition}

The coalescent free energy of these prefix laws is $\cC_\alpha(t) := -\log Z_t(\alpha)$, and it decomposes additively by time.

\begin{theorem}[Predictive coincidence decomposition]
\label{thm:predictive}
Let $\alpha_i \geq 0$ for all $i$ with $\bar\alpha := \sum_{i=1}^W \alpha_i
\geq 1$, let $Q_t^\alpha(\gamma) := \prod_i
P_t^{(i)}(\gamma)^{\alpha_i}/Z_t(\alpha)$ be the Gibbs law on prefixes, and let
\begin{equation}\label{eq:kappa}
  \kappa_t^\alpha(\gamma)
  := \sum_{x\in S} \prod_{i=1}^W
     P\bigl(X_{t+1}=x \mid X_{0:t}=\gamma; P^{(i)}\bigr)^{\alpha_i}
\end{equation}
be the one-step predictive coincidence kernel.  Then
\begin{align}
  Z_{t+1}(\alpha) &= Z_t(\alpha)\,\E_{Q_t^\alpha}[\kappa_t^\alpha],
    \label{eq:predictive1}\\
  \cC_\alpha(t+1) - \cC_\alpha(t)
    &= -\log \E_{Q_t^\alpha}[\kappa_t^\alpha] \geq 0 .
    \label{eq:predictive2}
\end{align}
\end{theorem}

The increment $\cC_\alpha(t+1)-\cC_\alpha(t)$ is the one-step predictive-coincidence cost, the per-step growth of the coalescent free energy, and its nonnegativity makes $\cC_\alpha(t)$ nondecreasing in $t$.

For Markov families the free energy has a closed asymptotic form, read off a single matrix.
A nonnegative matrix that is irreducible (every state reaches every other) and aperiodic has, by the Perron--Frobenius theorem~\cite{seneta2006}, a largest eigenvalue that is real, positive, simple, and strictly larger in modulus than every other eigenvalue, with positive left and right eigenvectors.
That eigenvalue is the \emph{spectral radius} $r$, and it governs growth: iterating the matrix multiplies any positive vector by $r$ per step, up to a factor that converges.
Here the matrix is the tilted transition kernel, so its spectral radius is the per-step growth rate of the partition function.

\begin{proposition}[Transfer-operator representation for Markov families]
\label{prop:markov}
Let $S$ be finite and, for each $i = 1,\dots,W$, let $P^{(i)}$ be the law of a Markov chain on $S$ with initial distribution $\mu_i$ and transition matrix $K_i$.
Define $\mu_\alpha(x) := \prod_i \mu_i(x)^{\alpha_i}$, and $T_\alpha(x,y) := \prod_i K_i(x,y)^{\alpha_i}$.
Then
\[
  Z_t(\alpha)
  = \sum_{x_0,\dots,x_t}
     \mu_\alpha(x_0)\prod_{s=0}^{t-1} T_\alpha(x_s,x_{s+1})
  = \langle \mu_\alpha, T_\alpha^t \mathbf{1}\rangle.
\]
If $T_\alpha$ is irreducible and aperiodic (e.g., Perron--Frobenius applies) and $\mu_\alpha \not\equiv 0$ (the weighted initial supports overlap, i.e.\ $Z_0(\alpha) > 0$), then $\lim_{t \to \infty}\frac{1}{t}\log Z_t(\alpha) = \log r(T_\alpha)$,
where $r(T_\alpha)$ is the Perron--Frobenius spectral radius.
\end{proposition}

The coalescent free energy therefore grows at the logarithm of a positive-kernel spectral radius per step.
Figure~\ref{fig:coalescent-tree} draws this coincidence tree and its free-energy staircase, read from the completions of a real language model.
%
% [relocated 2026-08-09 -> P17 (merw_schrodinger), Prop. `pressure_cts':
%  the continuous-state (Krein--Rutman) transfer operator and the coalescent
%  multifractal spectrum, formerly Appendix `app:coalescent-cts' here.]

\begin{figure}[tbp]
\centering
\includegraphics[width=0.92\textwidth]{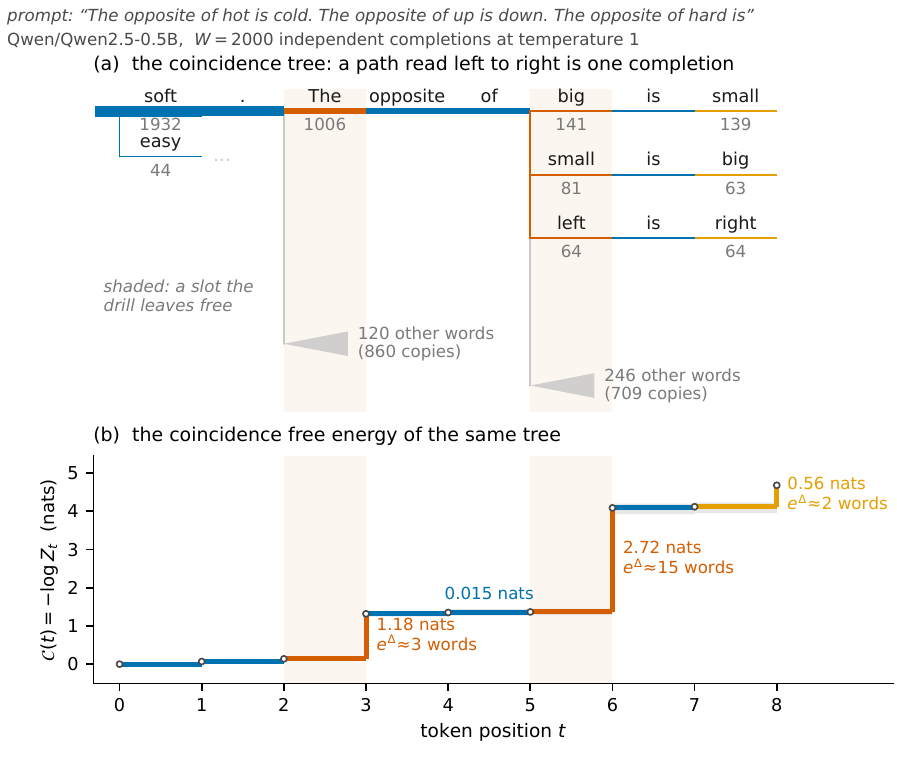}
\caption{\textbf{The free energy rises exactly where the antonym exercise leaves a choice.}
It is flat through every slot the exercise determines.  The $W=2000$ completions of a
small open-weights language model (Qwen2.5-0.5B) stay coincident through the
determined slots and branch at the free ones.  Both panels describe that one
run on a shared token axis.
\textbf{(a)} The coincidence tree: a path read left to right is one completion,
and edge width tracks the number of copies sharing that prefix.  The posed item
is answered near-unanimously---$1932$ copies take \emph{soft} and $44$ the
second admissible sense \emph{easy}---and the frame tokens ``of'' and
``is'' are taken by every surviving copy.  The two shaded slots are the ones
the exercise leaves free: the continuation after the answer, and the next word to
ask about.  A lineage that picks a word is then determined again by its own
choice, giving \emph{big}$\to$\emph{small}, \emph{small}$\to$\emph{big} and
\emph{left}$\to$\emph{right}; the remaining continuations at each contested slot
are folded into a wedge.
\textbf{(b)} The coincidence free energy $\cC(t)=-\log Z_t$ of the same tree,
the $\alpha=(1,1)$ case of Proposition~\ref{prop:coalescence}, where $Z_t$ is
the probability that two independent copies share the prefix at $t$, estimated
over all prefixes by the unbiased pair-collision estimator with a $95\%$
bootstrap interval over completions (band).  Its two risers are exactly the two
shaded slots of (a), costing $1.18$ and $2.72$ nats, while the determined slots
cost between $0.015$ and $0.069$; the exponential of a riser is the effective
number of equally likely words at that slot.  Color grades each slot by what
its token costs, from committed to contested.}
\label{fig:coalescent-tree}
\end{figure}

% ======================================================================
\section{Martingale inequalities from the coincidence calculus}
\label{sec:master}
% ======================================================================

The path-space identity yields concentration theory: a wide range of classical martingale tails, including every one treated in this section, is the master entropic identity read at a single exponential process, built from the cumulants of the increments.
Fix an adapted $Y = (Y_t)_{t \geq 0}$ with $Y_0 = 0$ and increments $d_t := Y_t - Y_{t-1}$, and fix $\lambda$ at which the conditional cumulant generating function of every increment is a.s.\ finite.
The \emph{running conditional log-CGF} is the predictable process
\begin{equation}\label{eq:exact-cgf}
  \Psi_t(\lambda) := \sum_{s=1}^t \psi_s(\lambda),
  \qquad
  \psi_s(\lambda) := \log \E\bigl[e^{\lambda d_s} \mid \cF_{s-1}\bigr],
\end{equation}
and the exponential process it normalizes,
\begin{equation}\label{eq:doleans}
  \cE_t(\lambda) := \exp\!\bigl(\lambda Y_t - \Psi_t(\lambda)\bigr),
\end{equation}
known as the \emph{Dol\'eans exponential} of $\lambda Y$, is a nonnegative martingale with $\cE_0(\lambda) = 1$ and $\E_P[\cE_t(\lambda)] = 1$ at every $t$: each one-step factor $e^{\lambda d_s - \psi_s(\lambda)}$ has conditional mean one by construction.
The log-expectation of this martingale is a Donsker--Varadhan free energy pinned at zero, and reading that normalization against every alternative path measure at once is an exact entropic equality (Figure~\ref{fig:slack-ladder}).

\begin{figure}[tbp]
\centering
\includegraphics[width=\textwidth]{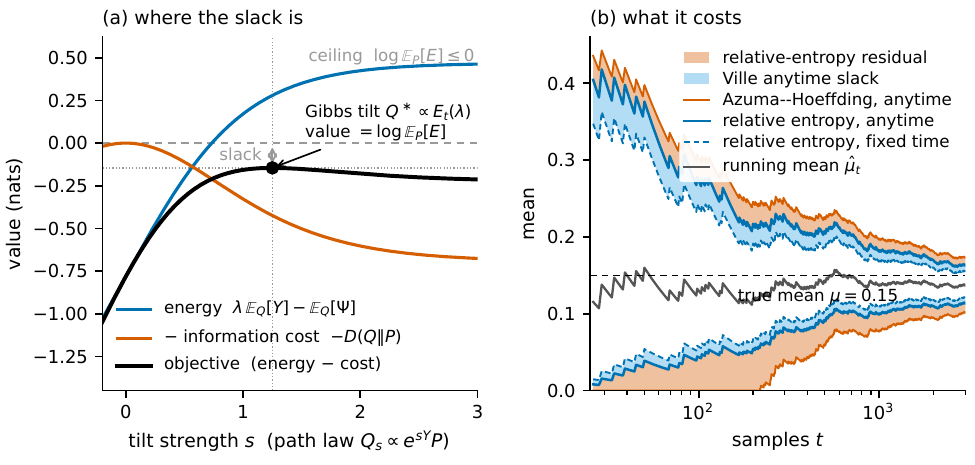}
\caption{\textbf{The master identity sits at the peak of a free-energy landscape, and every classical relaxation steps down from it by an amount that widens a confidence band.}
\textbf{(a)} The identity~\eqref{eq:master-identity} drawn as a Donsker--Varadhan free-energy landscape.
Along the exponential-tilt family $Q_s\propto e^{sY}P$ the variational objective, energy minus relative-entropy cost, is unimodal in the tilt and attains its supremum at the Gibbs tilt.
The value there equals the log-expectation of the exponential process: zero for the exact conditional log-CGF~\eqref{eq:doleans}, and below zero by exactly the cumulant gap~\eqref{eq:relax-deficit} under a majorant.
Shown for a Rademacher increment under the Hoeffding cumulant at $\lambda=5/4$, with dotted guides at the maximizing tilt and the value it attains.
\textbf{(b)} What those steps cost, on a Bernoulli stream with true mean $0.15$ read at level $\delta=0.05$.
Four running-mean envelopes at one level: the exact relative-entropy tail inverted at a single pre-specified time (dashed) and uniformly over all times (solid blue), and the range-based Azuma--Hoeffding tail inverted uniformly (orange).
The two shaded rings are the paper's two named residuals, and they are orthogonal.
The inner ring is the Ville anytime slack, the width that holding at every time costs over holding at one, $21\%$ of the band at $t=100$ and $28\%$ at $t=1000$.
The outer ring is the relative-entropy residual, what a range-based reading discards against the exact tail, $34\%$ and $31\%$ at the same two horizons.
The asymmetry of the exact envelopes is the variance-awareness the range-based reading throws away.}
\label{fig:slack-ladder}
\end{figure}

\begin{theorem}[Entropic martingale identity]
\label{thm:master}
For $\Psi_t$ the running conditional log-CGF~\eqref{eq:exact-cgf}, the exponential martingale~\eqref{eq:doleans} has $\log \E_P[\cE_t(\lambda)] = 0$ for every $t \geq 0$, and for every $Q \ll P|_{\cF_t}$,
\begin{equation}\label{eq:master-identity}
  \lambda\,\E_Q[Y_t] - \E_Q[\Psi_t(\lambda)]
    + \KL(Q\|Q^\star_t)
  = \KL(Q\|P|_{\cF_t}),
\end{equation}
where $\dd Q^\star_t/\dd P \propto \cE_t(\lambda)$ is the Gibbs tilt.
Dropping the nonnegative residual $\KL(Q\|Q^\star_t)$ gives the variational form $\sup_{Q\ll P|_{\cF_t}}\bigl\{\lambda\,\E_Q[Y_t] -
\E_Q[\Psi_t(\lambda)] - \KL(Q\|P|_{\cF_t})\bigr\} = 0$, attained at $Q^\star_t$.
\end{theorem}

Every term of~\eqref{eq:master-identity} is exact, and the statement is algebra at the fixed time $t$: no stopping rule, no crossing event, and no property of the path beyond its cumulants enters.
The classical bounds depart from it in one move---they replace the exact cumulant, which is rarely computable, by a computable predictable majorant---and the gap that replacement opens is itself exact.

\begin{proposition}[Supermartingale relaxation]
\label{prop:supermart-relax}
For an adapted $\overline\Psi_t(\lambda)$, set
\begin{equation}\label{eq:superexp}
  E_t(\lambda) := \exp\!\bigl(\lambda Y_t - \overline\Psi_t(\lambda)\bigr).
\end{equation}
\begin{enumerate}[label=(\roman*),leftmargin=2em]
\item If $\overline\Psi_t(\lambda) = \Psi_t(\lambda) + D_t(\lambda)$ with $D(\lambda)$ predictable and nondecreasing, $D_0(\lambda) = 0$, then $E_t(\lambda) = \cE_t(\lambda)\,e^{-D_t(\lambda)}$ is a nonnegative supermartingale with $E_0(\lambda) = 1$, and its mean measures the discarded cumulant gap exactly:
\begin{equation}\label{eq:relax-deficit}
  \E_P[E_t(\lambda)] = \E_{Q^\star_t}\bigl[e^{-D_t(\lambda)}\bigr] \leq 1 .
\end{equation}
\item In the setting of (i), for every $Q \ll P|_{\cF_t}$ the entropic equality~\eqref{eq:master-identity} relaxes to the entropic inequality
\begin{equation}\label{eq:master}
  \lambda\,\E_Q[Y_t] - \E_Q[\overline\Psi_t(\lambda)] \leq \KL(Q\|P|_{\cF_t}),
\end{equation}
with slack exactly $\KL(Q\|Q^\star_t) + \E_Q[D_t(\lambda)]$.
\item If~\eqref{eq:superexp} is a nonnegative supermartingale with $E_0(\lambda) \leq 1$ for an arbitrary adapted $\overline\Psi_t$---no exact-cumulant hypothesis---then $\log\E_P[E_t(\lambda)] \leq 0$, and~\eqref{eq:master} holds with slack $\KL(Q\|\overline Q^\star_t) - \log\E_P[E_t(\lambda)] \geq 0$, for $\dd \overline Q^\star_t/\dd P \propto E_t(\lambda)$.
\end{enumerate}
\end{proposition}

The relaxation moves the entropic statement from an equality to an inequality by an exactly known amount.
Against any $Q$ the slack splits into the tilt sub-optimality $\KL(Q\|Q^\star_t)$, already present in the identity, plus the $Q$-expected cumulant gap $\E_Q[D_t(\lambda)]$; the mean deficit~\eqref{eq:relax-deficit} is the same gap read under the Gibbs tilt, and the two accountings agree because $\KL(Q\|\overline Q^\star_t) = \KL(Q\|Q^\star_t) + \E_Q[D_t(\lambda)] + \log\E_P[E_t(\lambda)]$.
The construction~\eqref{eq:superexp} is standard \cite{howard2021ann,ramdas2020admissible}; typical majorants are the Hoeffding ($\overline\Psi_t(\lambda) = \frac{\lambda^2}{2}\sum c_s^2$, from the range of the increments) and Bernstein ($\overline\Psi_t(\lambda) = \frac{\lambda^2}{2(1-\lambda b/3)} \langle Y\rangle_t$, for $\langle Y\rangle$ the predictable variance) processes.
Part (iii) is the working form when the increments have no finite exact cumulant, and it is the hypothesis under which every named bound below---from Ville to PAC-Bayes---is derived, each a specialization through a different $E_t$ and event $A$.

\begin{corollary}[Event identity]
\label{cor:event}
In the setting of Proposition~\ref{prop:supermart-relax}(iii), for $A \in \cF_t$ with $P(A) > 0$,
\begin{equation}\label{eq:event}
  -\log\PP(A)
  = \lambda\,\E[Y_t \mid A] - \E[\overline\Psi_t(\lambda) \mid A]
    + \KL\bigl(P(\cdot \mid A)\,\big\|\,\overline Q^\star_t\bigr)
    - \log\E_P[E_t(\lambda)],
\end{equation}
for $\dd\overline Q^\star_t/\dd P \propto E_t(\lambda)$.
Both residuals are nonnegative, and discarding them gives the event bound
\begin{equation}\label{eq:event-bound}
  \PP(A) \leq \exp\!\bigl(-\lambda\,\E[Y_t \mid A] + \E[\overline\Psi_t(\lambda) \mid A]\bigr);
\end{equation}
if in addition $\lambda \geq 0$ and on $A$ one has $Y_t \geq x$ and $\overline\Psi_t(\lambda) \leq c$,
then $\PP(A) \leq e^{-\lambda x + c}$, at the further cost of the event relaxation $\lambda(\E[Y_t \mid A] - x) + (c - \E[\overline\Psi_t(\lambda) \mid A])$.
Since $E_t(\lambda) > 0$ a.s., the first residual is strictly positive whenever $\PP(A) < 1$: equality in~\eqref{eq:event-bound} holds only if $\PP(A) = 1$ and $\lambda Y_t = \overline\Psi_t(\lambda)$ a.s.
\end{corollary}

For any event, $-\log\PP(A) = \KL(P(\cdot\mid A)\|P)$ is its information, and~\eqref{eq:event} resolves that information into a tilted-energy part and two named residuals: the sub-optimality of the tilt at $Q = P(\cdot\mid A)$, and the cumulant relaxation paid for replacing the exact log-CGF by a majorant.
Every named tail bound below is this divergence relaxed; the informative content is the slack the tightness results pin down---the per-step decomposition of Theorem~\ref{thm:azuma} and the no-overshoot equality of Theorem~\ref{thm:ville-tight}.

The identity and its relaxation are read at a fixed time; the crossing statements below read them at a stopping time, and one theorem licenses that passage .

Stopping a nonnegative supermartingale cannot increase its mean.
A martingale preserves its mean under either of two further hypotheses, and integrability of the stopped value is not one of them, because mass can escape to the event that the time is infinite.

\subsection{Classical concentration inequalities as specializations}
\label{sec:mart-ineq}

A wide range of classical martingale concentration inequalities are the entropic inequality~\eqref{eq:master} of the supermartingale relaxation (Proposition~\ref{prop:supermart-relax})---equivalently its event identity Corollary~\ref{cor:event}---specialized to a particular cumulant majorant $\overline\Psi_t$ and then optimized over $\lambda$; only the majorant changes from one member to the next.
The same crossing event read through the Donsker--Varadhan identity (Corollary~\ref{cor:DV}) returns the exact tail as a relative-entropy equality, of which each bound is the exponential-tilt relaxation.

Table~\ref{tab:classical-ledger} does this for every member of the classical family at once: the identity each bound relaxes, what it discards to reach its bound, and the geometry in which that discard is the deficit of an optimizer.
Its lower segment lists the power-mean certificate family of Section~\ref{sec:doob-bellman} in the same form.

\begin{table}[!htp]
\centering
\caption{The standard bounds as identities with named residuals, across the three
geometries.  \emph{Upper segment}: the classical statement, the identity it
relaxes, what it discards to get there, and the geometry in which that discard
is the deficit of an optimizer; every entry points at the result that develops
it.  \emph{Lower segment} (the power-mean certificate family of
Section~\ref{sec:doob-bellman}): each row is an inequality
$\E[\Gamma]\le C\,\E[\Xi]$ proved by a certificate $U$ dominating $\Gamma-C\Xi$
and supermartingale along the path.  The certificate
form~\eqref{eq:power-certificate} names the discarded quantities---a
majorization deficit, an optional-stopping deficit, and an initial value---and
``Evaluated'' records whether this paper computes those quantities or only
names them.  The Burkholder--Davis--Gundy row is
evaluated on a second axis, the cost of moving from a stopping time to a
random time; in that row $\Phi$ is a moderate convex function---increasing,
vanishing at the origin, and with $\Phi(2x)$ bounded by a constant multiple of
$\Phi(x)$.  $M^\ast_T:=\sup_{t\le T}M_t$; $\langle M\rangle_T^{1/2}$ is the
square function.}
\label{tab:classical-ledger}
\small
\begin{tabular}{@{}p{0.215\textwidth}p{0.245\textwidth}p{0.315\textwidth}p{0.135\textwidth}@{}}
\toprule
\textbf{Classical statement} & \textbf{Identity it relaxes} &
\textbf{What it discards} & \textbf{Geometry} \\
\midrule
Ville
  & first-passage identity (Thm.~\ref{thm:firstpassage-exact})
  & the mean overshoot at the crossing; for a supermartingale also the
    predictable loss and the never-crossing mass
  & optional stopping \\
\addlinespace[2pt]
Azuma--Hoeffding
  & per-step relative-entropy sum \eqref{eq:azuma-exact}
  & the event relaxation, the tilt's sub-optimality, and the range-based
    cumulant majorant \eqref{eq:azuma-three}
  & Gibbs / entropic \\
\addlinespace[2pt]
Bennett--Bernstein--Freedman
  & the same identity at a variance-aware cumulant
  & the same three, the majorant gap now keeping a proxy for the higher
    cumulants
  & Gibbs / entropic \\
\addlinespace[2pt]
Parameter mixtures (method of mixtures)
  & pathwise mixture identity, tail exact
    (Prop.~\ref{prop:mixture-exact})
  & the overshoot at the crossing and the set on which the mixture crosses
    while the certificate never fires
  & Gibbs / entropic \\
\addlinespace[2pt]
Curved boundary of any form
  & discrete crossing law
    (Thm.~\ref{thm:curved-crossing-discrete})
  & a quadrature residual and an overshoot, both zero in continuous time and
    computable in advance
  & optional stopping \\
\addlinespace[2pt]
PAC-Bayes (martingale form)
  & pathwise Gibbs-tilt identity
    (Prop.~\ref{prop:mixture-exact})
  & the divergence $\KL(\rho\|\rho^\ast_t)$ to the running tilt, zero at the
    Bayes posterior, then the Ville relaxation
  & Gibbs / entropic \\
\addlinespace[2pt]
Pooled Ville over $W$ tests
  & exact crossing probability
    (Cor.~\ref{cor:pooling-crossing-exact}), on the crossing-mass identity
    (Prop.~\ref{prop:pooling-exact})
  & the compensator shed before the crossing, the below-threshold mass, and
    the overshoot
  & optional stopping \\
\bottomrule
\end{tabular}
\par\vspace{2pt}
\begin{tabular}{@{}p{0.19\textwidth}p{0.20\textwidth}p{0.225\textwidth}p{0.15\textwidth}p{0.12\textwidth}@{}}
\toprule
\textbf{Inequality} & \textbf{Target $\Gamma$ vs.\ comparand $\Xi$} &
\textbf{Certificate} & \textbf{Discarded} & \textbf{Evaluated} \\
\midrule
Doob $L^p$ maximal, $p>1$ (Theorem~\ref{thm:doob})
  & $(M^\ast_T)^p$ against $M_T^p$ at $C=q^p$, $q=p/(p-1)$; as an identity,
    $\|M^\ast\|_p=q\|M_T\|_p-\cR$
  & $U(x,y)=y^p-q\,y^{p-1}x$, \eqref{eq:doob-bellman}; affine in the running
    value, smooth fit on the diagonal; the Az\'ema--Yor process at
    $\Phi(y)=-y^p/(p-1)$
  & H\"older deficit $\delta_{\mathrm H}$, optional-stopping deficit
    $\delta_{\mathrm B}$, initial value $M_0^{p}/(p-1)$
  & yes, Corollary~\ref{cor:doob-residual}; $\delta_{\mathrm B}$ per step,
    Proposition~\ref{prop:ay-certificate} \\
Maximal identity at concave $\Phi$
(Proposition~\ref{prop:ay-certificate})
  & $\Phi(M^\ast_T)$ against $\Phi(M_0)$ and the terminal gap
    $(M^\ast_T-M_T)\Phi'(M^\ast_T)$
  & the Az\'ema--Yor process $A^\Phi$ of~\eqref{eq:ay-process}, one member per
    concave $\Phi$
  & one Bregman divergence $D_{-\Phi}$ per advance of the running maximum, and
    nothing else
  & yes, exactly; no H\"older step and no constant \\
Burkholder--Davis--Gundy, including the $p=1$ Davis
endpoint~\cite{burkholder1972}
  & $\Phi(M^\ast_T)$ against $\Phi(\langle M\rangle_T^{1/2})$, each direction separately
  & a Burkholder function for the given $\Phi$~\cite{burkholder1972,osekowski2012}
  & named by \eqref{eq:power-certificate} once a $\Phi$-certificate is fixed;
    at a random time, an inflation factor
  & deficits no; at a random time yes,
    Appendix~\ref{sec:moment-random-time} \\
Burkholder--Rosenthal, $p\ge2$~\cite{burkholder1973ann}
  & $|M_T|^p$ against the conditional-variance and per-step moment terms
  & distribution-function and good-$\lambda$
    inequalities~\cite{burkholder1973ann}
  & named by \eqref{eq:power-certificate} per certificate
  & no \\
Sharp martingale-transform and square-function bounds~\cite{osekowski2012}
  & the transform against the martingale in $L^p$, at the sharp constant
  & Burkholder's explicit special function; the obstacle boundary of the
    Bellman function supplies the equality
  & $\delta_{\mathrm M}=0$ at a sharp certificate, leaving the
    optional-stopping deficit and the initial value
  & no \\
\bottomrule
\end{tabular}
\end{table}

The object a bound is about fixes which of the table's three geometries it lives in.
An inequality about an exponential moment has a Gibbs tilt for its extremizer, so its residual is a relative entropy; the whole cumulant family falls here.
An inequality about a power mean has for its extremizer the point at which a certificate function meets the quantity it dominates, so its residual is the deficit of a function affine in the running value.
Doob's $L^p$ maximal inequality is that instance, its residual written $\cR$ against the running maximum $M^\ast:=\sup_{t\le T}M_t$ at the conjugate exponent $q=p/(p-1)$; the certificate rows occupy the lower segment, developed in Section~\ref{sec:doob-bellman}.
The crossing results occupy the third.
Their residuals are optional-stopping quantities---a mean overshoot, a shed compensator, a quadrature residual---governed by where the path first exceeds a level, with neither a tilt nor a contact entering.

The exponential process itself is the first specialization.

\begin{corollary}[Ville's inequality]
\label{cor:ville}
Let $(E_t)_{t \geq 0}$ be a nonnegative supermartingale with $E_0 \leq 1$.  Then for every $x > 0$, $\PP\!\left(\sup_{t\geq 0} E_t \geq e^x\right) \leq e^{-x}$.
Consequently, if~\eqref{eq:superexp} is a nonnegative supermartingale with $E_0(\lambda) \leq 1$ (Proposition~\ref{prop:supermart-relax}(iii)),
$\PP(\sup_t \{\lambda Y_t - \overline\Psi_t(\lambda)\} \geq x) \leq e^{-x}$.
\end{corollary}

The Donsker--Varadhan reading of the same crossing event sharpens Ville's inequality to the exact tail probability, a relative-entropy equality.

\begin{theorem}[Ville's inequality as an exact-tail identity]
\label{thm:ville-DV}
Let $(M_t)_{t \geq 0}$ be a nonnegative supermartingale with $M_0 = 1$.
For every $\lambda > 0$ with $\PP(\sup_t M_t \geq \lambda) > 0$, the first-passage tail is a relative-entropy equality,
\begin{equation}\label{eq:ville-exact}
  -\log\PP\!\left(\sup_t M_t \geq \lambda\right)
  = \KL\!\left(P(\cdot \mid \sup_t M_t \geq \lambda) \,\|\, P\right),
\end{equation}
and optional stopping bounds this residual below by $\log\lambda$, which is Ville's inequality $\PP\!\left(\sup_{t \geq 0} M_t \geq \lambda\right) \leq \frac{1}{\lambda}$,
itself valid at every $\lambda > 0$.
\end{theorem}

The gap between the residual $\KL(P(\cdot|A)\|P)$ and $\log\lambda$ measures how much tighter the true tail is than Ville's bound;
Theorem~\ref{thm:firstpassage-exact} closes it exactly for a vanishing martingale, in terms of the mean overshoot; for a general supermartingale the gap also includes the predictable loss and the never-crossing mass.

\subsection{Ville tightness at the ultimate maximum}

The random-time divergences later in the paper all trace to one extreme object: the value a vanishing martingale takes at its own global maximum.

For a nonnegative martingale $(M_t)$ with $M_0 = 1$ and $\lim_{t\to\infty} M_t = 0$ a.s., the \emph{ultimate-maximum time} is
\begin{equation}\label{eq:tau-star}
  \tau^\star := \min\!\bigl\{t \geq 1 : M_t = \sup_{s \geq 1} M_s\bigr\}
\end{equation}
Under $M_t \to 0$, the supremum is a.s.\ attained, so $\tau^\star$ is a.s.\
finite and positive-integer-valued, as a random time requires; it is \emph{not} a stopping time, because identifying that $t$ achieves the global maximum requires seeing all future values.
Since $\E[\sup_{t\ge0}M_t]\le 1+\E[\sup_{t\ge1}M_t]$, restricting the index to $t\ge1$ leaves the non-integrability of Lemma~\ref{lem:sup-nonintegrable} unaffected.

Ville's classical inequality~\cite{howard2021ann,ramdas2020admissible,ruf2022composite} bounds the probability that a nonnegative unit-mean martingale ever exceeds $x$ by $1/x$.
When the martingale converges to $0$ and crosses each level without overshoot---the continuous-path idealization---that bound is attained.
The equality combines the upper bound already carried by Ville's inequality with a matching lower bound from Doob's maximal identity, under which the running maximum is distributed as $1/U$ for $U$ uniform on $[0,1]$.
The resulting heavy Pareto-type $1/x$ tail of $M_{\tau^\star}$ is the source of the divergences in the random-time peeking penalties below; the estimation of such running extrema is itself possible~\cite{balsubramani2020p}.

\begin{theorem}[Doob's maximal identity / Ville tightness]
\label{thm:ville-tight}
Let $(M_t)$ be a nonnegative martingale with $M_0=1$ and $M_t\to0$ a.s., let $\tau^\star$ be its ultimate-maximum time of~\eqref{eq:tau-star}, and fix a level $x > 1$.  Then
\[
  \PP(M_{\tau^\star} \geq x)
  = \PP\!\left(\sup_{t\geq 0} M_t \geq x\right)
  \leq \frac{1}{x},
\]
with equality at that level exactly when $x$ is crossed without overshoot: writing $\sigma_x := \inf\{t : M_t \geq x\}$ for the first passage to $x$, $M_{\sigma_x} = x$ a.s.\ on $\{\sigma_x < \infty\}$; the deficit is the mean overshoot of Theorem~\ref{thm:firstpassage-exact}, divided by the level.
A continuous-time vanishing martingale with continuous paths crosses every level without overshoot, and its running maximum is then exactly Pareto$(1)$~\cite{nikeghbali2006doobs}.
No discrete-time vanishing martingale attains equality at every level at once: absence of overshoot at every $x \geq 1$ forces $\sup_{t} M_t \leq 1$ a.s., and a nonnegative martingale with $M_0 = 1$ bounded by $1$ is constant.
The Pareto$(1)$ law is the continuous-time idealization that the discrete-time tail approaches from below, level by level, as the overshoot vanishes.
\end{theorem}

Theorem~\ref{thm:ville-tight} reads the ultimate maximum from the start of the path.
The same accounting holds from every stopping time onward, with the current value in place of $M_0$.
Equality is attained in the continuous-path regime, and the conditional form of the maximal identity lives there as well,
so Proposition~\ref{prop:maximal-conditional} --- that form, read at a stopping time and against a level already attained --- is stated in continuous time; the discrete-time convention of the paper applies elsewhere.

At $\tau=0$ and $m=1$, part~(ii) reduces to part~(i), which is the Pareto$(1)$ law of Theorem~\ref{thm:ville-tight}.
On the logarithmic scale that law is a standard exponential, so $\E[\log M^\ast_\infty]=1$: the ultimate maximum stands one nat above the starting value in expectation.
Part~(ii) says the same nat is available from every stopping time onward, at the same value and whatever the path has done so far.
The anticipation an observer could still collect by holding out for the global maximum does not deplete along the trajectory; it renews at every moment the observer is entitled to recognize, and the elapsed history contains no information about how much of it remains.
The comparison with Lemma~\ref{lem:sup-nonintegrable} and Proposition~\ref{prop:infinite} is a matter of scale: the same object has $\E[M^\ast_\infty]=\infty$ and $\E[\log M^\ast_\infty]=1$, so the peeking penalty diverges while the anticipation budget it measures stays at one nat.

Each hypothesis supplies part of the statement.
Positivity and $M_\infty=0$ place the whole mass of $M^\ast_\infty$ on $[m,\infty)$ with the Pareto scaling; continuity of $M^\ast$ is the no-overshoot condition, under which the level $a$ is met at $M_{\sigma_a}=a$.
The absence of positive jumps in (ii) makes the same no-overshoot property hold from every stopping time onward, at levels below the running record as well.
The discrete-time limitation of Theorem~\ref{thm:ville-tight} binds here too, and Proposition~\ref{prop:maximal-conditional} is the continuous-path idealization of that limit, at every stopping time simultaneously.

The deficit in Ville's bound has a closed form: for a vanishing martingale the first-passage probability falls below the bound by exactly the mean overshoot divided by the level.

\begin{theorem}[Discrete-time first-passage identity]
\label{thm:firstpassage-exact}
Let $(M_t)_{t\ge0}$ be a nonnegative martingale with $M_0=1$ and $M_t\to0$ a.s.  For $x\ge1$ let $\sigma_x:=\inf\{t: M_t\ge x\}$ and, on $\{\sigma_x<\infty\}$, the overshoot $J_x:=M_{\sigma_x}-x\ge0$.
Then the optional-stopping mass is conserved,
\begin{equation}\label{eq:fp-mass}
  \E\!\left[M_{\sigma_x}\,\ind\{\sigma_x<\infty\}\right]=1,
\end{equation}
and the first-passage probability satisfies the identity
\begin{equation}\label{eq:fp-exact}
  \PP\!\Bigl(\sup_{t\ge0}M_t\ge x\Bigr)
  =\PP(\sigma_x<\infty)
  =\frac1x\Bigl(1-\E\!\left[J_x\,\ind\{\sigma_x<\infty\}\right]\Bigr).
\end{equation}
The named slack is the deficit $\E[J_x\ind\{\sigma_x<\infty\}]/x$.
Equality with the Pareto$(1)$ law of Theorem~\ref{thm:ville-tight} holds iff $J_x=0$ a.s.\ (no overshoot).
The overshoot law depends on the increment distribution, so no universal closed form for $\E[J_x]$ exists beyond the overshoot-free extreme; the identity~\eqref{eq:fp-exact} accounts for the slack in every case.
\end{theorem}

The mass the identity conserves is exactly the quantity a sample cannot reach.
The same partition function that produced the crossing statements, read at the wealth itself, says how large a sample would have to be.

\begin{proposition}[R\'enyi spectrum of a test martingale]
\label{prop:renyi-wealth}
Let $(M_t)$ be a nonnegative martingale with $M_0=1$.
Taking the single factor $\Pi=M$ in Theorem~\ref{thm:pathspace} gives the wealth cumulant $\Phi_T(s)=\log\E_P[M_T^{\,s}]$, convex in $s$, with
\[
  \Phi_T(1)=0,\qquad
  \Phi_T'(0)=\E_P[\log M_T]=-\KL(P\|Q_M),\qquad
  \Phi_T'(1)=\E_P[M_T\log M_T]=\KL(Q_M\|P),
\]
for $\dd Q_M/\dd P := M_T$.
The separation between the typical value and the mass-holding one is $\Phi_T'(1)-\Phi_T'(0)=\KL(Q_M\|P)+\KL(P\|Q_M)$.
For a likelihood-ratio stream $M_T=\prod_{t\le T}R_t$, at any $s$ where the conditional one-step cumulants $\log\E_P[R_t^{\,s}\mid\cF_{t-1}]$ are all deterministic --- independent steps being the canonical case --- the spectrum is their sum, $\Phi_T(s)=\sum_{t\le T}\log\E_P[R_t^{\,s}]$, evaluated term by term without simulating a single path.
\end{proposition}

The two halves of the spectrum answer different questions about the same process.
Its slope at $0$ is the almost-sure decay that drives $M_t$ to zero, and its slope at $1$ is the improbability of the paths that keep the mean at one; the Jeffreys separation between them measures how far a sample mean falls short of the mass.

Dependence between the steps breaks the term-by-term sum away from $s=1$: the conditional cumulants become $\cF_{t-1}$-measurable random variables, their realized sum varies from path to path, and the deterministic sum $\sum_{t\le T}\log\E_P[R_t^{\,s}]$ need no longer equal $\Phi_T(s)$.
The slope at $1$ still decomposes: the chain rule (Lemma~\ref{lem:chain-rule}) resolves $\Phi_T'(1)=\KL(Q_M\|P)$ into the sum $\sum_{t\le T}\E_{Q_M}[\KL(Q_t(\cdot\mid X_{1:t-1})\,\|\,P_t(\cdot\mid X_{1:t-1}))]$ of per-step conditional divergences averaged under $Q_M$.

For an i.i.d.\ stream --- one-step ratios $R_t=q(X_t)/p(X_t)$ with the $X_t$ drawn from $p$ under $P$ and $\KL(q\|p)$ finite --- the spectrum grows linearly, $\Phi_T(s)=T\log\E_P[R_1^{\,s}]$.
The mass then acquires an asymptotic shape: $T^{-1}\log M_T$ concentrates at $\KL(q\|p)$ under $Q_M$ by the law of large numbers.
The change of measure $\dd P=M_T^{-1}\,\dd Q_M$ on $\{M_T>0\}$ therefore places the concentration set at $P$-probability $e^{-\Phi_T'(1)+o(T)}$.
Estimating $\E_P[M_T]=1$ by averaging $n$ independent draws under $P$ therefore requires $\log n\gtrsim\Phi_T'(1)$, and below that the sample mean is exponentially small.
The mechanism is the truncated-mean identity $\E_P[M_T\ind\{\log M_T\le u\}]=Q_M(\log M_T\le u)$, whose right-hand side collapses once the ceiling $u\approx\log n$ that a size-$n$ sample resolves falls below the concentration point.
The two routes to the mass diverge exactly here: the identity $\Phi_T(1)=0$ returns $\E_P[M_T]=1$ for every test martingale, while a sample mean of $M_T$ reaches that same $1$ only once $n$ passes the size $\Phi_T'(1)$ sets.

\paragraph{Other sharpenings.}
In concurrent and independent work, de la Pe\~na and Klass~\cite{delapena2026exact} establish an exact identity for the first-passage probability of a nonnegative supermartingale $(M_n)$ with $M_0=1$.
Writing $T_b$ for the first time $M$ reaches a level $b>1$,
they show $\PP(T_b<\infty)=(1-D_b-R_b)/(b+J_b)$, where $J_b$ is the expected overshoot at crossing, $D_b$ the cumulative predictable loss accrued before crossing, and $R_b$ the residual mass on paths that never cross.
Their decomposition splits the slack in Ville's inequality into three mechanisms that collapse, for a vanishing martingale, to the single overshoot term of Theorem~\ref{thm:firstpassage-exact}.
The overshoot $J_b$ is that identity's deficit conditioned on crossing, the two differing by the crossing probability, while the predictable-loss and survival terms account for the stopping defect and the never-crossing paths that the martingale case suppresses.
Both first-passage identities are conservation statements read off optional stopping.
A separate line sharpens those same tails by a complementary route: the randomized and exchangeable refinements of Markov's, Chebyshev's, and Chernoff's inequalities~\cite{ramdas2026randomized} add external randomization to tighten the bound itself, whereas the present account leaves the bound in place and measures its information cost.
The two are orthogonal---randomization sharpens the inequality, while the relative-entropy residual of this paper explains it.

The identity names a residual; which parameter governs it is a separate question, settled by measurement.
Ville tightness is a statement about the first-passage ratio $x\,\PP(\sup_{t} M_t \geq x)$, which sits at the continuous-path value $1$ minus the mean overshoot of Theorem~\ref{thm:firstpassage-exact}; the corresponding residual of Theorem~\ref{thm:ville-DV} is its logarithm, $-\log\bigl(x\,\PP(\sup_t M_t\ge x)\bigr)$.
The deficit there is set by the increment scale and not by the horizon, so it is a property of the increment law at the crossing and does not wash out as the horizon grows (Section~\ref{subsec:eval-overshoot-convergence}).

\subsection{Cumulant majorants and intrinsic time}
\label{sec:cumulant-majorants}

How much of the tail entropy a bound keeps depends on the bound as much as on the law.
The range-matched and the variance-aware bound are both Chernoff bounds that replace the increment's cumulant generating function by a quadratic upper bound (majorant) $\lambda^2v/2$, differing only in the proxy $v$ they use, and both are surrogates for one tilt-dependent quantity.
Matching on variance is legitimate only for laws that are $\sigma^2$-sub-Gaussian; for a skewed or heavy-tailed law it is not, and there the variance-aware bound of Proposition~\ref{prop:freedman} supplies the correction.
Neither closes the gap.  The residual persists under both and remains largest where the increments are heaviest, and the identity of Section~\ref{sec:master} names it in either case.

A bounded-increment cumulant bound turns the same master inequality into Azuma--Hoeffding, and states the exact tail alongside it.
That bound is the conditional Hoeffding lemma: for a sub-$\sigma$-field $\cG$ and a random variable $X$ with $a \leq X \leq b$ a.s.\ and $\E[X\mid\cG]=0$,
$\E[e^{\lambda X}\mid\cG] \leq \exp\bigl(\lambda^2(b-a)^2/8\bigr)$.

The bound that results is the Azuma--Hoeffding tail (Theorem~\ref{thm:azuma}).

The Azuma bound sits above the exact tail~\eqref{eq:azuma-exact} by three relaxations.
Writing $Q := P(\cdot \mid \sum_s d_s \geq x)$ and $\overline Q^\star$ for the law tilted by the Hoeffding exponential~\eqref{eq:superexp} at the optimal $\lambda^\ast$,
\begin{equation}\label{eq:azuma-three}
  -\log P\Bigl(\textstyle\sum_s d_s \geq x\Bigr)
  - \frac{2x^2}{\sum_s (b_s-a_s)^2}
  = \underbrace{\lambda^\ast\bigl(\E[\textstyle\sum_s d_s \mid Q] - x\bigr)}
      _{\text{event relaxation}}
  + \underbrace{\KL(Q \| \overline Q^\star)}_{\text{tilt sub-optimality}}
  + \underbrace{\bigl(-\log \E[E_t(\lambda^\ast)]\bigr)}_{\text{cumulant relaxation}} .
\end{equation}
This display is the event identity~\eqref{eq:event} read at the optimizing tilt $\lambda^\ast$, with its event relaxation split off: the tilt sub-optimality and the cumulant relaxation are the corollary's two residuals, and the first term is the cost of passing from $\E[\sum_s d_s \mid Q]$ to the threshold $x$.
Only the third is the Hoeffding cumulant, and how the three compare depends on how tightly the range bounds the increment.
For $\pm c_s$ Rademacher increments the cumulant relaxation is tight --- at $t=8$, $x=3$ it contributes $0.013$ of a $1.372$ gap, against $0.578$ for the event relaxation and $0.781$ for the tilt --- so the combined gap is most of the exact tail.
When the range badly overstates the variance the ordering changes: for increments uniform on $[-1,1]$, at $t=12$, $x=4$, the three terms are $0.233$, $2.459$ and $0.445$, so the cumulant is no longer the smallest.

Replacing the bounded-increment cumulant with a variance-aware one gives the Bennett--Bernstein--Freedman family.
Let $\langle M\rangle_t := \sum_{s=1}^t \E[d_s^2 \mid \cF_{s-1}]$ denote the predictable variance process.

The one-step Bernstein bound (Lemma~\ref{lem:bernstein}) is the majorant behind the variance-sensitive members.

Substituting it gives the Freedman--Bernstein line-crossing bound (Proposition~\ref{prop:freedman}).
The proxy $v$ is a single function of the tilt, so these two members approximate the same object.
Write the \emph{intrinsic time} of an increment $d$ at tilt $\lambda$ as
\begin{equation}\label{eq:intrinsic-time}
  v(\lambda) \;:=\; \frac{2\Psi(\lambda)}{\lambda^{2}},
  \qquad \Psi(\lambda) = \log\E\bigl[e^{\lambda d}\bigr],
\end{equation}
the variance proxy a Chernoff bound must use to majorize the cumulant by the quadratic $\lambda^{2}v/2$.
Since $\Psi(\lambda) = \lambda^{2}\sigma^{2}/2 +
\kappa_{3}\lambda^{3}/6 + O(\lambda^{4})$ for a mean-zero law whose cumulant
generating function is finite near the origin, $v(\lambda) \to \sigma^{2}$ as $\lambda \to 0$: \emph{the variance is the small-tilt limit of intrinsic time}.
The name records that $v(\lambda)$ is the variance of the Gaussian increment whose cumulant matches $\Psi$ at that tilt, and on the Brownian clock variance and elapsed time coincide.
The bounded increments of this section meet that condition, and the two classical proxies are then surrogates for one quantity.
Azuma--Hoeffding uses $R^{2}/4$, flat in $\lambda$; Bennett--Bernstein--Freedman uses $\sigma^{2}/(1-\lambda b/3)$, which meets the intrinsic time at $\lambda=0$ and grows thereafter.
Neither proxy dominates the other: which is tighter depends on the deviation being tested at, and the crossing has a closed form in that deviation.
Over $T$ steps a Chernoff bound on the quadratic majorant captures $x^{2}/(2Tv)$ nats at deviation $x$, so the variance-aware proxy is the tighter of the two exactly while
\[
  x \;<\; x^{\dagger} \;:=\; \frac{3T\bigl(R^{2}/4-\sigma^{2}\bigr)}{b}.
\]
Popoviciu's inequality makes the numerator nonnegative, so $x^{\dagger}$ is well defined for every bounded law and vanishes exactly at $R=2\sigma$; a variance small enough against the range makes the variance-aware bound the sharper one at every reachable deviation~\cite{semenova2023simpler}.
The simplest mean-zero martingale that is not symmetric---two atoms, $+2$ with probability $1/3$---has $\sigma^{2}=2$, $R=3$ and $b=2$, so $x^{\dagger}=0.375\,T$ and the two bounds cross at a deviation a run of length $T$ can reach.

\subsection{Parameter mixtures and curved boundaries}

A single exponential supermartingale $E_t(\lambda)$ certifies a bound only at the parameter $\lambda$ fixed in advance, yet the analyst rarely knows the right $\lambda$ before seeing the data.
Averaging the whole family over a prior $\mu$ on $\Lambda$ yields one supermartingale valid simultaneously for every $\lambda$, at the cost of a relative-entropy penalty for whichever posterior the data end up favoring.
This is the method of mixtures, and it is how a fixed-parameter tail bound becomes a single time-uniform curved boundary.

That construction is the pathwise mixture bound (Proposition~\ref{prop:mixture}).

Markov's inequality reaches~\eqref{eq:mixturePAC} in one step and leaves no record of what that step cost. 
In fact, two quantities are discarded in passing from~\eqref{eq:mixturepathwise} to~\eqref{eq:mixturePAC}, and Proposition~\ref{prop:mixture-exact} restores both. 
On path space the same passage is an equality once its two residuals are named: the mass on which the mixture crosses the level while the posterior's divergence from the running tilt exceeds the margin,
and the deficit at the crossing itself.

Its exact form (Proposition~\ref{prop:mixture-exact}) evaluates both of the discarded quantities.

On $\cV_x\setminus\cA_x$ the mixture crosses the level but the certificate never fires; at the Bayes posterior, that set is empty.

The posterior $\rho^*_t$ moves as data arrive, and the relative entropy between consecutive posteriors is the information the path spends on that step.
Accumulating it puts information on an axis against which each first crossing of an evidence level can be located, so what the next factor of evidence requires is the information accumulated between consecutive crossings (Figure~\ref{fig:evidence-spend}).

\begin{figure}[t]
  \centering
  \includegraphics[width=0.98\linewidth]{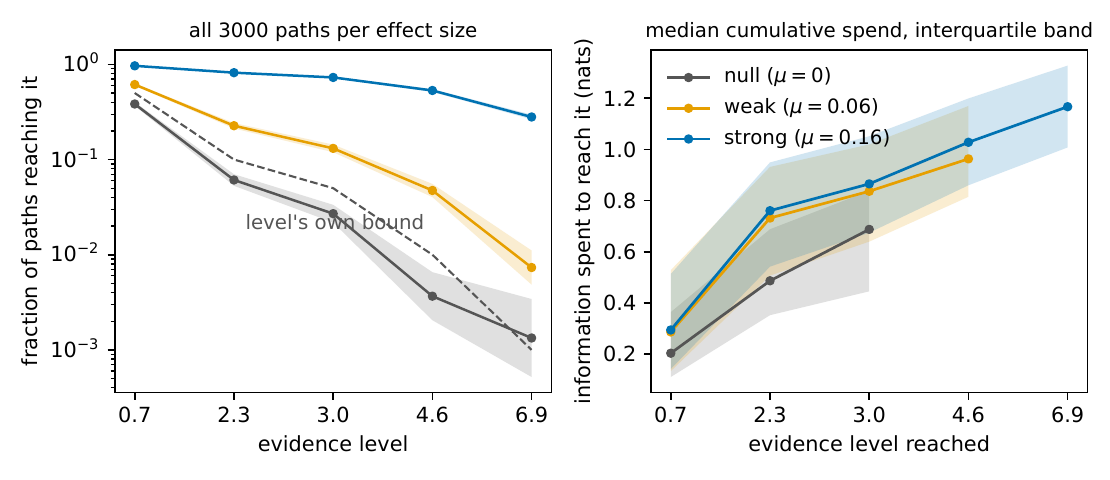}
  \caption{\textbf{What evidence costs, in nats, along realized paths.}
  A Gaussian stream read through the mixture of Proposition~\ref{prop:mixture} over a grid of tilts; the spend on a step is $\KL(\rho^*_t\|\rho^*_{t+1})$ between consecutive optimal posteriors, and a milestone is the first crossing of $\log(1/\alpha)$ by $\log\overline E_t$.
  Both panels are read against the same evidence level $\log(1/\alpha)$, so the pair gives what a level costs beside how often it is reached.
  \textbf{Left:} the fraction of all $3000$ paths per effect size reaching each level, with $95\%$ Wilson intervals, against the level's own bound.
  The null arm tracks below that bound throughout: crossing rates $0.383$, $0.061$, $0.027$, $0.0037$ and $0.0013$ against $0.5$, $0.1$, $0.05$, $0.01$ and $0.001$, the last being four crossings where three are expected.
  \textbf{Right:} the cumulative spend to first reach each level, median and interquartile band, so the rise between two ticks is the information that next factor of evidence requires.
  Each level is drawn only where at least $30$ paths reached it, and the statistics are conditional on arrival, which is why the null and weak curves stop early: reaching a deep level under a weak signal selects the paths that moved fastest, and too few remain to summarize.
  The effect size decides how often a level is reached and barely touches what reaching it costs.
  At the deepest level the three arms differ $210$-fold in the fraction arriving, while their median total spend differs by a factor of $1.66$: $0.77$ nats under the null, $0.96$ at the weak effect and $1.28$ at the strong one.}
  \label{fig:evidence-spend}
\end{figure}

\subsection{Exact crossing of a curved boundary}
\label{sec:curved-crossing}

A sequential test that spends its error level along a curved boundary is calibrated by one integral of that boundary in the continuous-path idealization, and the calibration fails once the test is run on a discrete clock.
On a symmetric $\pm1$ walk read through its running maximum, a boundary set to spend $\alpha=1-e^{-1/2}\approx0.394$ spends about $0.442$ instead---an eighth over its nominal level.
Theorem~\ref{thm:curved-crossing-discrete} locates that overrun in two computable terms.

The instance to hold in view is the drawdown of a martingale below its own running maximum: the boundary caps how far the path may fall from its record, and the record itself is the clock the boundary is read against.
Drawdown limits in sequential monitoring have exactly this form.
Posing the question needs no continuum---a boundary $\varphi$, a clock $A$, and the event $\{\exists\,t:\ X_t>\varphi(A_t)\}$ make sense on any filtration---and the closed form below comes from one function-indexed family of martingales sharing a single increasing process, whose continuous-path member is the classical law (Appendix~\ref{app:curved-cts}).
Call $X=N+A$ a \emph{discrete class-$(\Sigma)$ process} when $X\ge0$ is adapted with $X_0=0$, $N$ is a martingale with $N_0=0$, $A$ is adapted and nondecreasing with $A_0=0$, and $\Delta A_t>0$ only at times with $X_t=0$.
The drawdown just described is one such process, with $A$ the running maximum itself.

\begin{theorem}[Discrete-time crossing of a curved boundary]
\label{thm:curved-crossing-discrete}
Let $X=N+A$ be a discrete class-$(\Sigma)$ process and let $\varphi\colon\R_+\to(0,\infty]$ be Borel with $I:=\int_0^\infty\varphi(x)^{-1}\dd x<\infty$.
Write $w(x):=\exp\bigl(-\int_x^\infty\varphi(z)^{-1}\dd z\bigr)$, $\zeta:=1-w$ and $h:=w/\varphi$, so that $w'=h$ and $\zeta'=-h$, and assume $h$ is bounded.
Then $M_t:=\zeta(A_t)+h(A_t)X_t$ is nonnegative with $M_0=1-e^{-I}$, the crossing event is exactly
\begin{equation}\label{eq:discrete-crossing-event}
  \bigl\{\exists\,t:\ X_t>\varphi(A_t)\bigr\}
  =\bigl\{\sup\nolimits_t M_t>1\bigr\},
\end{equation}
and $M$ decomposes pathwise as
\begin{equation}\label{eq:discrete-crossing-decomp}
  M_t=M_0+\sum_{s\le t}h(A_{s-1})\,\Delta N_s+\sum_{s\le t}R_s,
  \qquad
  R_s:=h(A_{s-1})\,\Delta A_s-\bigl(w(A_s)-w(A_{s-1})\bigr),
\end{equation}
the first sum a martingale and $R_s=\int_{A_{s-1}}^{A_s}\bigl(h(A_{s-1})-h(z)\bigr)\dd z$ the quadrature residual of $h$ across the jump of $A$.
If in addition $h$ has finite total variation $\mathrm{var}(h)$ on $\R_+$, the jumps of $A$ are uniformly bounded, $A_t\to\infty$, and $M_t\to0$ off the crossing event, then with $T:=\inf\{t:X_t>\varphi(A_t)\}$ and overshoot $J:=M_T-1\ge0$,
\begin{equation}\label{eq:discrete-crossing-law}
  \PP\bigl(\exists\,t:\ X_t>\varphi(A_t)\bigr)
  =1-e^{-I}
   +\E\Bigl[\sum_{s\le T}R_s\Bigr]
   -\E\bigl[J\,\ind\{T<\infty\}\bigr].
\end{equation}
\end{theorem}

The event identity~\eqref{eq:discrete-crossing-event} is the geometric content, and it is pointwise: the transform is a change of the vertical coordinate that maps the boundary, and with it every level set of $M$, onto a horizontal line (Figure~\ref{fig:curved-crossing}).

Each correction is controlled by one feature of the path, and each vanishes in the continuum.
The quadrature residual obeys $\sum_s|R_s|\le\|\Delta A\|_\infty\,\mathrm{var}(h)$, with $\mathrm{var}(h)$ the total variation of $h$ on $\R_+$, equal to $h(0)=e^{-I}/\varphi(0)$ when $h$ is monotone; it shrinks as the jumps of $A$ refine and is identically zero when $A$ is continuous, since $R_s=0$ whenever $\Delta A_s=0$.
The overshoot vanishes when $X$ has no positive jumps, the hypothesis of Theorem~\ref{thm:curved-crossing} in continuous time and generally false in discrete time.
Setting both to zero returns~\eqref{eq:curved-crossing} exactly: a continuous clock spends its whole budget $\int_0^\infty\varphi^{-1}$ without skipping, and a continuous path meets the boundary without passing it.

The practical consequence is a level, and a schedule for committing it.
For a target $\alpha\in(0,1)$, any boundary with $\int_0^\infty\varphi(x)^{-1}\dd x=-\log(1-\alpha)$ is crossed with probability $\alpha$ in the continuous-path limit, and $u\mapsto1-\exp\bigl(-\int_0^u\varphi(x)^{-1}\dd x\bigr)$ reports how much of $\alpha$ has been committed once the clock reaches $u$; the form of $\varphi$ redistributes the level along the clock while leaving the total fixed.
On a discrete clock the realized probability is instead the nominal $\alpha$ plus the quadrature residual, minus the overshoot, and on a coarse clock both terms are non-trivial.
At $\varphi(x)=2(1+x)^2$ on the $\pm1$ walk the residual contributes $0.134$ and the overshoot $0.086$, placing the realized level at $0.442$.
A direct count of crossings climbs toward that value from below as the horizon grows---$0.355$ at $2\times10^{2}$ steps, $0.404$ at $5\times10^{3}$, $0.413$ at $1.25\times10^{5}$---since the identity asks for $A_t\to\infty$.
The shortfall at any finite horizon is exactly the mass of $M$ left on the paths that have not yet crossed.
A boundary calibrated by the continuum integral alone is therefore anticonservative in practice.
Both corrections contain the factor $h=w/\varphi$, so the discrepancy shrinks as the boundary is set further from the record in units of the clock's own step, and a boundary only a few steps out is where calibrating by the integral alone costs most.
Both corrections can be computed before any data arrive---the residual from $\varphi$ and the clock's jump sizes, the overshoot from the increment law at the boundary---so the budget can be set against the level the discrete rule actually attains.

The law is checked in Section~\ref{subsec:eval-curved-crossing} against a closed form built from gambler's ruin, which never mentions the martingale that proves it, and the horizon a discrete clock needs to approach it is measured there.

The continuous-time construction behind the limit---the class-$(\Sigma)$ family of function-indexed local martingales, the closed-form crossing law for arbitrary Borel boundaries, and the Brownian drawdown instance in which the law is Knight's---is collected in Appendix~\ref{app:curved-cts}.

\subsection{PAC-Bayes martingale inequalities}
\label{sec:pacbayes-mart}

Reading the mixture parameter as a model index turns the same inequality into a PAC-Bayes bound: it holds simultaneously for every data-dependent posterior over a space of models, with the relative entropy from prior to posterior as the complexity term.
It inherits the exact form of Proposition~\ref{prop:mixture-exact}: the slack against $e^{-x}$ resolves into the same two residuals, with the model mixture in place of the parameter mixture.

\begin{theorem}[PAC-Bayes martingale inequality]
\label{thm:pacbayes}
The exponential case $E_t^\theta = \exp(\lambda Y_t^\theta -
\overline\Psi_t^\theta(\lambda))$ of Proposition~\ref{prop:mixture}, with prior
$\pi$ on $\Theta$ and any adapted posterior $\rho_t \ll \pi$, is the PAC-Bayes martingale inequality
\begin{equation}\label{eq:pacbayes-bound}
  \PP\!\left(\exists\,t :
    \lambda\,\E_{\rho_t}[Y_t^\theta] - \E_{\rho_t}[\overline\Psi_t^\theta(\lambda)]
    - \KL(\rho_t\|\pi) \geq x\right) \leq e^{-x}.
\end{equation}
\end{theorem}

Theorem~\ref{thm:pacbayes} is the PAC-Bayes martingale mechanism of \cite{seldin2011pacbayesian}, deployed there as a proof device and read here as a literal mixed coincidence identity on parameter space.
The identity holds already off the path.
Writing $\pi_\lambda\propto\pi\,e^{\lambda\ell}$ for the Gibbs tilt a loss $\ell$ induces on the prior, $\lambda\,\E_\rho[\ell]=\KL(\rho\|\pi)-\KL(\rho\|\pi_\lambda)+\log\E_\pi[e^{\lambda\ell}]$ at every posterior $\rho\ll\pi$, so the usual bound discards exactly $\KL(\rho\|\pi_\lambda)$, the divergence from the posterior to that tilt \cite{balsubramani2026information}.
The martingale form inherits that deficit.
The prior's anchor then decides whether the certificate says anything at all, while the mechanism behind it does not change.
At deep-net scale, holding the data, the classifier and the martingale fixed and varying only the anchor, an uninformed prior returns a vacuous anytime-valid certificate where a prior at the pretrained weights returns a tight one (Section~\ref{sec:eval-pacbayes-cifar-fixed-prior}).

A whole family of priors sheds the multi-way coincidence divergence of Proposition~\ref{prop:multi-PAC}, a geometric mixture of $\pi_1,\dots,\pi_W$ costing less than the weighted average of the individual penalties by exactly $\cC_\alpha(\pi_{1:W})$.
That identity transfers to path space unchanged, with the priors read as whole models over trajectories, and the sequential setting resolves its coincidence divergence one step at a time (Proposition~\ref{prop:seq-PAC}).
On the back of the path-time identity of Section~\ref{sec:pathtime} it holds at random evaluation times as well as at a fixed horizon.

\section{Exact residuals in the power-mean geometry}
\label{sec:doob-bellman}

Doob's $L^p$ maximal inequality also has an exact residual, carried by a second geometry: its slack is a H\"older deficit, an optional-stopping deficit and an initial value, and the same accounting reads off the slack of any bound proved by exhibiting a dominating certificate.
The certificate that supplies it is one member of the Az\'ema--Yor family, which replaces the exponent by an arbitrary concave function and resolves the optional-stopping deficit step by step.
That second geometry is needed because the entropy method does not sharpen this inequality, which is stated in its classical form as Theorem~\ref{thm:doob}.  (The matching lower bound behind Theorem~\ref{thm:ville-tight} is Doob's maximal \emph{identity}, a separate classical result.)

The coincidence calculus sharpens a classical bound into a relative-entropy equality when the bound's extremizer is a Gibbs tilt, so that its slack is a relative entropy (Theorems~\ref{thm:ville-DV},
\ref{thm:master}). Doob's $L^p$ inequality resists this for a structural reason.
Its natural entropy extremal is the no-overshoot vanishing martingale of Theorem~\ref{thm:ville-tight}, whose running maximum is Pareto$(1)$.
For it the layer-cake integral is $\E[(\sup_t M_t)^p]=1+\int_1^\infty p\,x^{p-2}\,\dd x=+\infty$, while the almost-sure limit has $\E[M_\infty^p]=0$.
The ratio degenerates to $\infty/0$ in the limit; it does not approach the finite constant $\bigl(\tfrac{p}{p-1}\bigr)^p$.
At each finite $T$ both sides are finite and Doob's inequality holds as stated; $\E[M_T^p]$ does not vanish, being at least $M_0^p$ by Jensen's inequality.
The sharp constant is attained instead by exhibiting a function that dominates the target and is a supermartingale along the process, the method of Burkholder's special functions and of the Bellman functions of the associated obstacle problem; the dynamic-programming object of the same name is unrelated.  That is an $L^p$ (power-mean) convex geometry disjoint from the relative-entropy geometry of this paper.
There the constant has its own certificate---the Bellman function is an equality whose obstacle boundary yields the inequality---so Doob's inequality has an identity of its own, in a different calculus.
Corollary~\ref{cor:doob-residual} states that identity.

The power-mean geometry supplies a residual decomposition of exactly the same form as the entropic one, with two named nonnegative deficits in place of a relative entropy.
The role the Gibbs tilt plays in the entropic geometry is played here by a family of certificate processes indexed by an arbitrary concave function.
For $C^1$ concave $\Phi$, the Az\'ema--Yor process of $M$ with respect to $\Phi$~\cite{AY79,balsubramani2020p} reads the record $M^\ast_t$ through $\Phi$ and subtracts the gap $M^\ast_t-M_t$ at $\Phi$'s slope there.
Every member is affine in $M_t$ below the running maximum and meets the diagonal with matching slope, since $\Phi$ enters only through its value and first derivative at $M^\ast_t$.
As a function of the current value, each certificate is the tangent extension of $\Phi$ at the record---the tangent line of $\Phi$ at $M^\ast_t$ below the record and $\Phi$ itself above---so it is $C^1$ and concave, and nonincreasing whenever $\Phi$ is.
That smooth fit is flat to first order when the running maximum advances, and reading the certificate step by step resolves its optional-stopping deficit exactly.

\begin{proposition}[Az\'ema--Yor certificates and the per-step form of the optional-stopping deficit]
\label{prop:ay-certificate}
Let $(M_t)_{t=0}^T$ be a nonnegative martingale with deterministic $M_0$ and $M^\ast_t:=\max_{s\le t}M_s$,
let $\Phi$ be $C^1$ and concave on $(0,\infty)$, and set
\begin{equation}\label{eq:ay-process}
  A^\Phi_t \;:=\; \Phi(M^\ast_t)-(M^\ast_t-M_t)\,\Phi'(M^\ast_t),
  \qquad
  D_{-\Phi}(a,b)\;:=\;\Phi(b)-\Phi(a)+(a-b)\,\Phi'(b)\;\ge\;0,
\end{equation}
with $D_{-\Phi}$ the Bregman divergence of the convex function $-\Phi$.
Assume every expectation below is finite.
\begin{enumerate}[label=(\roman*),leftmargin=2em]
\item $A^\Phi$ is a supermartingale with an exact increment:
  \begin{equation}\label{eq:ay-increment}
    A^\Phi_t-A^\Phi_{t-1}
    \;=\;(M_t-M_{t-1})\,\Phi'(M^\ast_{t-1})\;-\;D_{-\Phi}\bigl(M^\ast_t,M^\ast_{t-1}\bigr).
  \end{equation}
\item The record obeys the exact identity
  \begin{equation}\label{eq:ay-identity}
    \E\bigl[\Phi(M^\ast_T)\bigr]
    \;=\;\Phi(M_0)
    \;+\;\E\bigl[(M^\ast_T-M_T)\,\Phi'(M^\ast_T)\bigr]
    \;-\;\E\Bigl[\sum_{t=1}^{T}D_{-\Phi}\bigl(M^\ast_t,M^\ast_{t-1}\bigr)\Bigr].
  \end{equation}
\item For $p>1$ and its conjugate exponent $q:=\tfrac{p}{p-1}$, at $\Phi(y)=-y^{p}/(p-1)$ the process $A^\Phi_t$ is $U(M_t,M^\ast_t)$ for the certificate function
  \begin{equation}\label{eq:doob-bellman}
    U(x,y)\;:=\;y^{p}-q\,y^{p-1}x,
    \qquad 0<x\le y,
  \end{equation}
  and the optional-stopping deficit $A^\Phi_0-\E\bigl[A^\Phi_T\bigr]$ is the record's Bregman sum $\E\bigl[\sum_{t=1}^{T}D_{-\Phi}(M^\ast_t,M^\ast_{t-1})\bigr]$.
\end{enumerate}
For a nonnegative submartingale and $\Phi$ in addition nonincreasing, which $-y^{p}/(p-1)$ is, (i) holds verbatim and $A^\Phi$ is again a supermartingale, while the optional-stopping deficit $A^\Phi_0-\E\bigl[A^\Phi_T\bigr]$ exceeds the Bregman sum of (iii) by $-\E\bigl[\sum_t (M_t-M_{t-1})\Phi'(M^\ast_{t-1})\bigr]\ge0$, the drift weighted by the certificate's slope.
\end{proposition}

A step that leaves the running maximum where it was contributes nothing.
Equation~\eqref{eq:ay-increment} values a step that advances it at one Bregman divergence between consecutive records, the second-order remainder the smooth fit leaves: $D_{-\Phi}(y+h,y)=\tfrac12(-\Phi)''(y)h^{2}+o(h^{2})$, so an advance costs nothing to first order and the exact cost is quadratic in the advance.
Continuous paths advance the maximum only infinitesimally, which is why $A^\Phi$ is a local martingale there and strictly a supermartingale on a discrete clock.
The entropic residuals of Section~\ref{sec:master} resolve by the relative-entropy chain rule into per-step conditional divergences, and the optional-stopping deficit resolves by telescoping into per-step Bregman divergences of the record process.
The divergence differs; the per-step resolution does not.

The family also reaches past the $p$-powers, which fix the exponent as the only free parameter: $\Phi=\log$, $\Phi=\sqrt{\cdot}$ and any bounded concave $\Phi$ each give~\eqref{eq:ay-identity} for the corresponding functional of the record.
For nondecreasing concave $\Phi$ the sign of $\Phi'$ reverses, and the middle term of~\eqref{eq:ay-identity} adds to the initial value instead of subtracting from it; the identity then estimates the ultimate record, which is the use made of it in~\cite{balsubramani2020p}.
Doob's inequality is recovered from the decreasing branch by one further step, and that step is where the exponent re-enters: H\"older converts $\E[(M^\ast)^{p-1}M_T]$ into $\|M^\ast\|_p^{p-1}\|M_T\|_p$, and no such conversion is available at general $\Phi$.
The certificate identity and the H\"older conversion together price Doob's inequality exactly.
The certificate~\eqref{eq:doob-bellman} is affine in the running value $x$ with slope $-q\,y^{p-1}$ and, by the smooth fit $\partial_yU(y,y)=0$, is flat to first order when the running maximum advances.
Its tangent extension at the record $y$ is affine below $y$ and equals $-x^{p}/(p-1)$ above, the two branches meeting in value and slope at $x=y$, with both slopes strictly negative.

\begin{corollary}[Doob's $L^p$ slack is an exact three-term residual]
\label{cor:doob-residual}
Let $(M_t)_{t=0}^T$ be a nonnegative \emph{sub}martingale with deterministic $M_0>0$ and $M^\ast:=\sup_{t\le T}M_t$, let $p>1$ and $q=p/(p-1)$ with $\E[M_T^p]<\infty$, and let $U$ be as in~\eqref{eq:doob-bellman}.
\begin{enumerate}[label=(\roman*),leftmargin=2em]
\item $U(M_t,M^\ast_t)$ is a supermartingale and the \emph{optional-stopping deficit}
  \begin{equation}\label{eq:delta-B}
    \delta_{\mathrm B}\;:=\;U(M_0,M_0)-\E\bigl[U(M_T,M^\ast_T)\bigr]
    \;=\;q\,\E\bigl[(M^\ast)^{p-1}M_T\bigr]-\E\bigl[(M^\ast)^{p}\bigr]-\frac{M_0^{p}}{p-1}
  \end{equation}
  is nonnegative: it is the Bregman record sum of Proposition~\ref{prop:ay-certificate}(iii), plus the nonnegative drift term of the submartingale form.
\item With the \emph{H\"older deficit}
  $\delta_{\mathrm H}:=\E[(M^\ast)^p]^{(p-1)/p}\,\E[M_T^p]^{1/p}-\E[(M^\ast)^{p-1}M_T]\ \ge 0$,
  \begin{equation}\label{eq:doob-residual}
    \bigl\|M^\ast\bigr\|_p\;=\;q\,\bigl\|M_T\bigr\|_p\;-\;\cR,
    \qquad
    \cR\;=\;\frac{q\,\delta_{\mathrm H}+\delta_{\mathrm B}+M_0^{p}/(p-1)}
               {\E[(M^\ast)^p]^{(p-1)/p}}\;\ge\;0 .
  \end{equation}
\end{enumerate}
\end{corollary}

\noindent Doob's inequality (Theorem~\ref{thm:doob}) is the statement $\cR\ge0$.
Substituting the two deficits into the numerator of~\eqref{eq:doob-residual} cancels the cross term $\E[(M^\ast)^{p-1}M_T]$ against itself and the initial value against itself, leaving $\cR=q\|M_T\|_p-\|M^\ast\|_p$ exactly, so the inequality and the nonnegativity of $\cR$ are one statement.
The decomposition supplies the source of that nonnegativity, in three interpretable pieces: the H\"older deficit, which vanishes exactly when $M_T$ is proportional to $M^\ast$; the optional-stopping deficit of $U$, which vanishes exactly when $U(M_t,M^\ast_t)$ is a true martingale; and the initial value.
The residual is an identity: \eqref{eq:doob-residual} holds with both deficits nonnegative, and both vanish simultaneously on the constant martingale, where $M_T=M^\ast$.
The initial-value term nonetheless keeps $\cR$ strictly positive there and everywhere:
$\cR\ge M_0^{p}\bigl/\bigl((p-1)\,\E[(M^\ast)^p]^{(p-1)/p}\bigr)>0$, so the sharp constant $(p/(p-1))^p$ is approached but never attained by a nonnegative martingale with $M_0>0$ and $\E[M_T^p]<\infty$.
The decomposition also locates the obstruction that stops the entropy method at Doob quantitatively.
The entropy-natural extremal is the vanishing martingale, and it is exactly the configuration that inflates $\delta_{\mathrm B}$: on the two-point family that sends $M_T$ to zero with probability $r$, the optional-stopping leg of $\cR$ grows $0,\ 0.098,\ 0.511,\ 2.453$ as $r$ runs over $0,\ 0.3,\ 0.6,\ 0.9$ at $p=2$.
Where the entropic geometry is tight, the power-mean geometry loses the most; the two are complementary.

Nothing in the three-term decomposition of Corollary~\ref{cor:doob-residual} is special to Doob's inequality.
Any $L^p$ bound proved by exhibiting a certificate function --- a Burkholder or Bellman function that dominates the target and is a supermartingale along the process --- discards exactly two nonnegative quantities plus an initial value, and which two they are is read off the certificate.
That is the content of the certificate method as developed in~\cite{burkholder1972,burkholder1973ann} and systematized in~\cite{osekowski2012}.
Proposition~\ref{prop:power-certificate} states the accounting once, for an arbitrary certificate.
Value is discarded in exactly three places: between the certificate and the target it dominates, between the certificate's start and its stopped end, and at the start itself.

\begin{proposition}[Certificate form of a power-mean inequality]
\label{prop:power-certificate}
Let $\Gamma,\Xi\ge0$ be integrable and $\cF_T$-measurable, let $C>0$, and let $(U_t)_{t=0}^T$ be an integrable adapted process --- the certificate read along the path --- with $U_0$ deterministic and
\begin{enumerate}[label=(\roman*),leftmargin=2em]
\item $U_T\ \ge\ \Gamma - C\,\Xi$ almost surely, and
\item $(U_t)_{t\le T}$ a supermartingale.
\end{enumerate}
Write $\delta_{\mathrm M}:=\E[U_T-\Gamma+C\,\Xi]\ge0$ for the majorization deficit and $\delta_{\mathrm S}:=U_0-\E[U_T]\ge0$ for the optional-stopping deficit.  Then
\begin{equation}\label{eq:power-certificate}
  C\,\E[\Xi] - \E[\Gamma]
  \;=\; \delta_{\mathrm M} + \delta_{\mathrm S} - U_0,
\end{equation}
and the inequality $\E[\Gamma]\le C\,\E[\Xi]$ is the statement that the right side is nonnegative, which holds whenever $U_0\le0$.
\end{proposition}

Corollary~\ref{cor:doob-residual} is this accounting read at the certificate~\eqref{eq:doob-bellman}, in two steps. 
At the bilinear comparand---$\Gamma=(M^\ast)^p$ against $\Xi=(M^\ast)^{p-1}M_T$ at $C=q$---the Az\'ema--Yor certificate satisfies $U_T=\Gamma-C\,\Xi$ identically, so the majorization deficit vanishes.
The certificate is sharp against this target, \eqref{eq:power-certificate} is exactly~\eqref{eq:delta-B}, and $\delta_{\mathrm S}=\delta_{\mathrm B}$ with initial value $-U_0=M_0^p/(p-1)$.
H\"older then converts the bilinear comparand into the power mean $\|M^\ast\|_p^{p-1}\|M_T\|_p$ losing $\delta_{\mathrm H}$, the one step particular to Doob: against the power-mean target itself, pointwise majorization fails wherever $0<q\,M_T<M^\ast$, and the H\"older deficit stands in expectation where the majorization deficit $\delta_{\mathrm M}$ would stand pointwise.
Whenever a member of the Burkholder--Davis--Gundy or Burkholder--Rosenthal family is certified this way,
\eqref{eq:power-certificate} names what it discards without evaluating it:
the two deficits are computable only once a certificate is written down, and none is written down here at general $p$.
A certificate that is not sharp also inflates $\delta_{\mathrm M}$ by its own sub-optimality, so the decomposition is relative to the certificate chosen and not to the process alone.

The power-mean segment of Table~\ref{tab:classical-ledger} collects the members of the family this section reaches, beside the upper segment's entropic and crossing rows.
The two segments share a pattern and not a mechanism: a classical inequality is an identity minus named nonnegative deficits in both, but those deficits are relative entropies to a Gibbs tilt in one case and certificate deficits in the other.

The last row of that segment is where the method's sharpness is visible: a certificate that attains the constant drives its majorization deficit to zero, so the whole of the slack sits in the optional-stopping deficit and the initial value.
Doob's inequality, whose constant is approached but never attained by a nonnegative martingale with $M_0>0$, keeps a strictly positive initial-value leg for that reason.
Doob's maximal identity is not a row of the segment: it is already an equality,
and it enters this paper through Theorem~\ref{thm:ville-tight} in the optional-stopping geometry.
All three geometries state a classical inequality in the same form, an identity minus named nonnegative deficits, and differ only in which optimizer supplies them: the Gibbs tilt, the dominating certificate, and the crossing itself.

\section{Pooling benefit for multi-model safe testing}
\label{sec:pooling-benefit}

The classical bounds of Section~\ref{sec:master} each used a single test martingale.
Pooling several at once closes the crossing probability entirely.
Corollary~\ref{cor:pooling-crossing-exact} writes it with no inequality left in it, as the Ville baseline minus three named quantities: the compensator shed before the crossing, the mass that never reaches the threshold, and the overshoot at it.
The two inequalities in common use are that decomposition with terms removed, and what they discard is the disagreement among the pooled tests, which is itself evidence.
Let $M^{(1)},\dots,M^{(W)}$ be $W$ nonnegative unit-mean test martingales adapted to $(\cF_t)$ on a common probability space.

Write $M_t^{(\alpha)} := \prod_{i=1}^W \bigl(M_t^{(i)}\bigr)^{\alpha_i}$ for the geometric mixture at weights $\alpha \in \Delta([W])$.
Weighted arithmetic--geometric-mean comparison makes it a nonnegative supermartingale of unit initial value (Lemma~\ref{lem:geometric-mixture}), so with $Z_t(\alpha) := \E_P[M_t^{(\alpha)}]$ the coincidence divergence at these factors is
\begin{equation}\label{eq:pathwise-Calpha}
  \cC_\alpha(t) := -\log Z_t(\alpha) \geq 0
\end{equation}
This coincidence divergence is the coalescent free energy of Section~\ref{sec:coalescent-pathspace}, read at the test-martingale factors $M^{(i)}$; since each $M^{(i)}$ is the density of a path law, it is equally the barycentric form of Section~\ref{sec:barycentric} at those laws.

In the testing-by-betting reading each member is a \emph{bettor} wagering against the null, and the one-step gap of Lemma~\ref{lem:geometric-mixture} removes a slice of mass equal to the bettors' one-step disagreement.
When they disagree the mixture is therefore a \emph{strict} supermartingale, shedding mass before crossing any level.
So $\cC_\alpha$ is the running total of that disagreement along the path, growing the faster the more the $M^{(i)}$ diverge under $P$: without bound when the mass vanishes, $\E_P[M_t^{(\alpha)}] \to 0$, and to a finite limit when only finite mass is lost.
This is the sequential counterpart of the static multi-way coincidence divergence~\cite{balsubramani2026information}.
A test gains something else, and the level indexes it: the mass $g_x(\alpha) \le 1$ the mixture still retains when it first reaches $x$, whose deficit $1 - g_x(\alpha)$ is the disagreement accumulated by that moment.

\begin{corollary}[Coincidence-adjusted Ville bound / pooling benefit]
\label{cor:pooling-benefit}
For $\alpha \in \Delta([W])$ and $x > 1$, let $\tau^{(\alpha)}_x := \inf\{t : M_t^{(\alpha)}
\geq x\}$ and write $g_x(\alpha) := \E\bigl[M_{\tau^{(\alpha)}_x}^{(\alpha)}\,
\ind\{\tau^{(\alpha)}_x < \infty\}\bigr]$ for the mass the geometric mixture
retains across level $x$.  Then
\begin{equation}\label{eq:pooling}
  \PP(\tau^{(\alpha)}_x < \infty)
  = \PP\!\left(\exists\,t : \prod_i \bigl(M_t^{(i)}\bigr)^{\alpha_i}
            \geq x\right)
  \leq \frac{g_x(\alpha)}{x}
  \leq \frac{1}{x},
\end{equation}
and the improvement over the plain Ville bound is the mass deficit $1 - g_x(\alpha) \geq 0$, strictly positive whenever the bettors disagree along the path.
\end{corollary}

Disagreement is sufficient but not necessary: optional-stopping leakage can leave $g_x(\alpha) < 1$ without any pathwise disagreement, a distinction the evaluation below has to control for.

Neither inequality need be there.
The Doob decomposition of the mixture accounts for the second (Proposition~\ref{prop:pooling-exact}) and the overshoot at the crossing for the first, which leaves the pooled crossing probability in closed form where Markov's inequality gives only an upper bound.

\begin{corollary}[Exact pooled crossing probability]
\label{cor:pooling-crossing-exact}
In the setting of Corollary~\ref{cor:pooling-benefit}, let $A$ be the predictable nondecreasing part of the mixture's Doob decomposition, so that $A_{\tau^{(\alpha)}_x}=\sum_{t=1}^{\tau^{(\alpha)}_x}\bigl(M^{(\alpha)}_{t-1}-\E[M^{(\alpha)}_t\mid\cF_{t-1}]\bigr)$ is the disagreement shed by the crossing (Proposition~\ref{prop:pooling-exact}), and write $J^{(\alpha)}_x := M^{(\alpha)}_{\tau^{(\alpha)}_x} - x \geq 0$ for the overshoot at it.
The crossing probability bounded in~\eqref{eq:pooling} is
\begin{equation}\label{eq:pooling-crossing-exact}
  \PP(\tau^{(\alpha)}_x < \infty)
  = \frac{1}{x}\Bigl(
      1
      - \underbrace{\E[A_{\tau^{(\alpha)}_x}]}_{\text{shed before }\tau^{(\alpha)}_x}
      - \underbrace{\E\bigl[M^{(\alpha)}_\infty\ind\{\tau^{(\alpha)}_x=\infty\}\bigr]}
          _{\text{never reaches }x}
      - \underbrace{\E\bigl[J^{(\alpha)}_x\ind\{\tau^{(\alpha)}_x<\infty\}\bigr]}
          _{\text{overshoot at }\tau^{(\alpha)}_x}
    \Bigr).
\end{equation}
The first inequality of~\eqref{eq:pooling} discards the third subtracted term, and the second inequality discards the first two subtracted terms.
Each is an equality exactly when its own terms vanish, so the pooled bound is attained only by a mixture that crosses $x$ without overshoot, sheds nothing before $\tau^{(\alpha)}_x$, and leaves nothing below the threshold.
At $W=1$ the mixture is a martingale and the shed term is absent; if in addition $M^{(1)}_t\to0$ a.s.\ the below-threshold term is absent too, and~\eqref{eq:pooling-crossing-exact} is the first-passage identity~\eqref{eq:fp-exact}.
\end{corollary}

The mass the mixture retains at the crossing and the realized tightening of Ville's bound are two readings of the same deficit, differing by the overshoot; Appendix~\ref{app:pooling-decomposition} separates them.

The benefit is not confined to synthetic bettors, and a paired single-bettor control separates it from leakage on trained vision models: with one bettor's crossing mass at $1.005$ the leakage is negligible, so the deficit at larger ensemble sizes is the mixture's.
That deficit is positive in every multi-member cell, with bootstrap intervals excluding zero in both pools at every ensemble size.
At matched accuracy and ensemble sizes four and eight the architecturally diverse pool sheds more mass than the recipe-diverse one (Appendix~\ref{sec:eval-pooling-benefit-decorrelated}), which the matched accuracy attributes to composition and not to model quality.

% ======================================================================
The bound has a betting reading of its own.
Its partition function is the wealth of a Kelly bettor splitting stake by the weights $\alpha$ across the $W$ strategies, and $\cC_\alpha$ is the accumulated log-wealth disagreement among them.
The pooling benefit is the growth a rebalanced portfolio earns from that disagreement, the sequential and entropic form of the diversification behind universal-portfolio growth.
$\cC_\alpha$ measures disagreement and is blind to membership, so $W$ identical bettors give zero at every $W$ and it need not grow like the $\log W$ regret of the arithmetic-mixture portfolio.

\section{Path-time identities and peeking penalties}
\label{sec:pathtime}
% ======================================================================

Everything up to this point is defined on \emph{path space}, where a trajectory is drawn and every quantity is read off at a fixed horizon or at a stopping time---a moment whose arrival the process itself can recognize.
The phenomena that remain---evaluating an e-process at an \emph{arbitrary} random time, measuring the advantage of a well-chosen observation rule,
quantifying anticipation---all depend on \emph{when} one looks. 
The setting for all of them is the \emph{path-time space} $\Omega\times\N$, whose points are a trajectory paired with one instant at which it might be read.
Choosing when to look is then part of the sample point itself, and the mixed coincidence identity transfers to it essentially unchanged.
The freedom to choose the moment of observation enters there as one new factor,
the \emph{anticipation martingale} of the random time, which isolates the part of the time a process cannot foresee and measures that freedom in nats.

The object being evaluated is an e-process: a nonnegative process $(E_t)_{t \in \N_0}$ with $E_0 = 1$ and $\E_P[E_\tau] \leq 1$ at every $(\cF_t)$-stopping time $\tau$, possibly infinite; equivalently \cite{ramdas2020admissible}, one dominated by a nonnegative supermartingale.
The reciprocal $1/E_\tau$ is a valid p-value at every stopping time, which makes sequential inference on it safe; the penalties below measure the departure from one.

\subsection{The anticipation martingale}
\label{sec:survival}

We isolate the \emph{anticipation martingale} of an arbitrary discrete random time---the martingale factor of the time's occurrence density that measures anticipation in the peeking penalty below, equal to $1$ for stopping times.
Everything below is read off the time's conditional survival and hazard.

Let $\tau$ be a positive-integer-valued random variable on $(\Omega,\cF,(\cF_t), P)$.
The \emph{conditional survival process} and \emph{conditional hazard process} are
\begin{equation}\label{eq:surv-haz}
  S_t := P(\tau > t \mid \cF_t),
  \qquad
  H_t := \frac{P(\tau = t \mid \cF_t)}{P(\tau \geq t \mid \cF_t)},
\end{equation}
with the convention $0/0 := 0$.  When $\tau$ is \emph{oblivious}
(independent of $\cF_\infty$), $S_t$ and $H_t$ are deterministic;
when $\tau$ is a stopping time, $H_t = \ind\{\tau = t\}$ on $\{\tau \geq t\}$.

The hazard resolves the time into two multiplicative parts: a hazard clock that the observed history determines, and a martingale factor that absorbs whatever the history cannot.
Neither displays $\tau$ in its defining formula the way the survival and hazard of~\eqref{eq:surv-haz} do, and both are built from $\tau$ alone, so both take it as a superscript.

Define the \emph{hazard clock} $\clock{\tau}$ and the \emph{martingale factor} $\Amart{\tau}$ by
\begin{equation}\label{eq:clock-mart}
  \clock{\tau}_T := 1 - \prod_{t=1}^T (1 - H_t),
  \qquad
  \Amart{\tau}_T := \prod_{t=1}^T \frac{P(\tau \geq t \mid \cF_t)}
                            {P(\tau \geq t \mid \cF_{t-1})},
\end{equation}
with $\clock{\tau}_0 := 0$ and $\Amart{\tau}_0 := 1$, and with each factor read as $1$ on the event $\{P(\tau \geq t \mid \cF_{t-1}) = 0\}$, where the numerator vanishes almost surely as well.
The convention matters after a stopping time has exhausted its support: without it the factor would read $0$ where the convention gives $1$, and $\Amart{\tau}$ would leave the value it has already settled at.

The two facts about this decomposition that the rest of the paper uses are the product form itself and the \emph{mass identity} $P(\tau = t \mid \cF_t) = \Amart{\tau}_t\,\Delta \clock{\tau}_t$, proved as Theorem~\ref{thm:surv-factor}(v).
The hazard $H_t$ of~\eqref{eq:surv-haz} is $\cF_t$-measurable, so $\clock{\tau}_t$ is adapted but not in general \emph{predictable} ($\cF_{t-1}$-measurable); none of the arguments below requires predictability, and the name records what the object is built from: the hazard $H_t$.

The martingale factor $\Amart{\tau}_T$ is the \emph{anticipation martingale}: it encodes the anticipatory part of $\tau$, and every departure from a stopping time is confined to it, while the hazard clock $1-\clock{\tau}_T$ records what the history determines.
For stopping times, $\Amart{\tau} \equiv 1$ and the hazard clock holds all the information; the construction is unconditional, requiring no integrability or finite-$\tau$ hypothesis.  $\clock{\tau}$ and $\Amart{\tau}$ are built from $\tau$ alone: one hazard clock and one anticipation martingale per random time.
Under a null $P$ that is itself in question the construction is repeated measure by measure, giving $\Amart{\tau,P}$ (Theorem~\ref{thm:composite-peeking}).

\begin{theorem}[The anticipation martingale]
\label{thm:surv-factor}
For the survival and hazard of~\eqref{eq:surv-haz} and the hazard clock and martingale factor of~\eqref{eq:clock-mart}, the processes $\clock{\tau}$ and $\Amart{\tau}$ satisfy:
\begin{enumerate}[label=(\roman*)]
\item $\clock{\tau}$ is nondecreasing with $\clock{\tau}_0 = 0$ and $\clock{\tau}_T \in [0,1]$;
\item $\Amart{\tau}$ is a nonnegative $P$-martingale with $\Amart{\tau}_0 = 1$;
\item the conditional survival satisfies $S_T = (1 - \clock{\tau}_T)\,\Amart{\tau}_T$;
\item the incremental relations are $(\clock{\tau}_t - \clock{\tau}_{t-1})/(1 - \clock{\tau}_{t-1}) = H_t$ and $\Amart{\tau}_t/\Amart{\tau}_{t-1} = P(\tau \geq t \mid \cF_t)/P(\tau \geq t \mid \cF_{t-1})$;
\item $P(\tau = t \mid \cF_t) = \Amart{\tau}_t \cdot (\clock{\tau}_t - \clock{\tau}_{t-1})$.
\end{enumerate}
\end{theorem}

The factorization delivers the identity the peeking calculus runs on: an expectation taken at the random time $\tau$, rewritten on the fixed time grid with each instant weighted by $\Amart{\tau}_t\,\Delta \clock{\tau}_t$.

\begin{theorem}[Representation at random times]
\label{thm:rep}
For any adapted process $(V_t)_{t \geq 1}$ with $\E\bigl[\sum_{t \geq 1} |V_t|\,P(\tau = t \mid \cF_t)\bigr] < \infty$,
\begin{equation}\label{eq:rep}
  \E[V_\tau]
  = \E\!\left[\sum_{t \geq 1} V_t\,\Amart{\tau}_t\,(\clock{\tau}_t - \clock{\tau}_{t-1})\right].
\end{equation}
\end{theorem}

Read pathwise on the index $\{1,2,\ldots\}$ with counting measure, the two summand factors $V_t$ and $\Amart{\tau}_t\,\Delta \clock{\tau}_t$ are themselves nonnegative factors.
For strictly positive $V$ the realized sum in~\eqref{eq:rep} is therefore an instance of Theorem~\ref{thm:mci}, whose optimizer $w^\ast(t)\propto V_t\,\Amart{\tau}_t\,\Delta \clock{\tau}_t$ is the typical-time distribution along that path.

Two classes of random time have zero anticipation ($\Amart{\tau}_t \equiv 1$), so the peeking penalty below degenerates for both: \emph{stopping times}, whose hazard clock $\clock{\tau}$ stays at $0$ until $\tau$ and then jumps to $1$; and \emph{pseudo-stopping times}~\cite{NY05,coculescu2012hazard},
whose martingale factor is $\Amart{\tau}_t \equiv 1$ pointwise (under $P(\tau<\infty)=1$).
Equivalently, the random-time evaluation acts as if it were a stopping time on every bounded test martingale, though the time itself need not be one.
Only $\Amart{\tau}_t \equiv 1$ is used below.

Beyond this dichotomy the anticipation martingale has a full R\'enyi spectrum,
of which the peeking calculus uses only the extremes.

\begin{proposition}[R\'enyi spectrum of anticipation]
\label{prop:renyi-anticipation}
Taking the single factor $\Pi=\Amart{\tau}$ in the path-space partition function of Theorem~\ref{thm:pathspace} gives the anticipation cumulant $\Phi_T(\alpha)=\log Z_T(\alpha)=\log\E[(\Amart{\tau}_T)^\alpha]$, convex in $\alpha$
(Proposition~\ref{prop:grad}), with
\[
  \Phi_T(1)=0,\qquad \Phi_T(0)=\log P(\Amart{\tau}_T>0)\le0,\qquad
  \Phi_T'(1)=\E[\Amart{\tau}_T\log \Amart{\tau}_T]=\KL(Q_A\|P|_{\cF_T}),\quad \tfrac{\dd Q_A}{\dd P}:=\Amart{\tau}_T,
\]
the last the total marginal anticipation of $\tau$ in nats.
Here $\Phi_T\equiv0$ iff $\Amart{\tau}\equiv1$, i.e.\ $\tau$ is a stopping or pseudo-stopping time (Section~\ref{sec:survival}).
More generally $\Phi_T$ is affine exactly when $\Amart{\tau}_T=\ind_B/P(B)$ for some event $B$ with $P(B)>0$ --- the all-or-nothing peek,
of which $P(B)=1$ is the case $\Amart{\tau}\equiv1$ above --- and is strictly convex otherwise.
A fuller classification of the intermediate anticipation classes by the form of $\Phi_T$ lies beyond the present scope.
\end{proposition}

\subsection{The path-time identity}

To price the choice of time itself, move from path space to the \emph{path-time space} that pairs each trajectory with the instant at which it is examined.
Define $\widehat\Omega := \Omega \times \N$ with $\sigma$-field generated by $A \times \{t\}$, $A \in \cF_t$, and the finite positive measure
\begin{equation}\label{eq:phat}
  \widehat P(A\times\{t\}) := \E[\ind_A \Delta \clock{\tau}_t],
\end{equation}
together with the probability measure
\begin{equation}\label{eq:Rlaw}
  R(A\times\{t\}) := \PP(A \cap \{\tau = t\}).
\end{equation}

\begin{lemma}[Random-time path-time law]
\label{lem:Rt}
$\dd R / \dd \widehat P (\omega,t) = \Amart{\tau}_t(\omega)$ on each time-slice.
Equivalently, for every nonnegative adapted process $V$,
$\E[V_\tau] = \int_{\widehat\Omega} V_t(\omega)
\Amart{\tau}_t(\omega)\,\widehat P(\dd\omega,\dd t)$.
\end{lemma}

On this substrate the mixed coincidence identity acquires one extra factor of $\log \Amart{\tau}_t$ beyond the path-space form, and that factor is the information, in nats, that moving the factors $\Pi_{i,t}$ from a deterministic horizon to the random time $\tau$ requires.

\begin{theorem}[Path-time mixed coincidence identity]
\label{thm:pathtime}
Let $\Pi_{1,t},\ldots,\Pi_{W,t}$ be strictly positive adapted factors with $0 < Z_\tau(\alpha) := \E\bigl[\prod_i \Pi_{i,\tau}^{\alpha_i}\bigr]
< \infty$.  Then for every $Q\ll\widehat P$ with
$Q(\{\Amart{\tau}_t\prod_i\Pi_{i,t}^{\alpha_i}=0\})=0$,
\begin{equation}\label{eq:pathtime}
  \log Z_\tau(\alpha)
  = \sum_{i=1}^W \alpha_i\,\E_Q[\log \Pi_{i,t}]
    + \E_Q[\log \Amart{\tau}_t]
    - \KL(Q\|\widehat P)
    + \KL(Q\|Q_\tau^\alpha),
\end{equation}
where the optimizer is
\begin{equation}\label{eq:pathtimeopt}
  \frac{\dd Q_\tau^\alpha}{\dd \widehat P}(\omega,t)
  = \frac{\Amart{\tau}_t(\omega)\prod_i \Pi_{i,t}(\omega)^{\alpha_i}}{Z_\tau(\alpha)}.
\end{equation}
Dropping the nonnegative residual $\KL(Q\|Q_\tau^\alpha)$ recovers the variational form (a supremum over $Q \ll \widehat P$), attained at $Q_\tau^\alpha$.
\end{theorem}

For stopping times \eqref{eq:pathtime} reduces to the path-space identity on $\cF_\tau$, and for anticipatory times the correction $\E_Q[\log \Amart{\tau}_t]$ becomes the peeking penalty of Theorem~\ref{thm:peeking}.

\subsection{Peeking penalties}

Evaluating an e-process (or any nonnegative supermartingale) at an arbitrary random time $\tau$ inflates its expectation by $\log \E[E_\tau]$ relative to the deterministic-time bound $\E[E_t] \leq 1$.
This is the \emph{peeking penalty}: what the analyst incurs, in nats, for looking at the process at a moment that depends on the trajectory in ways the process itself cannot foresee.
It measures, in nats, the effect behind optional-stopping bias---the downward bias in a $p$-value reported at a data-dependent moment, and the estimation of the running extrema that such peeking tracks~\cite{balsubramani2020p}.
As a logarithm of an expectation it is an information-theoretic quantity, on the same nat scale as the relative entropies of the identity that produces it.
It admits an identity---a Donsker--Varadhan free energy that holds at every path-time measure $Q$, the non-optimal ones included, and gains one extra additive term $\E_Q[\log \Amart{\tau}_t]$ that vanishes precisely when $\Amart{\tau} \equiv 1$---so the penalty is the path-time relative-entropy discrepancy the anticipation martingale makes explicit, with no additional slack.

Read directly against the reference measure, before the path-time refinement that exposes the anticipation factor, the penalty is already a Donsker--Varadhan free energy.

For an e-process $(E_t)$ and a random time $\tau$ with $0<\E[E_\tau]<\infty$, that reading is Corollary~\ref{cor:DV} at $g=E_\tau$: for every $Q\ll P$ charging no $\{E_\tau=0\}$,
$\log\E[E_\tau]=\E_Q[\log E_\tau]-\KL(Q\|P)+\KL(Q\|P_{E_\tau})$ with $\dd P_{E_\tau}/\dd P\propto E_\tau$, so dropping the residual gives $\log\E[E_\tau]=\sup_{Q\ll P}\{\E_Q[\log E_\tau]-\KL(Q\|P)\}$.
The path-time refinement below keeps the equality and splits the first term, exposing the anticipation factor the reference measure alone cannot see.

Refining the substrate from $\Omega$ to path-time resolves that free energy one step further, separating the anticipation the time contributes from the tilt the process supplies.

\begin{theorem}[Random-time peeking penalty]
\label{thm:peeking}
Let $E_t(\lambda) = \exp(\lambda Y_t - \overline\Psi_t(\lambda))$ be a nonnegative supermartingale with $E_0(\lambda) \leq 1$, as in Proposition~\ref{prop:supermart-relax}(iii).
Define the peeking penalty $\cP_\tau(\lambda) := \log \E[E_\tau(\lambda)]$, and assume $0 < \E[E_\tau(\lambda)] < \infty$.
Then for every $Q \ll \widehat P$ with $Q(\{\Amart{\tau}_t=0\})=0$,
\begin{equation}\label{eq:peeking}
  \cP_\tau(\lambda)
  = \lambda \E_Q[Y_t] - \E_Q[\overline\Psi_t(\lambda)]
    + \E_Q[\log \Amart{\tau}_t] - \KL(Q\|\widehat P) + \KL(Q\|Q^\star_\tau),
\end{equation}
where $\dd Q^\star_\tau/\dd\widehat P \propto \Amart{\tau}_t\,E_t(\lambda)$; dropping the residual $\KL(Q\|Q^\star_\tau) \geq 0$ gives the variational (supremum) form.
For $\lambda > 0$ and any $A \in \sigma(\tau,\cF_\tau)$ such that on $A$ one has $Y_\tau \geq x$ and $\overline\Psi_\tau(\lambda) \leq c$, $\PP(A) \leq \exp\!\bigl(-\lambda x + c + \cP_\tau(\lambda)\bigr)$.
\end{theorem}

\begin{proposition}[When the penalty vanishes]
\label{prop:vanish}
Let $E_t$ be a nonnegative supermartingale with $E_0 \leq 1$.
\begin{enumerate}[label=(\alph*)]
\item If $\tau$ is a stopping time and $E_\tau$ is integrable,
$\cP_\tau = \log\E[E_\tau] \leq 0$.
\item If $\tau$ is a pseudo-stopping time and $(E_t)$ is a uniformly integrable martingale with $\E[E_0] = 1$, $\cP_\tau = 0$.
\end{enumerate}
\end{proposition}

The pseudo-stopping class admits a further description in these terms: it is exactly the class of times whose anticipation index is identically one (Corollary~\ref{cor:peeking-ball}).

At the opposite extreme the penalty is unbounded, and the mechanism is the non-integrability of a vanishing martingale's running supremum.

\begin{lemma}[Vanishing martingales have non-integrable suprema]
\label{lem:sup-nonintegrable}
Let $(E_t)_{t\in\N_0}$ be a nonnegative martingale with $E_0 = 1$ and $E_t \to 0$ a.s.  Then $\E[\sup_t E_t] = \infty$.
\end{lemma}

\begin{proposition}[Fully anticipatory times yield infinite penalty]
\label{prop:infinite}
Let $(E_t)$ be a nonnegative martingale with $E_0 = 1$ and $E_t \to 0$ a.s., and let $\tau^\star$ be its ultimate-maximum time of~\eqref{eq:tau-star}.
Then $\E[E_{\tau^\star}] = \infty$, hence $\cP_{\tau^\star} = \infty$.
\end{proposition}

Proposition~\ref{prop:infinite} shows that there is no universal finite correction for evaluating arbitrary e-processes at arbitrary non-stopping times: the peeking penalty at the ultimate-maximum time can be infinite.
This is the random-time analogue of the classical warning that peeking can destroy sequential validity; it is quantified here by the path-time partition function.

The penalty extends from a simple null to a composite one as an upper envelope over the null: each measure has its own anticipation factor.

\begin{theorem}[Composite random-time peeking, upper-expectation form]
\label{thm:composite-peeking}
Let $\cP_0$ be a set of measures and $E=(E_t)$ a nonnegative process that is a $P$-supermartingale for every $P\in\cP_0$ with $E_0\le1$.
For a random time $\tau$, let $\Amart{\tau,P}$ and $\widehat P_P$ be the anticipation martingale and clock measure of $\tau$ under $P$
(Theorem~\ref{thm:surv-factor}). Writing $\overline{\E}_{\cP_0}[\cdot]:=\sup_{P\in\cP_0}\E_P[\cdot]$,
the composite peeking penalty $\overline{\cP}_\tau:=\log\overline{\E}_{\cP_0}[E_\tau]$ is the upper envelope $\overline{\cP}_\tau=\sup_{P\in\cP_0}\cP^P_\tau$ of the simple penalties $\cP^P_\tau:=\log\E_P[E_\tau]$, each of which is itself an identity: for every $P\in\cP_0$ with $0<\E_P[E_\tau]<\infty$ and every $Q\ll\widehat P_P$ with $Q(\{\Amart{\tau,P}_tE_t=0\})=0$,
\begin{equation}\label{eq:composite-peeking-simple}
  \cP^P_\tau=\E_Q[\log E_t]+\E_Q[\log \Amart{\tau,P}_t]-\KL(Q\|\widehat P_P)+\KL(Q\|Q^\star_P),
\end{equation}
with $\dd Q^\star_P/\dd\widehat P_P\propto \Amart{\tau,P}_t\,E_t$.
Hence $\overline{\E}_{\cP_0}[E_\tau]\le\exp\!\big(\sup_{P\in\cP_0}\cP^P_\tau\big)$, each simple penalty with its own correction $\E_Q[\log \Amart{\tau,P}_t]$ that vanishes iff $\tau$ is a (pseudo-)stopping time under $P$ (Proposition~\ref{prop:vanish}).
\end{theorem}

When $\cP_0$ is composite in the general sense---$E$ merely \emph{dominated} by a $P$-specific supermartingale $\widetilde E^{P}$ ($E\le\widetilde E^{P}$)---the same argument gives the upper bound $\overline{\cP}_\tau\le\sup_{P\in\cP_0}\cP^{P,\widetilde E^{P}}_\tau$, with slack the domination gap $\E_P[\widetilde E^{P}_\tau-E_\tau]\ge0$ evaluated at $\tau$; equality holds iff the dominating supermartingale is tight at the random time.

% ======================================================================

When $P$ is invariant under a group action---permutations of an exchangeable stream, rotations of a symmetric design---the nonnegative supermartingales that respect the symmetry are exactly the invariant test martingales, and they generate group-symmetric e-processes \cite{ortiz2024ann,ramdas2022testing,lardy2025anytime}.
The peeking penalty (Theorem~\ref{thm:peeking}) then inherits an equivariant form, because the anticipation an invariant random time can exploit is confined to the information the group action leaves free.
The residual $\log \Amart{\tau}_t$ is measured against the invariant hazard clock alone, connecting the random-time calculus to the conformal-martingale tests of exchangeability.

\subsection{The worst-case peeking complexity}
\label{sec:peeking-complexity}

The entropic and power-mean geometries both attach a quantity to a random time, and the two accounts agree at the ends and differ between them.
In the entropic geometry that quantity is the peeking penalty, and its extremal form over a class of times is a constrained Donsker--Varadhan program indexed by a relative entropy on path-time.
In the power-mean geometry it is an inflation factor on the Burkholder--Davis--Gundy comparison, and there it has a closed-form law.
This subsection identifies the index that supports a finite interpolation between the ends---a supremum norm of the anticipation the class allows---and shows that the two weaker readings of the same anticipation, a logarithmic average and a finite moment, both fail.

The worst-case peeking complexity
\begin{equation}\label{eq:peeking-complexity}
  \mathfrak{C}_\cT(E) := \sup_{\tau \in \cT}\log\E[E_\tau]
\end{equation}
of an e-process $E$ against a class $\cT$ of random times has two settled extremes.
It is at most $0$ when $\cT$ holds only stopping times---the defining property of an e-process---and, for uniformly integrable martingales, when it holds pseudo-stopping times as well (Proposition~\ref{prop:vanish}).
It is $+\infty$ the moment $\cT$ admits the ultimate-maximum time $\tau^\star$ of a vanishing martingale (Proposition~\ref{prop:infinite}), whose Pareto-tailed $E_{\tau^\star}$ makes $\E[E_{\tau^\star}]$ diverge.
Between the two extremes sits a family of classes, indexed by how much anticipation each allows.
The first candidate index is the expected log-anticipation,
which cuts out the bounded-anticipation classes
\begin{equation}\label{eq:TB}
  \cT_B := \{\tau : \E[\log \Amart{\tau}_\tau] \leq B\},
\end{equation}
with $\Amart{\tau}$ the martingale factor of $\tau$.  The quantity $B$ caps,
$\E[\log \Amart{\tau}_\tau]$, is the \emph{anticipation budget} of $\tau$, and it is a relative entropy.

\begin{proposition}[The anticipation budget is a relative entropy]
\label{prop:budget-relent}
For any random time $\tau$ with path-time law $R$~\eqref{eq:Rlaw} and clock measure $\widehat P$~\eqref{eq:phat}, the anticipation budget is a relative entropy,
\[
  \E[\log \Amart{\tau}_\tau]
  = \E_R\!\left[\log\frac{\dd R}{\dd\widehat P}\right]
  = \KL(R\,\|\,\widehat P),
\]
so $\cT_B = \{\tau : \KL(R\|\widehat P)\le B\}$ is a divergence ball in the path-time law, and $\mathfrak{C}_{\cT_B}(E) = \log V(B)$ with the linear-objective value $V(B) := \sup\{\E_R[E_t] : R \text{ a path-time law with }
\Omega\text{-marginal } P,\ \KL(R\|\widehat P(R))\le B\}$.
\end{proposition}

The log budget fails, because it reads the wrong functional of the anticipation martingale: $\E[\log \Amart{\tau}_\tau]$ is a logarithmic average, where the complexity is a supremum norm of a linear one.
One example shows that the failure is complete---the complexity of $\cT_B$ is infinite at every positive budget.
Take $E_t=2^t\ind\{X_1=\cdots=X_t=1\}$ on i.i.d.\ fair bits.
At the ultimate-maximum time $\Amart{\tau^\star}=\tfrac12E$ on $\{X_1=1\}$ and $\Amart{\tau^\star}\equiv1$ on $\{X_1=0\}$, so the budget $\E[\log \Amart{\tau^\star}_{\tau^\star}]$ is $\tfrac12\log 2$ while $\E[E_{\tau^\star}]=\infty$.
Truncating to $\tau_K:=\tau^\star$ on $\{\tau^\star\ge K\}$ and $1$ elsewhere drives the budget to $0$ and leaves $\E[E_{\tau_K}]$ infinite.
So $\mathfrak{C}_{\cT_B}(E)=+\infty$ at every $B>0$, and the class $\cT_B$ supports no finite interpolation.

A budget that controls the complexity must read a moment of $\Amart{\tau}_\tau$ itself.
In the same example $\Amart{\tau^\star}_{\tau^\star}$ is half the divergent $E_{\tau^\star}$ on $\{X_1=1\}$, so even a first-moment budget excludes $\tau^\star$.
Such a budget cuts out classes on which the peeking complexity~\eqref{eq:peeking-complexity} is finite, interpolating the endpoint values $0$ and $+\infty$, and on such a class Theorem~\ref{thm:peeking} reads as a constrained Donsker--Varadhan program with the budget an additive Lagrangian term.

A Lagrangian treatment of $V(B)$ would need the budget functional $R\mapsto\KL(R\|\widehat P(R))$ to be convex, and it is not: the multiplicative hazard clock $\clock{\tau}=1-\prod(1-H)$ makes $R\mapsto\widehat P(R)$ nonlinear, so convexity is not inherited from the additive compensator.
On a \emph{fixed} path-time reference $\widehat P$ the constraint $Q\mapsto\KL(Q\|\widehat P)$ is convex and the anticipation is capped against a hazard clock the time cannot move, so the Lagrangian collapses and $B\mapsto\mathfrak{C}_{\cT_B}$ is concave and non-decreasing.
That case speaks to neither the self-referential map $R\mapsto\widehat P(R)$ nor the interpolation below, which is exact and needs no Lagrangian.

The budget that succeeds reads the same anticipation against its own hazard clock, and in $L^\infty$.  Write
\begin{equation}\label{eq:anticipation-index}
  \mathfrak{A}^\tau
  := \sum_{t\ge1} P(\tau = t \mid \cF_t)
   = \sum_{t\ge1} \Amart{\tau}_t\,\Delta \clock{\tau}_t
\end{equation}
for the \emph{anticipation index} of $\tau$, the second equality being the mass identity of Theorem~\ref{thm:surv-factor}(v).
It is $\cF_\infty$-measurable with $\E[\mathfrak{A}^\tau]=\sum_{t\ge1}\PP(\tau=t)=1$, so $\mathfrak{A}^\tau$ is a probability density against $P$, and it is the quantity the worst case over e-processes reads.

\begin{theorem}[The worst case over e-processes]
\label{thm:peeking-radius}
Let $\tau$ be an almost surely finite random time.  Then
\begin{equation}\label{eq:peeking-radius}
  \sup_E\,\E[E_\tau] \;=\; \esssup \mathfrak{A}^\tau ,
\end{equation}
the supremum running over nonnegative supermartingales $E$ with $E_0\le1$.
\end{theorem}

\begin{corollary}[The interpolating classes]
\label{cor:peeking-ball}
For $\beta\ge1$ let $\cT^{\mathfrak{A}}_\beta:=\{\tau:\esssup\mathfrak{A}^\tau\le\beta\}$.
Then $\mathfrak{C}_{\cT^{\mathfrak{A}}_\beta}(E)\le\log\beta$ for every nonnegative supermartingale $E$ with $E_0\le1$, and $\sup_E\mathfrak{C}_{\cT^{\mathfrak{A}}_\beta}(E)=\log\beta$ whenever the class contains a time attaining $\beta$.
Moreover $\mathfrak{A}^\tau\equiv1$ almost surely if and only if $\tau$ is a pseudo-stopping time, and since $\E[\mathfrak{A}^\tau]=1$ the class $\cT^{\mathfrak{A}}_1$ is exactly that one, where the penalty is $0$.
\end{corollary}

A finite moment of the index is not enough either.
At the ultimate-maximum time of the same example $\mathfrak{A}^{\tau^\star}$ is half the number of leading ones, and $1$ when there are none, so every moment $\E[(\mathfrak{A}^{\tau^\star})^s]$ with $s<\infty$ is finite while $\esssup\mathfrak{A}^{\tau^\star}=\infty$ and the complexity is already $+\infty$.
Mixing $\tau^\star$ with a stopping time on an independent coin drives any one of those moments to its floor of $1$ and changes neither the essential supremum nor the complexity.

The peeking identity carries over verbatim to continuous time, where the anticipation martingale plays the part of the exponential density in a Girsanov change of measure and no quadratic-variation term enters (Appendix~\ref{app:random-time-ct}).
In the power-mean geometry the Burkholder--Davis--Gundy comparison meets the same question, agreeing with the peeking penalty at both extremes and supplying a closed-form inflation law between them.

% ======================================================================
\section{Related work}
\label{sec:related}
% ======================================================================

\paragraph{Variational principles.}
The Donsker--Varadhan variational principle \cite{donsker1975asymptotic} and the Gibbs variational formula are classical tools in large deviations \cite{dembo1998b} and statistical mechanics; the large-deviations route to the free energies and Legendre--Fenchel duals that recur throughout is surveyed in \cite{touchette2009large}.
The mixed coincidence identity generalizes these by allowing multiple nonnegative factors with arbitrary real exponents, and by upgrading the usual \emph{inequality} into an \emph{equality} with an explicit relative-entropy residual \cite{balsubramani2026information}.
The same functional fixes the value of a robust optimization.
For a diffusion whose region and instantaneous covariation are known but whose drift is not, the admissible models are those whose occupancy-time measures converge to a prescribed density.
The best asymptotic growth rate attainable against that whole class equals the Donsker--Varadhan occupancy-time rate function at the density, and one explicit strategy attains it under every model in the class \cite{kardaras2021ergodic}.

\paragraph{Classical martingale concentration.}
The line-crossing and time-uniform Bernstein bounds recovered in Section~\ref{sec:master} specialize the classical Freedman tail bound~\cite{freedman1975tail}; the method of mixtures (Proposition~\ref{prop:mixture}) supplies the anytime / curved-boundary forms of the confidence-sequence literature~\cite{howard2021ann}.
The sharper one-sided Hoeffding-type inequalities with explicit conditional-variance dependence~\cite{fan2015exponential} follow by the same route, on substituting the corresponding conditional-variance cumulant majorant in the entropic inequality of Proposition~\ref{prop:supermart-relax}.
The same exponential-supermartingale construction---a running sum of conditional log-moment-generating functions controlled by Ville's inequality---underlies the data-dependent concentration bounds for sequential prediction~\cite{zhang2005data}, which convert online mistake bounds into batch generalization guarantees; Theorem~\ref{thm:master} records the relative-entropy equality behind that exponential step.

\paragraph{Certificate functions and sharp constants.}
The sharp constant in Doob's $L^p$ maximal inequality, and in the wider family of martingale transform and square-function bounds, is reached by exhibiting a function that dominates the target and is a supermartingale along the process---Burkholder's special functions~\cite{burkholder1972,burkholder1973ann}, systematized through the associated obstacle problems in~\cite{osekowski2012}.
Section~\ref{sec:doob-bellman} reads that construction as a residual identity: a bound proved this way discards a majorization deficit, an optional-stopping deficit and an initial value, and which two deficits they are is read off the certificate.

\paragraph{R\'enyi divergences.}
R\'enyi divergences and their variational characterizations have been studied extensively \cite{renyi1961measures,erven2014b}.
The sequential decomposition of R\'enyi divergences into per-step conditional terms via the chain rule appears implicitly in the information-theoretic literature but has not previously been connected to a master identity acting on multiple factors. 
Our usage of unnormalized factors, outside standard probability distributions, does follow an initial theme of R\'enyi divergences \citep{renyi1961measures, renyi1965foundationsinfotheory, balsubramani2026information}.

\paragraph{Random times and survival analysis.}
The survival-analytic framework for random times draws on the classical theory of Az\'ema supermartingales~\cite{quelques1972quelques}, their multiplicative decomposition~\cite{nikeghbali2006doobs}, the pseudo-stopping time characterization~\cite{NY05}, the hazard-process approach~\cite{coculescu2012hazard}, and the behavior of optional processes up to random times~\cite{kardaras2015stochastic}.
The enlargement-of-filtrations and hazard-process theory underlying these constructions is collected in~\cite{jeanblanc2009mathematical}.
Two strands of that theory run in the opposite direction to Theorem~\ref{thm:surv-factor}, and together they delimit what the factorization records.
The first is the inverse problem.
Every $[0,1]$-valued c\`adl\`ag submartingale with terminal value $1$ is the conditional-distribution process of some random time, built from a multiplicative system in Meyer's sense, and the correspondence is onto without being injective~\cite{li2012multiplicative}.
The predictable and the optional multiplicative system attached to one Az\'ema supermartingale have different conditional laws, with uniqueness restored under a cocycle hypothesis on the conditional-survival field together with predictability~\cite{li2012multiplicative}.
The hazard clock and anticipation martingale of~\eqref{eq:clock-mart} therefore determine the conditional survival and the diagonal masses, and nothing past them: two times sharing both may still differ in law.
The second strand concerns the dependence on $P$ that Theorem~\ref{thm:composite-peeking} asserts, and it is realized exactly in continuous time~\cite{gapeev2010constructing}.
For a strictly positive local martingale $N$ with $N_0=1$ and a continuous hazard $\Lambda$ with $Ne^{-\Lambda}\le1$, a random time and an equivalent measure $Q$ agreeing with $P$ on $(\cF_t)$ exist whose Az\'ema supermartingale under $Q$ is $Ne^{-\Lambda}$.
One time then has an anticipation factor equal to $1$ under $P$ and to $N$ under $Q$, so the per-measure factors of Theorem~\ref{thm:composite-peeking} genuinely differ even across measures that agree on the base filtration.
The hazard fixes the behavior before the time; the conditional law after it is recovered only from the full density $P(\tau\in\dd\theta\mid\cF_t)$, of which the intensity is the diagonal~\cite{elkaroui2010what}.
Our discrete-time construction specializes these results to a setting where every step can be made algebraically explicit.

\paragraph{E-processes and safe anytime-valid inference.}
E-processes and safe anytime-valid inference have been developed by \cite{howard2021ann,heide2024safe,ramdas2020admissible,ruf2022composite},
against a backdrop of testing-by-betting and the e-value calibration framework \cite{shafer2021testingb,vovk2021b}, including the data-driven significance levels that e-values license beyond the Neyman--Pearson paradigm \cite{Grunwald22} and the online testing of randomness and exchangeability via conformal martingales \cite{vovk2005algorithmic,vovk2021testingb}.
Recent betting-based confidence-sequence work \cite{waudbysmith2024estimatingf,shekhar2024nearoptimality}, the game-theoretic synthesis of \cite{ramdas2023statistical}, and the universal-inference literature \cite{wasserman2020universal,larsson2025adaptive} are particularly close in spirit; these constructions implicitly use the variational equality exhibited here.
PAC-Bayes martingale inequalities originate with \cite{seldin2011pacbayesian}; recent unified treatments \cite{chugg2023unified,jang2024tighter} likewise reduce to specific instantiations of the path-space mixed coincidence identity.
In particular, the unified recipe for deriving time-uniform PAC-Bayes bounds from supermartingale constructions \cite{chugg2023unified} formally specializes to Theorem~\ref{thm:pacbayes} when the supermartingale is exponential.
Our pathwise mixture bound (Proposition~\ref{prop:mixture}) and PAC-Bayes martingale inequality (Theorem~\ref{thm:pacbayes}) recover these developments as path-space mixed coincidence identities.

\paragraph{Random-time taxonomy and multifractal analysis.}
The anticipation martingale $\Amart{\tau}$ of Section~\ref{sec:survival} has a R\'enyi spectrum of anticipation, a temporal analogue of the classical spatial multifractal formalism \cite{halsey1986,hentschel1983infinite,falconer2014fractal,bacry2001multifractal},
recorded in Proposition~\ref{prop:renyi-anticipation}.

\paragraph{Scope of the entropy side.}
The coincidence calculus covers the entropy side of martingale concentration, the bounds that depend on Gibbs factors, partition functions and relative entropy: exponential concentration, large deviations,
line-crossing bounds, change of measure, PAC-Bayes, and the thermodynamic formalism of path overlaps.
The complementary family---the sharp $L^p$ inequalities of Burkholder--Davis--Gundy, Davis, and Burkholder--Rosenthal type---lives in a different convex geometry, carried by Burkholder functions and deterministic pathwise inequalities, and Section~\ref{sec:doob-bellman} makes its accounting exact as well (Proposition~\ref{prop:power-certificate}).
A bound in any of the three geometries supplies the quantity that measures its slack---a relative entropy in the Gibbs geometry, a pair of deficits in the power-mean one, and an optional-stopping quantity at the crossing.

% ======================================================================
\appendix
\appendixtitle

\section{Notation}
\label{app:glossary}

$D$ denotes the relative and R\'enyi divergences, written with their two arguments, while $D_t(\lambda)$ is the cumulant gap of the supermartingale relaxation.  $R$ is the path-time law of $\tau$, a measure on $\widehat\Omega$ used only from Section~\ref{sec:pathtime} onward.
$R_t$ is the one-step likelihood ratio and $R^*_\alpha$ its tilted optimizer, and $R_s$ is the quadrature residual of the discrete crossing law.
The letter also writes the range of a bounded law throughout Section~\ref{sec:cumulant-majorants}.  $S$ is the state space of a Markov family, $S_t$ the conditional survival of $\tau$,
and $S_\alpha(p)$ an effective support size.  $Z(\alpha)$, $Z_T(\alpha)$ and $Z_t(\alpha)$ are partition functions, $Z_t$ the Az\'ema supermartingale of Section~\ref{sec:continuous-time}, and the integrand $Z$ of the transportation lemma at the close of Section~\ref{sec:info-background} is local to that paragraph.  $\tau$,
$\tau^\star$, $\tau^{\mathrm{cert}}$ and $\tau^{(\alpha)}_x$ are random times.

\subsection*{Spaces and laws}

\begin{center}
\begin{tabular}{@{}p{0.28\textwidth}>{\raggedright\arraybackslash}p{0.66\textwidth}@{}}
\toprule
Symbol & Meaning (first occurrence) \\
\midrule
$(\cX,\cB,\nu)$; $(\Omega,\cF,(\cF_t),P)$ & base measurable space and reference
  measure (Sec.~\ref{sec:info-background}); filtered probability space
  (Sec.~\ref{sec:pathspace}) \\
$P|_{\cF_T}$; $Q,\widetilde Q$ & path law at horizon $T$, and an alternative or
  variational path law (Thm.~\ref{thm:pathspace}) \\
$\widehat\Omega=\Omega\times\N$; $\widehat P$ & path-time space and its clock
  measure, \eqref{eq:phat}; $\widehat P_P$ under a null $P$
  (Thm.~\ref{thm:composite-peeking}) \\
$R$ & path-time law of $\tau$, \eqref{eq:Rlaw} \\
$\mu,\pi$ & prior on the parameter space (Prop.~\ref{prop:mixture}) and on the
  model space (Thm.~\ref{thm:pacbayes}) \\
$\rho_t,\rho^*_t$ & adapted posterior and running Gibbs tilt,
  \eqref{eq:mixturepathwise} \\
\bottomrule
\end{tabular}
\end{center}

\subsection*{Information quantities}

\begin{center}
\begin{tabular}{@{}p{0.28\textwidth}>{\raggedright\arraybackslash}p{0.66\textwidth}@{}}
\toprule
Symbol & Meaning (first occurrence) \\
\midrule
$H(p), H(p,\pi), H_\alpha, S_\alpha$ & entropy and cross-entropy
  (Sec.~\ref{sec:info-background}); R\'enyi entropy of order $\alpha$,
  \eqref{eq:renyi-ent}; the effective support size
  $S_\alpha(p):=\exp(H_\alpha(p))=\bigl(\int p^\alpha\,\mathrm{d}\nu\bigr)^{1/(1-\alpha)}$ \\
$\KL(\cdot\|\cdot), D_\alpha$ & two divergences, both set as
  $D$ and separated by their arguments: relative entropy, including the
  finite-measure version \eqref{eq:finite-entropy}; and R\'enyi divergence of
  order $\alpha$, \eqref{eq:renyi-div} \\
$\cC_\alpha$ & coincidence divergence $-\log Z(\alpha)$, and $-\log Z_t(\alpha)$
  on path space, \eqref{eq:pathwise-Calpha}; its mass and crossing forms at a
  crossing are written $\cC^{\mathrm{mass}}_\alpha$, $\cC^{\mathrm{cross}}_\alpha$ \\
$\cP_\tau(\lambda)$ & peeking penalty $\log\E[E_\tau(\lambda)]$
  (Thm.~\ref{thm:peeking}); $\cP^P_\tau$, $\overline{\cP}_\tau$ composite
  (Thm.~\ref{thm:composite-peeking}) \\
\bottomrule
\end{tabular}
\end{center}

\subsection*{Partition functions and tilts}

\begin{center}
\begin{tabular}{@{}p{0.28\textwidth}>{\raggedright\arraybackslash}p{0.66\textwidth}@{}}
\toprule
Symbol & Meaning (first occurrence) \\
\midrule
$\pi_1,\dots,\pi_W$; $\alpha$, $\bar\alpha$ & factors, exponent vector and its
  sum (Sec.~\ref{sec:mci}) \\
$Z(\alpha)$; $\Pi_{i,T}$, $Z_T(\alpha)$; $Z_\tau(\alpha)$ & mixed partition
  function \eqref{eq:Zalpha}; the path-space factors and partition function
  \eqref{eq:pathZ}; its path-time form (Thm.~\ref{thm:pathtime}) \\
$p^*_\alpha$, $Q_T^\alpha$, $R^*_\alpha$; $\pi_\lambda$ & the Gibbs optimizer of
  each, \eqref{eq:pstar}, \eqref{eq:Qtstar}; and the tilt of a PAC-Bayes prior by
  the loss (Sec.~\ref{sec:pacbayes-mart}) \\
$\Phi_T(\alpha)$ & log-partition function, convex in $\alpha$
  (Prop.~\ref{prop:grad}) \\
$\Psi(\lambda)$; $v(\lambda)$ & cumulant generating function of an increment and
  the intrinsic time $2\Psi(\lambda)/\lambda^2$ it sets,
  \eqref{eq:intrinsic-time} \\
$T_\alpha$, $r(T_\alpha)$ & transfer operator and its spectral radius
  (Prop.~\ref{prop:markov}) \\
\bottomrule
\end{tabular}
\end{center}

\subsection*{Processes, times and crossings}

\begin{center}
\begin{tabular}{@{}p{0.28\textwidth}>{\raggedright\arraybackslash}p{0.66\textwidth}@{}}
\toprule
Symbol & Meaning (first occurrence) \\
\midrule
$M_t$; $M^{(i)}$, $M^{(\alpha)}$; $M^\ast$ & a nonnegative unit-mean martingale
  (Sec.~\ref{sec:pathspace}); pooled test martingales and their geometric
  mixture (Lem.~\ref{lem:geometric-mixture}); the running maximum
  $\sup_{t\le T}M_t$ (Cor.~\ref{cor:doob-residual}) \\
$d_t$, $\langle M\rangle_t$ & martingale increment $M_t-M_{t-1}$ and predictable
  variance (Sec.~\ref{sec:mart-ineq}) \\
$R_t$, $L_T$ & one-step and cumulative likelihood ratio,
  \eqref{eq:one-step-LR}; $L_T^{(\alpha)}$ its tilted product, \eqref{eq:seqZ} \\
$\Psi_t(\lambda)$, $\cE_t(\lambda)$ & running conditional log-CGF
  \eqref{eq:exact-cgf} and its exponential martingale
  (Thm.~\ref{thm:master}) \\
$\overline\Psi_t(\lambda)$, $D_t(\lambda)$, $E_t(\lambda)$; $\overline E_t$ &
  cumulant majorant, its gap $\overline\Psi_t - \Psi_t$, and the exponential
  supermartingale (Prop.~\ref{prop:supermart-relax}); the mixture
  $\int E_t\,\dd\mu$ (Prop.~\ref{prop:mixture}) \\
$N$, $A_t$ & martingale part and predictable compensator of a Doob
  decomposition (Prop.~\ref{prop:pooling-exact}) \\
$U(x,y)$; $U_t$ & Doob certificate \eqref{eq:doob-bellman}; a certificate read
  along the path (Prop.~\ref{prop:power-certificate}) \\
$\delta_{\mathrm H}$, $\delta_{\mathrm B}$, $\cR$; $p$, $q=p/(p-1)$ & H\"older
  deficit, optional-stopping deficit and total Doob residual
  (Cor.~\ref{cor:doob-residual}); the conjugate exponents
  (\eqref{eq:doob-bellman}) \\
$\tau$; $\tau^\star$ & random time, a stopping time when adapted
  (Sec.~\ref{sec:pathspace}); the ultimate-maximum time,
  \eqref{eq:tau-star} \\
$\sigma_x$; $J_x$ & first passage of a single martingale to level $x$
  (Thm.~\ref{thm:firstpassage-exact}); the overshoot $M_{\sigma_x}-x$ there \\
$\tau^{(\alpha)}_x$; $J^{(\alpha)}_x$; $g_x(\alpha)$ & first passage of the
  geometric mixture $M^{(\alpha)}$ to level $x$, its overshoot
  $M^{(\alpha)}_{\tau^{(\alpha)}_x}-x$ there, and the mass the mixture retains
  across the level (Cor.~\ref{cor:pooling-benefit}); each depends on both the
  level and the weights \\
$S_t$, $H_t$; $Z_t$ & conditional survival and hazard of $\tau$,
  \eqref{eq:surv-haz}; the Az\'ema supermartingale, their continuous-time form
  (Thm.~\ref{thm:pathtime-ct}) \\
$\clock{\tau}_t$, $\Amart{\tau}_t$ & hazard clock and anticipation martingale of
  $\tau$, \eqref{eq:clock-mart}; $\Amart{\tau,P}$ under a null $P$
  (Thm.~\ref{thm:composite-peeking}) \\
$\varphi$, $A^{-1}_u$ & boundary read against the increasing process, and the
  right-continuous inverse of $A$ (Thm.~\ref{thm:curved-crossing}) \\
$\delta_{\mathrm M}$, $\delta_{\mathrm S}$ & majorization deficit and
  optional-stopping deficit of a certificate function
  (Prop.~\ref{prop:power-certificate}); the power-mean counterparts of
  $\delta_{\mathrm H}$, $\delta_{\mathrm B}$ \\
$\mathfrak{C}_\cT(E)$; $\cT_B$ & worst-case peeking complexity of an e-process
  over a class of times, and the class whose anticipation budget is at most $B$
  (Sec.~\ref{sec:peeking-complexity}) \\
$\mathfrak{A}^\tau$ & anticipation index of $\tau$, the hazard-clock average
  $\sum_t \Amart{\tau}_t\,\Delta \clock{\tau}_t$ of its anticipation martingale,
  \eqref{eq:anticipation-index} \\
$I_\tau$, $\Upsilon_\tau$ & running infimum of the Az\'ema supermartingale
  before an honest time, and the logarithmic factor it enters through
  (Prop.~\ref{prop:bdg-honest}) \\
\bottomrule
\end{tabular}
\end{center}

% ======================================================================
% Bibliography
% ======================================================================
\section{Classical specializations}\label{app:specializations}
% ======================================================================

Each statement below is the formal form of a specialization the body describes and puts to work; the entropic accounting that produces it is Theorem~\ref{thm:master} with Proposition~\ref{prop:supermart-relax}.

The first is a standard result, included because the crossing statements depend on its exact hypotheses.
Which of them a martingale satisfies decides whether its mass is conserved at the time or leaks away, so the precise statement is worth having in one place.

\begin{theorem}[Optional stopping for nonnegative supermartingales]
\label{thm:OST}
Let $(M_t)$ be a nonnegative supermartingale and $\tau$ a stopping time (possibly infinite, with $M_\infty := \liminf_{t\to\infty} M_t$).
Then
\[
  \E[M_\tau] \leq M_0 .
\]
If $(M_t)$ is a nonnegative martingale, the matching equality $\E[M_\tau] = M_0$ holds in each of the following cases.
\emph{(a)} $\tau$ is bounded.
\emph{(b)} The stopped family $\{M_{\tau\wedge n}\}_{n\geq 0}$ is uniformly integrable (in particular when $(M_t)$ is closed by an integrable terminal variable, or $\tau$ is a.s.\ finite and $(M_t)$ is uniformly integrable).
Integrability of $M_\tau$ alone does \emph{not} suffice: a nonnegative martingale can leak mass to $\{\tau = \infty\}$, so that $\E[M_\tau] < M_0$ even though $\E[M_\tau] < \infty$.
\end{theorem}

\begin{theorem}[Doob's $L^p$ maximal inequality]
\label{thm:doob}
Let $(M_t)_{t=0}^T$ be a nonnegative submartingale and $p > 1$.  Then $\E\bigl[\sup_{t \leq T} M_t^p\bigr] \leq \bigl(\tfrac{p}{p-1}\bigr)^p\,\E[M_T^p]$.
\end{theorem}

\begin{proposition}[Conditional maximal identity~\cite{Nikeghbali2006essay}]
\label{prop:maximal-conditional}
Let $(M_t)_{t\ge0}$ be a nonnegative continuous-time \emph{local} martingale --- a martingale up to each of a sequence of stopping times increasing to infinity --- with $M_0=m>0$ and $\lim_{t\to\infty}M_t=0$ a.s., whose running supremum $M^\ast_t:=\sup_{u\le t}M_u$ has continuous paths, and write $M^\ast_\infty:=\sup_{u\ge0}M_u$.
\begin{enumerate}[label=(\roman*),leftmargin=2em]
\item For every $a>0$, $\PP\bigl(M^\ast_\infty>a\bigr)=\frac{m}{a}\wedge1$, and $m/M^\ast_\infty$ is uniform on $(0,1)$.
\item Suppose in addition that $M$ is strictly positive with no positive jumps, and let $\tau$ be an a.s.\ finite stopping time.
Write $M^\ast_{\ge\tau}:=\sup_{u\ge\tau}M_u$ for the \emph{future supremum} from $\tau$.  Then for every $a>0$,
\begin{equation}\label{eq:maximal-cond}
  \PP\bigl(M^\ast_{\ge\tau}>a \,\big|\, \cF_\tau\bigr)=\frac{M_\tau}{a}\wedge1
  \qquad\text{a.s.}
\end{equation}
The ratio $M_\tau/M^\ast_{\ge\tau}$ is uniform on $(0,1)$ and independent of $\cF_\tau$, and $\log\bigl(M^\ast_{\ge\tau}/M_\tau\bigr)$ is a standard exponential variable independent of $\cF_\tau$.
\end{enumerate}
\end{proposition}

\begin{theorem}[Azuma--Hoeffding, exact tail and its bound]
\label{thm:azuma}
Let $(M_t)_{t \geq 0}$ be a martingale with increments $d_t = M_t - M_{t-1}$ satisfying $a_t \leq d_t \leq b_t$ a.s.\ and $\E[d_t \mid \cF_{t-1}] = 0$, where $a_t,b_t$ are deterministic.
Then for every $t \geq 1$ and every $x > 0$ with $\PP(M_t - M_0 \geq x) > 0$ the tail is a per-step relative-entropy sum,
\begin{equation}\label{eq:azuma-exact}
  -\log \PP(M_t - M_0 \geq x)
  = \sum_{s=1}^t \E_{P^*}[\KL(P^*_s\|P_s)],
\end{equation}
where $P^* = P(\cdot \mid M_t - M_0 \geq x)$, and relaxing each conditional relative entropy to the Hoeffding cumulant gives the Azuma--Hoeffding tail and line-crossing bounds
\begin{align}
  \PP(M_t - M_0 \geq x)
  &\leq \exp\!\left(-\frac{2 x^2}{\sum_{s=1}^t (b_s - a_s)^2}\right),
  \label{eq:azuma1}\\
  \PP\!\left(\sup_{1 \leq s \leq t}(M_s - M_0) \geq x\right)
  &\leq \exp\!\left(-\frac{2 x^2}{\sum_{s=1}^t (b_s - a_s)^2}\right).
  \label{eq:azuma2}
\end{align}
\end{theorem}

\begin{lemma}[One-step Bernstein bound]
\label{lem:bernstein}
Let $X$ satisfy $\E[X\mid\cG] = 0$ and $X \leq b$ a.s.  Then for every $0 \leq \lambda < 3/b$,
\begin{equation}\label{eq:bernstein-lemma}
  \E[e^{\lambda X}\mid\cG]
  \leq \exp\!\left(\frac{\lambda^2}{2(1 - \lambda b/3)}\,
                   \E[X^2\mid\cG]\right).
\end{equation}
\end{lemma}

\begin{proposition}[Freedman--Bernstein line-crossing]
\label{prop:freedman}
If $(M_t)$ is a martingale with increments $d_t \leq b$ a.s., then for each $0 \leq \lambda < 3/b$,
\[
  E_t(\lambda) := \exp\!\left(\lambda (M_t - M_0)
    - \frac{\lambda^2}{2(1 - \lambda b/3)} \langle M\rangle_t\right)
\]
is a nonnegative supermartingale.  Consequently, for all $x,v > 0$,
\begin{equation}\label{eq:freedman}
  \PP\!\left(\exists\, t : M_t - M_0 \geq x,\; \langle M\rangle_t \leq v\right)
  \leq \exp\!\left(-\frac{x^2}{2(v + bx/3)}\right).
\end{equation}
\end{proposition}

\begin{proposition}[Mixture supermartingales]
\label{prop:mixture}
Suppose that for every $\lambda \in \Lambda$ (a measurable parameter space), $E_t(\lambda)$ is a nonnegative supermartingale with $E_0(\lambda) \leq 1$.
Let $\mu$ be a prior probability measure on $\Lambda$ and define $\overline E_t := \int E_t(\lambda)\,\mu(\dd\lambda)$.
Then $(\overline E_t)$ is a nonnegative supermartingale.
Moreover, for every adapted posterior $\rho_t \ll \mu$ (possibly data-dependent) with $\rho_t(\{\lambda : E_t(\lambda) = 0\}) = 0$, the mixture log-wealth is a Donsker--Varadhan identity,
\begin{equation}\label{eq:mixturepathwise}
  \log \overline E_t
  = \int_\Lambda \log E_t(\lambda)\,\rho_t(\dd\lambda)
       - \KL(\rho_t \| \mu) + \KL(\rho_t \| \rho^*_t)
\end{equation}
with $\dd\rho^*_t/\dd\mu \propto E_t(\cdot)$; dropping the nonnegative residual $\KL(\rho_t\|\rho^*_t)$ gives the PAC-Bayes lower bound $\log \overline E_t \ge \int_\Lambda \log E_t(\lambda)\,\rho_t(\dd\lambda) - \KL(\rho_t\|\mu)$.
In particular, if $E_t(\lambda) = \exp(\lambda Y_t - \overline\Psi_t(\lambda))$,
\begin{equation}\label{eq:mixturePAC}
  \PP\!\left(
    \exists\,t :
    \int_\Lambda [\lambda Y_t - \overline\Psi_t(\lambda)]\rho_t(\dd\lambda)
    - \KL(\rho_t\|\mu) \geq x\right)
  \leq e^{-x}.
\end{equation}
\end{proposition}

\begin{proposition}[Exact form of the mixture tail]
\label{prop:mixture-exact}
In the setting of Proposition~\ref{prop:mixture}, fix $x > 0$ and an adapted posterior $(\rho_t)$ with $\rho_t \ll \mu$, and set
\[
  \cA_x := \Bigl\{\exists\,t : \int_\Lambda \log E_t(\lambda)\,\rho_t(\dd\lambda)
          - \KL(\rho_t\|\mu) \geq x\Bigr\},
  \qquad
  \cV_x := \{\exists\,t : \overline E_t \geq e^x\},
\]
so that $\cA_x$ is the event of~\eqref{eq:mixturePAC} in the exponential case $E_t(\lambda) = \exp(\lambda Y_t - \overline\Psi_t(\lambda))$.
Then $\cA_x \subseteq \cV_x$ and
\[
  \PP(\cA_x) = \PP(\cV_x) - \PP\bigl(\cV_x \setminus \cA_x\bigr),
  \qquad
  \cV_x\setminus \cA_x
  = \cV_x \cap \bigl\{\forall\,t:\ \KL(\rho_t\|\rho^*_t) > \log\overline E_t - x\bigr\},
\]
the second event being read at those $t$ with $\overline E_t>0$, the only ones at which either event can occur.
It vanishes when $\rho_t \equiv \rho^*_t$, at which $\cA_x = \cV_x$.
If $\overline E$ is a martingale with $\overline E_0 = 1$ and $\overline E_t \to 0$ a.s., then Theorem~\ref{thm:firstpassage-exact} at level $e^x$ evaluates $\PP(\cV_x)$ and
\begin{equation}\label{eq:mixture-exact-mart}
  \PP(\cA_x)
  = e^{-x}\Bigl(1 - \E\bigl[J\,\ind\{\sigma < \infty\}\bigr]\Bigr)
    - \PP\bigl(\cV_x\setminus \cA_x\bigr),
\end{equation}
with $\sigma := \inf\{t : \overline E_t \geq e^x\}$ and $J := \overline E_\sigma - e^x \geq 0$ the overshoot.
Under those hypotheses equality holds in~\eqref{eq:mixturePAC} exactly when $\PP(\cV_x\setminus \cA_x) = 0$ and $J = 0$ a.s.\ on $\{\sigma<\infty\}$.
For a mixture that is a strict supermartingale the first-passage term includes, beyond the overshoot, the predictable loss shed before the crossing and the mass that never crosses.
\end{proposition}

\begin{proposition}[Per-step form of the multi-prior coincidence divergence]
\label{prop:seq-PAC}
Let $\pi_1,\ldots,\pi_W$ be path priors on $\Omega_T$ (different models, time scales, or modalities), let $\rho$ be a posterior path measure, and let $\alpha \in \Delta([W])$.
Proposition~\ref{prop:multi-PAC} holds verbatim for these path measures, and its coincidence divergence resolves by time:
\begin{equation}\label{eq:seq-coin-decomp}
  \cC_\alpha(\pi_{1:W})
  = -\log \E_{Q^*}\!\Bigl[\,\prod_{t=1}^T z_t\,\Bigr]
  \;\le\; \sum_{t=1}^T \E_{Q^*}\!\left[
      \min_{\widetilde q_t}\sum_w \alpha_w\,\KL(\widetilde q_t\|\pi_{w,t})
    \right],
\end{equation}
where $\pi_{w,t}$ is the conditional of $\pi_w$ at time $t$,
$z_t(x_{1:t-1})=\int\prod_w\pi_{w,t}^{\alpha_w}$ is the local normalizer (so that the per-step minimum equals $-\log z_t$), and $Q^*=\prod_t\bigl(\prod_w\pi_{w,t}^{\alpha_w}/z_t\bigr)$ is the sequential geometric-mixture path measure.
The inequality is Jensen's, so it is tight exactly when $\prod_{t=1}^T z_t$ is $Q^*$-almost surely constant; it suffices that each $z_t$ be free of the history, which holds whenever every prior is a product measure over time, $\pi_w=\bigotimes_{t=1}^T \pi_{w,t}$ with $\pi_{w,t}$ not depending on $x_{1:t-1}$.
In that case $\cC_\alpha(\pi_{1:W})=\sum_{t=1}^T\cC_\alpha(\pi_{1:W,t})$.
\end{proposition}

% ======================================================================

\section{Pooling: the mixture and its decomposition}
\label{app:pooling-decomposition}
% ======================================================================

Section~\ref{sec:pooling-benefit} reads the pooling benefit off a geometric mixture of test martingales.
Collected here are the supermartingale property the mixture rests on, the Doob decomposition that resolves the mass it sheds, and the two readings of that mass which the section distinguishes but does not need side by side.

\begin{lemma}[Geometric mixture is a supermartingale]
\label{lem:geometric-mixture}
For $\alpha \in \Delta([W])$ (probability weights) the geometric mixture $M_t^{(\alpha)} := \prod_{i=1}^W \bigl(M_t^{(i)}\bigr)^{\alpha_i}$ is a nonnegative supermartingale with $M_0^{(\alpha)} = 1$ and $\E[M_t^{(\alpha)}] \leq 1$.
\end{lemma}

The pooling benefit accumulates that supermartingale's compensator.

\begin{proposition}[Pooling-benefit decomposition]
\label{prop:pooling-exact}
Let $M^{(\alpha)}=N-A$ be the Doob decomposition of the geometric-mixture supermartingale (Lemma~\ref{lem:geometric-mixture}) into a nonnegative martingale $N$ ($N_0=1$) and a predictable nondecreasing $A$ ($A_0=0$), so that $A_t-A_{t-1}=M^{(\alpha)}_{t-1}-\E[M^{(\alpha)}_t\mid\cF_{t-1}]\ge0$ is the one-step weighted arithmetic--geometric-mean gap and $\E[A_t]=1-Z_t(\alpha)=1-e^{-\cC_\alpha(t)}$ for $\cC_\alpha$ of~\eqref{eq:pathwise-Calpha}.
With $\tau^{(\alpha)}_x$ and $g_x(\alpha)=\E[M^{(\alpha)}_{\tau^{(\alpha)}_x}\ind\{\tau^{(\alpha)}_x<\infty\}]$ as in Corollary~\ref{cor:pooling-benefit}, the pooling boost obeys the identity
\begin{equation}\label{eq:pooling-exact}
  1-g_x(\alpha)
  =\underbrace{\E\!\left[A_{\tau^{(\alpha)}_x}\right]}_{\text{coincidence shed before }\tau^{(\alpha)}_x}
  +\underbrace{\E\!\left[M^{(\alpha)}_\infty\,\ind\{\tau^{(\alpha)}_x=\infty\}\right]}_{\text{mass never reaching }x},
  \qquad
  A_{\tau^{(\alpha)}_x}=\sum_{t=1}^{\tau^{(\alpha)}_x}\bigl(M^{(\alpha)}_{t-1}-\E[M^{(\alpha)}_t\mid\cF_{t-1}]\bigr).
\end{equation}
The boost is thus the multi-way coincidence divergence accumulated up to the first-passage time, plus the below-threshold leakage; its dominant term $\E[A_{\tau^{(\alpha)}_x}]$ is the running coincidence divergence $\cC_\alpha$ sampled at $\tau^{(\alpha)}_x$.
No deterministic-time closed form $1-g_x(\alpha)=f(\cC_\alpha(t))$ exists, because $\tau^{(\alpha)}_x$ is path-dependent; \eqref{eq:pooling-exact} is the exact object.
\end{proposition}

Two readings of the same disagreement at the crossing must be distinguished.
The \emph{mass form}
$\cC^{\mathrm{mass}}_\alpha(\tau^{(\alpha)}_x):=-\log g_x(\alpha)$ is the partition function read at the first-passage time; the \emph{crossing form}
$\cC^{\mathrm{cross}}_\alpha(x):=-\log\bigl(x\,\PP(\tau^{(\alpha)}_x<\infty)\bigr)$ is the realized tightening of the plain Ville bound.
Since $x\,\PP(\tau^{(\alpha)}_x<\infty)\le g_x(\alpha)$, the crossing form dominates the mass form, the difference being the overshoot at the crossing (Corollary~\ref{cor:pooling-crossing-exact}); the two coincide exactly when the mixture crosses $x$ without overshoot.
Both are read at the random crossing time $\tau^{(\alpha)}_x$, so neither equals the running coincidence divergence $\cC_\alpha(t)$ of~\eqref{eq:pathwise-Calpha}, which is indexed by a fixed horizon.
That divergence reaches them only through the compensator $\E[A_{\tau^{(\alpha)}_x}]$ it has accumulated by the crossing (Proposition~\ref{prop:pooling-exact}).

% ======================================================================

\section{The continuous-time crossing construction}
\label{app:curved-cts}
% ======================================================================

\begin{figure}[!ht]
\centering
\includegraphics[width=\textwidth]{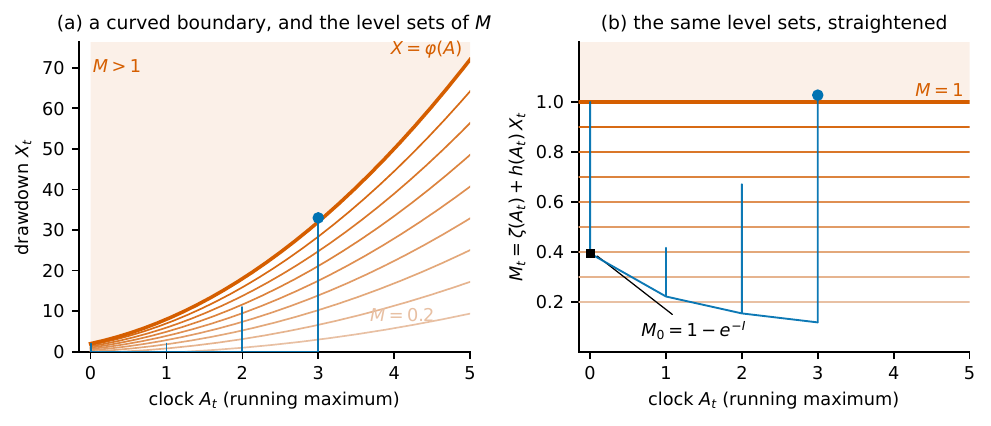}
\caption{\textbf{A boundary that moves in the natural coordinates is a flat level after the transform.}
\textbf{(a)} The drawdown $X_t$ of a symmetric $\pm1$ walk below its running maximum $A_t$, against the boundary $\varphi(A)=2(1+A)^2$ (heavy) and the level sets $M=0.2,\dots,0.9$ foliating the region beneath it; crossing is entry into the shaded region.
\textbf{(b)} The same level sets after the transform $M=\zeta(A)+h(A)X$, drawn at the same abscissae: each one is horizontal, and the boundary is the level $M=1$.
The clock advances only while the drawdown is at zero, so a path meets the boundary in one vertical probe per record.
The marked point is the path's first crossing, which lands on the boundary in (a) and on $M=1$ in (b), and the path starts at the level $M_0=1-e^{-I}$ the continuous-path idealization spends.
The path drawn is one that crosses late; at this boundary most crossings happen at the first record, where the boundary sits only two steps above it.}
\label{fig:curved-crossing}
\end{figure}

Work on $(\Omega,\cF,(\cF_t)_{t\ge0},\PP)$ under the usual conditions (i.e., completeness and right-continuity).
A nonnegative local submartingale $X$ with decomposition $X=N+A$ is of \emph{class $(\Sigma)$} when
\begin{itemize}[leftmargin=2em]
\item $N$ is a c\`adl\`ag local martingale with $N_0=0$;
\item $A$ is continuous, adapted and nondecreasing with $A_0=0$;
\item the measure $\dd A$ is carried by the zero set $\{t:X_t=0\}$.
\end{itemize}
When $N$ is continuous as well the class is written $(\Sigma_c)$.
In particular $X_0=0$, and $A$ is the \emph{increasing process} of $X$.
The two letters follow the Doob decomposition of Proposition~\ref{prop:pooling-exact}, with the sign of $A$ reversed because $X$ is a submartingale.

\begin{lemma}[Function-indexed martingales of a class-$(\Sigma)$
process~\cite{Nikeghbali2006essay}]
\label{lem:sigma-family}
Let $X=N+A$ be of class $(\Sigma)$ and let $f\colon\R_+\to\R$ be Borel and locally bounded.  Then
\begin{equation}\label{eq:sigma-family}
  \int_0^{A_t} f(z)\,\dd z-f(A_t)\,X_t
  \;=\;-\int_0^t f(A_u)\,\dd N_u
\end{equation}
so the left-hand side is a local martingale null at $0$.
Conversely, let $X$ be a nonnegative local submartingale with $X_0=0$ and let $C$ be a continuous adapted nondecreasing process with $C_0=0$ such that $\int_0^{C_t}f(z)\,\dd z-f(C_t)X_t$ is a local martingale for every locally bounded Borel $f$.
Then $X$ is of class $(\Sigma)$ and $C=A$.
\end{lemma}

One class-$(\Sigma)$ process therefore yields one local martingale per locally bounded Borel $f$, every member driven by the same $\dd N$ through the integrand $f(A_u)$ and every member reading the path only through the pair $(A_t,X_t)$.
The converse says the family characterizes the class, and identifies the process against which the functions are composed.
Choosing $f$ so that the member is nonnegative and normalizing by its initial value produces a test martingale; the crossing law below is the member picked out by one particular choice.

For a boundary read against the increasing process, that member gives the crossing probability in closed form.

\begin{theorem}[Crossing law for a curved boundary~\cite{Nikeghbali2006essay}]
\label{thm:curved-crossing}
Let $X=N+A$ be of class $(\Sigma)$ with only negative jumps and $\lim_{t\to\infty}A_t=\infty$ a.s., and let $\varphi\colon\R_+\to(0,\infty]$ be Borel, with the convention $1/\infty:=0$.  Then
\begin{equation}\label{eq:curved-crossing}
  \PP\bigl(\exists\,t\ge0:\ X_t>\varphi(A_t)\bigr)
  =1-\exp\Bigl(-\int_0^\infty \varphi(x)^{-1}\dd x\Bigr)
\end{equation}
and for every $u>0$,
\begin{equation}\label{eq:curved-crossing-u}
  \PP\bigl(\exists\,t\ge0:\ A_t<u,\ X_t>\varphi(A_t)\bigr)
  =1-\exp\Bigl(-\int_0^u \varphi(x)^{-1}\dd x\Bigr)
\end{equation}
each right-hand side being read as $1$ when its integral diverges.
\end{theorem}

Equation~\eqref{eq:curved-crossing-u} is indexed by the level the increasing process has reached; no time enters it.
Writing $A^{-1}_u:=\inf\{t\ge0:A_t>u\}$ for the right-continuous inverse of $A$, the event there is $\{\exists\,t\le A^{-1}_u:X_t>\varphi(A_t)\}$ whenever $A_t<u$ for every $t<A^{-1}_u$, that is, whenever $A$ does not pause at the level $u$ before passing it.
The levels at which a continuous nondecreasing process pauses form a countable set, so this holds for Lebesgue-almost every $u$, and at every fixed $u$ in the Brownian instance below.

Table~\ref{tab:classical-ledger} reads each classical statement as an identity with a named residual removed.
Theorem~\ref{thm:curved-crossing} is the degenerate case of that reading: the residual is zero for every Borel boundary.
The criterion is the one Proposition~\ref{prop:maximal-conditional}(i) supplies.
The crossing event is $\{\sup_t M_t>1\}$ for a nonnegative local martingale $M$ assembled from $X$ and $\varphi$ (Appendix~\ref{app:proofs}); $X$ has no positive jumps and $A$ is continuous, so $M$ has no positive jumps either, its running supremum is continuous, and the level $1$ is met without overshoot.
Since the level is met without overshoot, that maximal identity evaluates the crossing probability as $M_0$, the whole of the bound $M_0/x$ at $x=1$.
The boundary is arbitrary and the increasing process is the one $X$ supplies, so the closed form is a property of the class and holds for every $\varphi$.

Brownian motion supplies the worked instance.
For a standard Brownian motion $B$ with running supremum $B^\ast_t:=\sup_{u\le t}B_u$, put $X_t:=B^\ast_t-B_t$ and $A_t:=B^\ast_t$.
Then $X=(-B)+B^\ast$ is of class $(\Sigma_c)$: the process $-B$ is a continuous local martingale null at $0$,
$B^\ast$ is continuous, adapted and nondecreasing with $B^\ast_0=0$, and $\dd B^\ast$ is carried by $\{B^\ast=B\}=\{X=0\}$.
The paths are continuous and $B^\ast_\infty=\infty$ a.s., and $A^{-1}_u$ is the first time $B$ reaches $u$.
Theorem~\ref{thm:curved-crossing} therefore gives, in complementary form,
\begin{equation}\label{eq:knight}
  \PP\bigl(\forall\,t\ge0:\ B^\ast_t-B_t\le\varphi(B^\ast_t)\bigr)
  =\exp\Bigl(-\int_0^\infty\varphi(x)^{-1}\dd x\Bigr)
\end{equation}
and, with the horizon at the first passage of $B$ to level $x$,
\[
  \PP\bigl(\forall\,t\le A^{-1}_x:\ B^\ast_t-B_t\le\varphi(B^\ast_t)\bigr)
  =\exp\Bigl(-\int_0^x\varphi(y)^{-1}\dd y\Bigr)
\]
The drawdown of Brownian motion below its own running maximum stays under an arbitrary curved envelope with a probability determined by $\int\varphi^{-1}$ alone.
Knight obtained~\eqref{eq:knight} by excursion theory; the route through Lemma~\ref{lem:sigma-family} recovers it as one member of the function-indexed family~\cite{Nikeghbali2006essay}.

% ======================================================================
\section{Random times in continuous time}\label{app:random-time-ct}
% ======================================================================

\subsection{The continuous-time case}
\label{sec:continuous-time}

The peeking identity carries over verbatim to continuous time.
The random-time calculus is ultimately a change of filtration.
Reading $\widehat\Omega = \Omega \times \N$ as an enlargement of the base filtration and $\Amart{\tau}_t = \dd R/\dd\widehat P$ (Lemma~\ref{lem:Rt}) as the density linking the two, the anticipation martingale plays the part of the exponential density in a Girsanov change of measure.
The hazard clock $1 - \clock{\tau}$ is the finite-variation compensator the base filtration already predicts, and $\Amart{\tau}$ is the martingale residual that only the enlarged filtration sees.
The mass identity $P(\tau = t \mid \cF_t) = \Amart{\tau}_t\,\Delta \clock{\tau}_t$ that measures peeking in discrete time is the discrete image of the multiplicative decomposition of the \emph{Az\'ema supermartingale}~\cite{NY05,nikeghbali2006doobs,coculescu2012hazard},
which is the conditional survival $\PP(\tau > t \mid \cF_t)$ of~\eqref{eq:surv-haz} read in continuous time.
The same $\log \Amart{\tau}_t$ term measures anticipation there once that process, written $Z_t^\tau$, is read against the finite-variation hazard clock of the decomposition, with $\widehat P$ its clock measure and $\log
\Amart{\tau}_t$ its local-martingale density.  The path-time identity then carries over
with no quadratic-variation input.
A random time is \emph{quasi-left-continuous} when it lands with positive probability on no \emph{predictable} time---one the filtration can announce before it arrives.
That property makes its hazard clock continuous and puts the survival factor in the exponential-hazard form used below.
The filtration is assumed throughout to satisfy the usual conditions: it is right-continuous and contains the $P$-null sets.

\begin{theorem}[Continuous-time path-time identity]
\label{thm:pathtime-ct}
Let $\tau$ be a finite, quasi-left-continuous random time on $(\Omega,\cF,(\cF_t)_{t\ge0},P)$ under the usual conditions, with Az\'ema supermartingale $Z_t:=P(\tau>t\mid\cF_t)$ and multiplicative decomposition $Z_t=\Amart{\tau}_t(1-\clock{\tau}_t)$ on $\{Z_{-}>0\}$, where $1-\clock{\tau}$ is the predictable decreasing part of that decomposition (the hazard clock, $\clock{\tau}_0=0$) and $\Amart{\tau}$ the anticipation local martingale ($\Amart{\tau}_0=1$).
Quasi-left-continuity makes the hazard clock continuous, $1-\clock{\tau}_t=e^{-\Lambda_t}$ with $\Lambda$ the continuous predictable hazard, and licenses the disintegration below: a time charging a predictable instant separates the predictable hazard clock from the adapted one, and the density statement fails there.
Let $\widehat P(A\times \dd t):=\E[\ind_A\,\dd \clock{\tau}_t]$ be the clock measure on $\Omega\times\R_+$, so $\dd R/\dd\widehat P=\Amart{\tau}_t$ for $R(A\times \dd t):=P(A\cap\{\tau\in \dd t\})$ and $\E[V_\tau]=\E_{\widehat P}[V_t \Amart{\tau}_t]$ for adapted $V$.
Then for strictly positive adapted factors $\Pi_i$ with $0<Z_\tau(\alpha):=\E[\prod_i\Pi_{i,\tau}^{\alpha_i}]<\infty$ and every $Q\ll\widehat P$ with $Q(\{\Amart{\tau}_t\prod_i\Pi_{i,t}^{\alpha_i}=0\})=0$,
\begin{equation}\label{eq:pathtime-ct}
  \log Z_\tau(\alpha)
  =\textstyle\sum_i\alpha_i\E_Q[\log\Pi_{i,t}]+\E_Q[\log \Amart{\tau}_t]-\KL(Q\|\widehat P)+\KL(Q\|Q^\alpha_\tau),
\end{equation}
with optimizer $\dd Q^\alpha_\tau/\dd\widehat P=\Amart{\tau}_t\prod_i\Pi_{i,t}^{\alpha_i}/Z_\tau(\alpha)$;
dropping $\KL(Q\|Q^\alpha_\tau)\ge0$ recovers the supremum form.
\end{theorem}
Equation~\eqref{eq:pathtime-ct} is the discrete identity (Theorem~\ref{thm:pathtime}) read on the continuum clock measure, the variational step being measure-space-agnostic.
The per-step sums integrate against the \emph{predictable finite-variation} hazard $\Lambda$, and the predictable quadratic variation does not enter; the peeking term stays $\E_Q[\log \Amart{\tau}_t]$, and a quadratic-variation term appears only at second order in $\log \Amart{\tau}_t=\int_0^t (\Amart{\tau}_{s-})^{-1}\dd \Amart{\tau}_s-\tfrac12\int_0^t (\Amart{\tau}_{s-})^{-2}\dd[\Amart{\tau}]^c_s
+\sum_{s\le t}\big(\log\tfrac{\Amart{\tau}_s}{\Amart{\tau}_{s-}}-\tfrac{\Delta \Amart{\tau}_s}{\Amart{\tau}_{s-}}\big)$.

The anticipation martingale is the change-of-measure counterpart of the classical reduction of a random time to a randomized stopping time, which incurs no distributional loss for optional processes observed up to the time~\cite{kardaras2015stochastic}; that reduction extends the peeking calculus to continuous time.

\subsection{Moment inequalities at a random time}
\label{sec:moment-random-time}

The peeking penalty measures what a random time does to an e-process evaluated there.
The Burkholder--Davis--Gundy comparison of a martingale's maximal function against its square function meets the same question, agrees with the peeking penalty at the two extremes, and supplies a closed-form inflation factor between them.
The imported statements below are continuous-time where their hypotheses say so,
and they settle the random-time behavior of the Burkholder--Davis--Gundy row of Table~\ref{tab:classical-ledger}.

A random time $\tau$ is \emph{honest} for $(\cF_t)$ when for every $t$ it agrees on $\{\tau<t\}$ with an $\cF_t$-measurable random variable.
Equivalently, within the multiplicative-system description, a random time is honest precisely when it is $\cF_\infty$-measurable and its conditional-survival field is a multiplicative cocycle~\cite{li2012multiplicative}.
It \emph{avoids} $(\cF_t)$ stopping times when $\PP(\tau=\sigma<\infty)=0$ for every $(\cF_t)$ stopping time $\sigma$.
Conditions~(A) and~(C) of~\cite{Nikeghbali2006essay} are avoidance and continuity of every $(\cF_t)$ martingale, and~(CA) is both; a Brownian filtration satisfies~(C).  Avoidance is stronger than quasi-left-continuity: every predictable time is a stopping time, so a time satisfying~(A) is in particular quasi-left-continuous, and a finite such time falls within Theorem~\ref{thm:pathtime-ct}'s hypotheses.  Throughout,
$Z_t:=\PP(\tau>t\mid\cF_t)$ is the Az\'ema supermartingale of Theorem~\ref{thm:pathtime-ct}, whose multiplicative decomposition $Z_t=\Amart{\tau}_t(1-\clock{\tau}_t)$ supplies the anticipation martingale this paper measures.

An unrestricted random time admits no comparison at all, by a two-point argument that needs nothing beyond the definitions.

\begin{proposition}[No maximal-to-square-function constant at a general random
time]
\label{prop:no-bdg-general}
Let $(M_t)_{t\ge0}$ be a martingale with $M_0=0$, increments $d_t=M_t-M_{t-1}$ and predictable variance process $\langle M\rangle_t=\sum_{s\le t}\E[d_s^2\mid\cF_{s-1}]$.
Suppose $C<\infty$ satisfies $\E[|M_\tau|]\le C\,\E\bigl[\langle M\rangle_\tau^{1/2}\bigr]$ for every $\{0,1\}$-valued $\cF$-measurable random time $\tau$.
Then $|M_1|\le C\,\langle M\rangle_1^{1/2}$ almost surely.
\end{proposition}

The conclusion is unattainable as soon as the first increment is unbounded at fixed conditional variance: for $\cF_0$ trivial and $d_1$ standard Gaussian,
$\langle M\rangle_1=1$ and no finite $C$ bounds $|M_1|$.
The maximal form falls with the terminal one, since $|M_\tau|\le\sup_{s\le\tau}|M_s|$.
Padding $M$ with a null first step makes the offending $\tau$ positive-integer-valued,
so the failure does not depend on admitting the value $0$.
The continuous-time original takes $\tau=\ind_A$ against Brownian motion, where $\E[|B_\tau|]\le C\,\E[\sqrt\tau\,]$ collapses to $\E[|B_1|\ind_A]\le C\,\PP(A)$ for every $A$ and forces $|B_1|$ to be bounded~\cite{Nikeghbali2006essay}.

A pseudo-stopping time changes nothing.

\begin{theorem}[Burkholder--Davis--Gundy at a pseudo-stopping
time~\cite{Nikeghbali2006essay}]
\label{thm:bdg-pseudo}
Let $(\cF_t)_{t\ge0}$ satisfy the usual conditions, let $p>0$, let $(M_t)_{t\ge0}$ be a continuous $(\cF_t)$ local martingale with $M_0=0$ (so that $\langle M\rangle=[M]$), and let $\tau$ be an $(\cF_t)$ pseudo-stopping time.  Then
\begin{equation}\label{eq:bdg-pseudo}
  c_p\,\E\bigl[\langle M\rangle_\tau^{p/2}\bigr]
  \;\le\; \E\bigl[\bigl({\textstyle\sup_{s\le\tau}}|M_s|\bigr)^{p}\bigr]
  \;\le\; C_p\,\E\bigl[\langle M\rangle_\tau^{p/2}\bigr]
\end{equation}
with $c_p$ and $C_p$ the constants of the Burkholder--Davis--Gundy inequalities at stopping times, depending on $p$ alone.
\end{theorem}

Transport preserves the constants themselves: their dependence on $p$, and their asymptotics as $p$ grows, are those of the classical inequalities~\cite{Nikeghbali2006essay}.
A pseudo-stopping time is therefore free in the power-mean geometry exactly as it is free in the entropic one,
where the peeking penalty vanishes (Proposition~\ref{prop:vanish}(b)).

An honest time is not free.  Under conditions~(CA) the unmodified inequalities fail at honest times~\cite{Nikeghbali2006essay}, and a corrected comparison with one extra factor takes their place.

\begin{proposition}[Corrected maximal inequality at an honest
time~\cite{Nikeghbali2006essay}]
\label{prop:bdg-honest}
Let $(\cF_t)_{t\ge0}$ be the filtration of a standard Brownian motion $(B_t)_{t\ge0}$ and let $\tau$ be an honest time for it, with Az\'ema supermartingale $Z_t$, running infimum $I_\tau:=\inf_{u<\tau}Z_u$, and $\Upsilon_\tau:=\Bigl(1+\log\tfrac{1}{I_\tau}\Bigr)^{1/2}$.
Then $\E[|B_\tau|]\le C\,\E\bigl[\Upsilon_\tau\sqrt\tau\,\bigr]$ for a universal constant $C$.
\end{proposition}

The inflation factor has an exact law once the time also avoids stopping times.

\begin{corollary}[Law of the honest-time inflation factor]
\label{cor:honest-inflation-law}
In the setting of Proposition~\ref{prop:bdg-honest}, suppose in addition that $\tau$ avoids every $(\cF_t)$ stopping time.
Then $\log(1/I_\tau)$ is standard exponential, $\Upsilon_\tau$ has the law of $(1+\varepsilon)^{1/2}$ with $\varepsilon$ standard exponential, and $\E[\Upsilon_\tau^2]=2$.
\end{corollary}

At a stopping time $Z_u=\ind\{\tau>u\}$, so $I_\tau=1$ and $\Upsilon_\tau=1$,
and the corrected comparison reduces to the classical one.
Without avoidance the exact law weakens to a stochastic ordering,
$\E[f(\Upsilon_\tau)]\le\E[f((1+\varepsilon)^{1/2})]$ for every continuous increasing $f:\R_+\to\R_+$~\cite{Nikeghbali2006essay}, so $(1+\varepsilon)^{1/2}$ dominates the worst honest time.

The correction is denominated in the units the rest of the paper counts in.
The quantity $\log(1/I_\tau)$ is the \emph{surprisal} --- the negative logarithm of a probability, in nats --- of the lowest level the time's Az\'ema supermartingale reaches before $\tau$, and Corollary~\ref{cor:honest-inflation-law} fixes its mean at one nat; the moment inequality is inflated by the square root of one plus that surprisal.
The classes of Section~\ref{sec:survival} are ordered by what they cost, and that ordering has a counterpart in the power-mean geometry.
Stopping and pseudo-stopping times are free in both:
the peeking penalty is at most zero at stopping times and exactly zero at pseudo-stopping times for uniformly integrable martingales (Proposition~\ref{prop:vanish}), and the Burkholder--Davis--Gundy constants are untouched (Theorem~\ref{thm:bdg-pseudo}).
An unrestricted time admits no finite quantity in either: the peeking penalty at the ultimate-maximum time is infinite (Proposition~\ref{prop:infinite}), and no maximal-to-square-function constant exists (Proposition~\ref{prop:no-bdg-general}).
At the classes between them both supply one: the power-mean geometry an inflation factor with a closed-form law,
the entropic one the exact complexity $\log\beta$ of Corollary~\ref{cor:peeking-ball}.

The two ladders run over different functionals that have the same law.
The anticipation budget $\log \Amart{\tau}_\tau$ measures the time against its own hazard clock, while $\log(1/I_\tau)$ is read off the running infimum of its Az\'ema supermartingale, and neither determines the other pathwise.
In the setting of Corollary~\ref{cor:honest-inflation-law} the time is the ultimate-maximum time of a continuous vanishing local martingale $N$ with $N_0=1$, whose Az\'ema supermartingale is $Z_t=N_t/N^\ast_t$~\cite{nikeghbali2006doobs}.
Since $1/N^\ast$ is continuous and decreasing, that product is the multiplicative decomposition of $Z$, so $\Amart{\tau}=N$ and $1-\clock{\tau}=1/N^\ast$; at the ultimate maximum $\Amart{\tau}_\tau=N^\ast_\infty$, and Proposition~\ref{prop:maximal-conditional}(i) makes $\log \Amart{\tau}_\tau$ standard exponential.
The surprisal that inflates the moment inequality has the same law (Corollary~\ref{cor:honest-inflation-law}), so at the honest-time class the two ladders are calibrated to one another: each stands one nat above its reference in expectation.

The interpolating budget of \S\ref{sec:peeking-complexity} is a separate object: this measures one inequality on one class of times.
The three statements above settle the second axis of that row, the cost of moving from a stopping time to a random time.

% ======================================================================
\section{Empirical evaluation supplement}\label{app:eval-supplement}
% ======================================================================

The evaluations below measure what the theory does not state: how large a named residual is, what governs it, and where a natural reading of it fails.
Each has a negative or a scope demarcation of its own, and each is introduced here, so that a reader arriving at a section already knows its subject.
The identities themselves are established by the theorems and illustrated in Figure~\ref{fig:decompositions}, which reads off four decompositions, two measured on moderate-scale data and two computed exactly; three resolve into relative entropies and Doob's into a H\"older deficit, an optional-stopping deficit and an initial value.
Nothing below re-derives them.

Section~\ref{subsec:eval-maximal-qq} puts the first-passage identity on a deployed e-process, where the wealth jumps and Ville's inequality is loose in consequence.
Dividing each prediction by the crossing rate observed turns that looseness into a multiplicative factor: the bound overstates by as much as $2.4\times$, and the overshoot term recovers almost all of it.
The levels where it does not are the sparsest crossings, where the observed rate is itself the noisy quantity, and the conserved mass is not estimable at any feasible replicate count, so it is not used as a check.

Section~\ref{subsec:eval-curved-crossing} checks the discrete curved-boundary law against a closed form built from gambler's ruin, which never mentions the martingale that proves it.
The two agree exactly, and the event identity holds on every one of $1.25\times10^{6}$ simulated paths.
The asymptotic level is out of reach of any attainable horizon; the shortfall is not an error term but the mass of the martingale left on paths that have not yet crossed, and it is measured directly.

Section~\ref{subsec:eval-overshoot-convergence} asks which parameter closes the first-passage deficit.
The horizon does not: the ratio is unmoved across a factor-of-$40$ sweep.
The increment scale does, the deficit falling as a power law in $\sqrt{dt}$ on both a Gaussian and a lattice increment law, which is the discrete-time tail approaching Pareto$(1)$ from below in the parameter that governs it.
The two families are reported separately, the lattice one reading a shallower slope that a single power law describes less well.

Section~\ref{sec:eval-pacbayes-cifar-fixed-prior} takes the anytime-valid certificate to deep-net scale.
It is vacuous under a prior anchored at the origin, at $41.7$ against a held-out risk of $0.078$, and tight under one anchored at the trained weights, at $0.153$; the $273\times$ tightening sits entirely in the anchor and not in the optimizer.

Section~\ref{sec:eval-pooling-benefit-decorrelated} separates the pooling benefit from optional-stopping leakage with a paired single-bettor arm on trained vision models.
The excess deficit is positive in every multi-member cell, and at matched accuracy the architecturally diverse pool sheds more mass than the recipe-diverse one.
The corruption is a single family at one severity, and the ensembles are public checkpoints not trained for this purpose, so public training recipes bound the disagreement on offer.

% First-passage identity on a real discrete-time e-process

\subsection{The first-passage identity on a real e-process}
\label{subsec:eval-maximal-qq}

Theorem~\ref{thm:ville-tight}'s Pareto$(1)$ law is exact for continuous paths, and a deployed e-process jumps.
Theorem~\ref{thm:firstpassage-exact} says what the jump costs: in discrete time the first-passage probability is $\PP(\sup_t M_t \geq x) = \tfrac1x(1 - \E[J_x \ind\{\sigma_x<\infty\}])$ with $J_x$ the overshoot at the crossing, an equality where Ville's inequality gives only $1/x$.

The substrate is the sequential likelihood ratio of a held-out classifier against the label-marginal null, on WDBC breast cancer ($T=285$) and Pima diabetes ($T=384$), with $60\,000$ null replicates per stream.
Labels are redrawn independently from the measured marginal, which makes each factor exactly mean-one and the wealth an exact martingale; drawing them by permutation instead would fix the single-draw marginal but leave the draws dependent, and a product of dependent mean-one factors does not have mean one.
The bet is left uncapped, since Theorem~\ref{thm:firstpassage-exact} is stated for a martingale.
Anytime validity holds on both streams: $1/M^\ast$ is a conservative $p$-value, and its largest excursion below the uniform is $2\times10^{-5}$ on WDBC and $3\times10^{-5}$ on Pima.

A large multiplicative factor separates the bound from the identity.
Figure~\ref{fig:maximal-qq} divides each prediction by the crossing rate actually observed, so $1$ is exact and the height above it is the shortfall.
Ville's $1/x$ overstates the crossing rate by as much as $2.4\times$, and the overshoot term recovers almost all of it: read forward from the measured overshoot, the identity lands inside the observed binomial interval at $41$ of the $44$ levels, where the bound lands inside it at $1$.
Taken against the rate observed, the median relative gap the identity leaves is $0.009$ on WDBC and $0.028$ on Pima, against $0.402$ and $0.243$ for the bound at the same levels.

One check is out of scope.
The optional-stopping mass $\E[M_{\sigma_x}\ind\{\sigma_x<\infty\}]$ equals $1$ exactly when the identity holds, but it is a sample mean of a wealth that tends to zero almost surely while keeping mean one, so it is not estimable at any feasible replicate count and is not used as a check here.

The resolution of the comparison itself is set the same way.
In the ratio the figure plots, the interval on the observed rate has width $2z\sqrt{(1-p)/n_x}$ at a level crossed $n_x$ times, so it depends on the level only through that count;
the deepest level plotted is the one $0.1\%$ of replicates reach, whatever process is simulated.
Only more replicates narrow it; a different process does not.
Past a point they stop helping too.
The identity's departure from the observed rate is exactly the conserved mass's departure from one, divided by $x$ times the observed crossing rate, so a tighter interval resolves that estimate's error before it resolves the identity.

\begin{figure}[tbp]
\centering
\includegraphics[width=\textwidth]{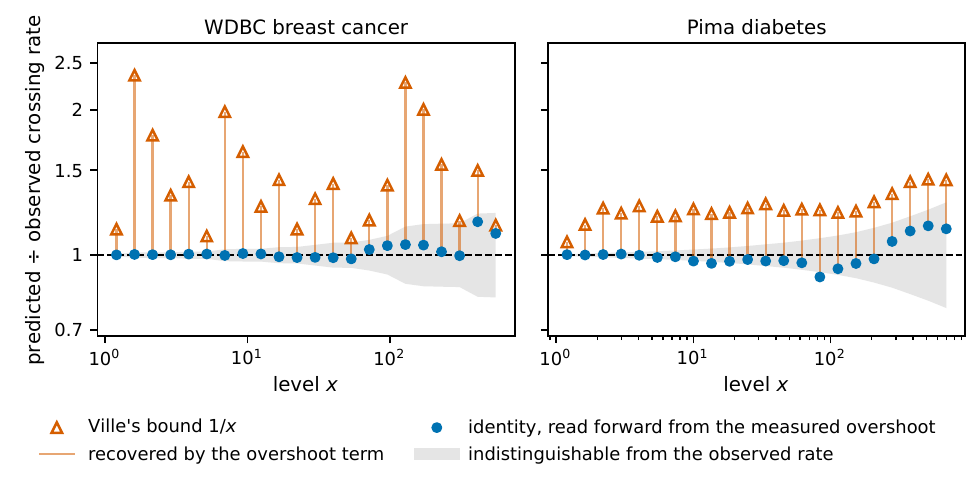}
\caption{\textbf{The overshoot term recovers almost all of what Ville's bound gives away.}  Each prediction of the crossing rate, divided by the rate observed on $60\,000$ null replicates of a sequential likelihood ratio against the label-marginal null; $1$ is exact, and height above it is the multiplicative shortfall.  At every level the two predictions are joined: the triangle is Ville's $1/x$, which overstates by up to $2.4\times$, the circle is the identity read forward from the measured overshoot, and the connector between them is the part the overshoot term recovers.  The identity sits on $1$ at $41$ of the $44$ levels, inside the shaded interval on the observed rate; the bound sits outside that interval at $43$ of them.  One panel per stream, since the two level grids share no positions.  The interval widens with the level because its width in this ratio is set by the number of crossings alone, and the deepest level plotted is the one $0.1\%$ of replicates reach.}
\label{fig:maximal-qq}
\end{figure}

% The discrete curved-boundary crossing law against an independent closed form,
% with the shortfall a finite horizon leaves.
%

\subsection{The discrete crossing law at a finite horizon}
\label{subsec:eval-curved-crossing}

Theorem~\ref{thm:curved-crossing-discrete} is checked on the instance Section~\ref{sec:curved-crossing} quotes, by an independent route that does not use the martingale the theorem is proved with.
The object throughout is the drawdown of a symmetric $\pm1$ walk against the boundary $\varphi(x)=2(1+x)^2$; there is no dataset, and every rate below is either a closed form or a frequency over simulated walks.
On this walk the record advances by exactly one, so the residual sum is a deterministic function of the record reached.
Within an excursion the record is fixed, so the boundary is constant there and the excursion maximum obeys the gambler's-ruin law.
Both sides of~\eqref{eq:discrete-crossing-law} are then convergent series over record levels, the left-hand side assembled from gambler's ruin alone.

The two routes agree exactly.
At $\varphi(x)=2(1+x)^2$ the nominal level is $0.3935$, the quadrature residual contributes $0.1343$ and the overshoot $0.08571$, placing the right-hand side at $0.4421$, which is the crossing probability the excursion decomposition returns.
Truncating the series stops moving the level at six figures.
The realized level exceeds the nominal by $12.35\%$.

Equation~\eqref{eq:discrete-crossing-law} equates two events, so it can be tested on each path.
Across $1.25\times10^{6}$ simulated paths at five horizons, $\{\exists\,t: X_t>\varphi(A_t)\}$ and $\{\sup_t M_t>1\}$ agreed on every one, and the pathwise decomposition~\eqref{eq:discrete-crossing-decomp} reproduced $M_t$ at every time step.

A finite horizon leaves a substantial shortfall, and the identity says exactly how much.
The law asks for $A_t\to\infty$, and the running maximum of a $\pm1$ walk grows like $\sqrt{T}$, so no reachable horizon attains it: at $1.25\times10^{5}$ steps the realized rate is still $0.029$ short.
The shortfall is not an error term but a quantity that can be measured directly, since optional stopping at $T\wedge T_{\max}$ closes exactly with one further term $\E[M_{T_{\max}}\ind\{\text{no cross}\}]$, the mass of $M$ left on paths that have not yet crossed.
Figure~\ref{fig:curved-crossing-horizon} adds that mass back.
Across a $625$-fold range of horizons the identity's finite-horizon prediction tracks the simulated rate, and the residual separating them --- a mean-zero quantity, being a stopped martingale sum --- covers zero at every horizon.

\begin{figure}[tbp]
\centering
\includegraphics[width=\textwidth]{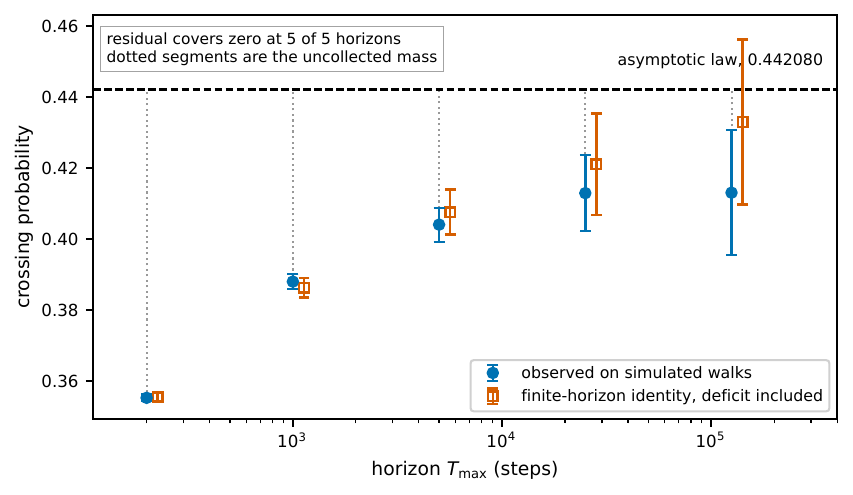}
\caption{\textbf{A finite horizon falls short of the asymptotic crossing rate by exactly the mass still sitting on paths that have not crossed.}  Simulated drawdowns of a symmetric $\pm1$ walk against $\varphi(x)=2(1+x)^2$, at five horizons from $200$ to $1.25\times10^{5}$ steps.  Circles are the crossing rate observed on those walks; squares give the finite-horizon identity's prediction once the uncrossed mass is included, offset slightly in the horizon for legibility.  The dotted segments are that uncollected mass, closing as the horizon grows.  Intervals are $95\%$: on the observed rate its own binomial interval, and on the prediction the interval of the residual, which is the uncertainty the comparison actually incurs.  The residual is mean-zero and covers zero at every horizon.}
\label{fig:curved-crossing-horizon}
\end{figure}

The size of the correction depends on the boundary.
Take the family $\varphi(x)=\varphi_0\bigl(1+x/(\varphi_0 I)\bigr)^2$, which has $\int_0^\infty\varphi^{-1}=I$ for every $\varphi_0$ and contains the instance above at $\varphi_0=2$.
At $I=\tfrac12$ the same closed form puts the excess over nominal at $12.35\%,\ 7.66\%,\ 4.90\%,\ 2.29\%$ and $1.31\%$ as $\varphi(0)$ runs over $2,\ 4,\ 8,\ 20$ and $32$, each satisfying the law.
Both corrections contain the factor $h=w/\varphi$, so a boundary set only a few steps above the record is misread most.
The instance above is the tightest member of the family, and it places $0.566$ of its crossings at record $0$, before the clock has moved at all --- where the boundary's shape does least work and the discrete correction is largest.

% Overshoot deficit in the discrete-time first-passage identity

\subsection{The first-passage deficit is the overshoot}
\label{subsec:eval-overshoot-convergence}

Theorem~\ref{thm:firstpassage-exact} makes the first-passage ratio $R(x) := x\,\PP(\sup_t M_t \geq x)$ equal to $1 - \E[J_x \ind\{\sigma_x<\infty\}]$,
so it reaches the continuous-path Pareto$(1)$ value of Theorem~\ref{thm:ville-tight} exactly when the crossing has no overshoot.
Two parameters could close the gap: the horizon, if the running maximum were merely truncated, or the step, which sets the increment scale and with it the overshoot.
This evaluation separates them on two exponential martingales with $M_0=1$ and $M_t\to0$ a.s.---a Dol\'eans exponential with Gaussian increments, and a two-point likelihood ratio whose increments are lattice-valued in $\log M$---at $\theta\in\{0.3,0.5,0.8\}$ and $x\in\{2,5,10,20\}$.

The horizon leaves the deficit untouched.
The horizon sweep runs $T$ over $\{500,1000,2000,4000,8000,20000\}$ at unit step on the Gaussian family, read off nested prefixes of the same $300\,000$ paths so the comparison is within-path.
Across that factor of $40$, $R(5)$ moves by $10^{-4}$ at $\theta=0.3$ and not at all in the fourth decimal at $\theta=0.5$ and $\theta=0.8$; the ratio sits at $0.8416$, $0.7504$ and $0.6326$ throughout.

Refining the step closes it, at the rate the increment scale sets.
On the Gaussian family the deficit at $x=5$ falls from $0.16$--$0.37$ at unit step to $0.0055$--$0.0148$ at $dt=1/1024$, and the fitted log-log slope against $\sqrt{dt}$ is $0.976$, $0.963$ and $0.936$ at the three tilts, with $R^2 \geq 0.999$ on six steps.
The ratio clears $0.95$ at every one of the $24$ (family, tilt, level) cells at the finest step, spanning $0.9525$ to $1.031$; Theorem~\ref{thm:ville-tight} caps the ratio at $1$, so the cells above unity are sampling error.
This is the approach from below that Theorem~\ref{thm:ville-tight} describes, in the parameter that governs it.

The two-point family behaves differently, and the difference is the one Theorem~\ref{thm:firstpassage-exact} anticipates when it records that the overshoot law depends on the increment distribution.
Its deficit also vanishes monotonically in the step, but a single power law describes it less well and reads a shallower slope: $0.921$, $0.901$ and $0.862$ at the three tilts, with $R^2$ falling to $0.966$ in the $\theta=0.8$ cell against $0.999$ or better on the Gaussian side.
The two families are therefore reported separately and not pooled.

The conserved mass of~\eqref{eq:fp-mass} is the substantive check on the identity itself, $\sigma_x$ being unbounded and $M$ not uniformly integrable.
Across the $24$ cells at unit step, at $400\,000$ paths each, the estimate $\widehat\E[M_{\sigma_x}\ind]$ has median absolute deviation $0.0024$ from $1$ and worst $0.0096$; its bias-corrected and accelerated bootstrap $95\%$ interval covers $1$ in $23$.
The single exclusion (two-point, $\theta=0.5$, $x=2$) reads $1.004$ with interval $[1.001,1.008]$, a departure of four parts in a thousand on a heavy-tailed mean whose interval under-covers at these path counts.
The quantity itself does not sit away from $1$.
Given a sample the two sides of~\eqref{eq:fp-exact} differ by exactly $\widehat\E[M_{\sigma_x}\ind] - 1$, so the split of the ratio into $x\widehat\PP(\sigma_x<\infty)$ and $1 - \widehat\E[J_x\ind]$ is bookkeeping and introduces no second estimate.

Figure~\ref{fig:overshoot-convergence} separates the two candidate parameters.
The ratio is unmoved by the horizon, and step refinement drives it to the continuous-path limit.

\begin{figure}[h]
\centering
\includegraphics[width=\textwidth]{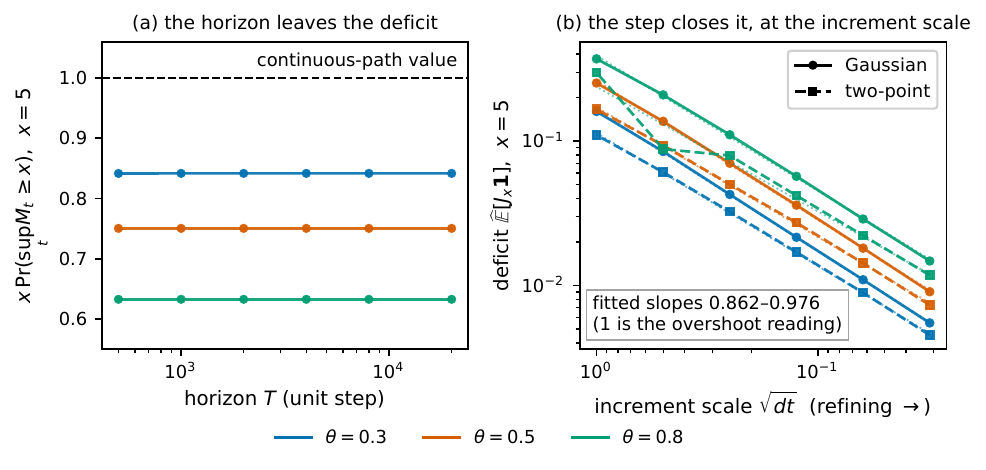}
\caption{\textbf{The horizon leaves the first-passage deficit where it is; refining the step
closes it, at the rate the increment scale sets.}
\textbf{(a)} sweeping the horizon at unit step leaves $x\,\PP(\sup_t M_t\ge x)$ unmoved
across a factor of $40$, at every tilt; this sweep is run on the Gaussian family.
\textbf{(b)} the deficit $\widehat\E[J_x\ind]$ against the increment scale $\sqrt{dt}$, on
log-log, where it is a power law: dotted lines are the fitted decays, and a slope of $1$
is the overshoot reading.  The Gaussian family sits on its fit at every tilt; the
two-point family reads a shallower slope and, at $\theta=0.8$, departs visibly from a
single power law --- the increment-distribution dependence Theorem~\ref{thm:firstpassage-exact}
records, which is why the two families are reported separately and not pooled.
Both panels are read at $x=5$.}
\label{fig:overshoot-convergence}
\end{figure}

% PAC-Bayes martingale anytime-valid generalization (CIFAR-10):
% the prior anchor separates a vacuous certificate from a tight one.
%

\subsection{At deep-net scale the prior anchor separates a vacuous certificate from a tight one}
\label{sec:eval-pacbayes-cifar-fixed-prior}

At deep-net scale the anytime-valid certificate of Theorem~\ref{thm:pacbayes} is a generalization bound that holds at every $t$ at once, and so at any random time.
A Gaussian prior $\pi$ and posterior $\rho$ over the last-linear weights, run through the parameter-mixture Freedman supermartingale of Proposition~\ref{prop:mixture}, invert Ville's inequality into
\[
  R(\rho) \;\le\; \widehat R_t(\rho)
  + \frac{\KL(\rho\|\pi) + \log(1/\delta)}{\lambda t}
  + \tfrac12\,\frac{\lambda\,\widehat V_t}{1 - \lambda/3},
\]
minimized over a $\lambda$-grid.
A frozen CIFAR ResNet-20 (held-out $92.2\%$) supplies the feature map, the posterior is the pretrained last-linear head ($650$ weights) lightly smoothed, and the stream is the held-out CIFAR-10 test set.
Write $\tau^{\mathrm{cert}}$ for the time at which the certificate is largest over the run. That choice looks ahead over the whole run and is not a stopping time; the uniform-in-$t$ guarantee covers it all the same.
Only the prior anchor is varied.

Anchored at the origin, $\pi = N(0,\sigma_\pi^2 I)$ at $\sigma_\pi=0.01$, the certificate is vacuous: the mean bound at $\tau^{\mathrm{cert}}$ is $41.7$ against a held-out risk of $0.078$.
The vacuity is entirely the complexity term, since anchoring at the origin makes $\KL(\rho\|\pi) = \|\theta_{\mathrm{pre}}\|^2/(2\sigma_\pi^2) + 206.8 \approx 3.3\times10^5$ swamp the empirical risk.
Anchored at the trained weights instead, $\pi = N(\theta_{\mathrm{pre}},\sigma_\pi^2 I)$ under identical data and martingale, the certificate at $\tau^{\mathrm{cert}}$ equals $0.153$ (five seeds) against the same held-out risk $0.078$.
There $\KL(\rho\|\pi)=206.8$ is the variance-mismatch term alone, so the tightening against the origin-anchored bound is $273\times$ and sits entirely in the anchor.
The martingale mechanism is correct in both legs; an uninformed prior simply incurs a divergence no amount of data can offset, even for a classifier that generalizes well.
The construction governs the bound, and the optimizer that reaches the weights leaves it unchanged.

The slack that separates the two is the computable divergence the identity names: the usual bound gives up exactly $\KL(\rho\|\pi_\lambda)$ against the Gibbs tilt, and the gulf between a vacuous and a tight certificate is the anchor.

% pooling benefit on a decorrelated ensemble (Corollary cor:pooling-benefit)

%
%
% Shipped legs, since 2026-08-17 (author direction): the paired single-bettor
%
% Held back for further work: the pairwise-decorrelation surrogate arm, the
% W=2 difference, the sensitivity-corner exceptions, and the ghat>1 spread in
%
\subsection{The pooling benefit survives a paired single-bettor control}
\label{sec:eval-pooling-benefit-decorrelated}

Two controls separate the pooling benefit of Corollary~\ref{cor:pooling-benefit} from optional-stopping leakage and from differences in model quality, on trained vision models.

The first is a single-bettor arm.
Every ensemble is run alongside $W=1$ on the same resampled paths, bet sizes and levels, so the reported quantity is the paired difference $\Delta = \widehat g_{1} - \widehat g_{W}$, zero by construction when pooling does nothing and positive only when the mixture sheds mass a single bettor does not.
The second is a panel of two disjoint pools of twelve checkpoints: twelve ResNet-50 checkpoints differing only in training recipe, and twelve distinct architecture families.
Twelve members allow $W\in\{1,2,4,8\}$ to be drawn without replacement inside each pool, so the two ensembles are genuinely different draws.
The pools are matched on clean accuracy to $0.013$ (means $0.686$ and $0.699$), so a difference between them is not a difference in model quality.
A third arm draws its members from the union of the two pools.
It enters the cell counts and the pooled masses reported below, and is excluded from Table~\ref{tab:eval-pooling-benefit-decorrelated}, which contrasts the two disjoint pools.

The stream is the $5{,}000$-image ImageNet-100 validation split; the alternative is the same images under a fixed corruption.
Each checkpoint bets its per-image predictive log-likelihood ratio against its own clean-stream calibration.
Across the panel, clean top-1 is $0.692$ and corrupted top-1 is $0.252$, and the least confident model still places $0.425$ of its softmax mass on the true class, so every bettor is reading a real signal.
Twenty stream orderings per cell, across $240$ cells.

A single bettor's crossing mass is $\widehat g_1 = 1.005$: at this bet size optional-stopping leakage is negligible, so the deficit at larger $W$ is attributable to the mixture; it is not inherited from the stopping rule.
The mixture's mass falls monotonically with ensemble size, $0.739$ at $W=2$, $0.590$ at $W=4$, $0.514$ at $W=8$, and $\Delta$ is positive in every one of the $180$ multi-member cells, with bootstrap intervals excluding zero in both pools at every $W$ (Table~\ref{tab:eval-pooling-benefit-decorrelated}).

At matched $W$ and matched accuracy the architecturally diverse pool sheds more mass than the recipe-diverse one: $+0.086$ at $W=4$ ($95\%$ CI $[0.047,
0.124]$) and $+0.123$ at $W=8$ (CI $[0.095, 0.151]$), which are the two ensemble sizes at which the contrast was specified in advance.

\begin{table}[h]
\centering
\caption{Both pools show a pooling benefit; the
architecturally diverse pool shows more of it at $W\ge 4$.  Paired excess deficit
$\Delta = \widehat g_1 - \widehat g_W$ by
pool and ensemble size, $20$ stream orderings per cell, $95\%$
bootstrap intervals.  The $W=1$ row is the control and is zero by
construction.}
\label{tab:eval-pooling-benefit-decorrelated}
\small
\begin{tabular}{l|ccc}
\hline
$W$ & within-family & cross-family & difference \\
\hline
$1$ & $0$ (control) & $0$ (control) & --- \\
$2$ & $0.241\ [0.217, 0.265]$ & $0.279\ [0.247, 0.315]$ & $+0.039\ [-0.002, 0.081]$ \\
$4$ & $0.363\ [0.340, 0.385]$ & $0.448\ [0.416, 0.480]$ & $+0.086\ [0.047, 0.124]$ \\
$8$ & $0.424\ [0.400, 0.449]$ & $0.547\ [0.532, 0.562]$ & $+0.123\ [0.095, 0.151]$ \\
\hline
\end{tabular}
\end{table}

At $W=1$ the geometric mixture of a single test martingale is that martingale, so the shed term $\E[A_{\tau^{(\alpha)}_x}]$ of~\eqref{eq:pooling-exact} is identically zero.
A member's increments are positive and have exact unit mean under the resampling null --- the clean stream itself, resampled --- so the member is a nonnegative martingale started at $1$.
The increments also vary, so the member has strictly negative expected log-increment and decays to $0$ along almost every path.
Theorem~\ref{thm:firstpassage-exact} asks for a nonnegative martingale started at $1$ that decays to $0$, and~\eqref{eq:fp-mass} then fixes the single-bettor crossing mass at $g_1 = 1$.
That mass is read at the value the path attains at the crossing, overshoot included, so the overshoot sits inside the conserved quantity.
The single-bettor arm estimates that mass at $1.005$ on average, so it is read here as establishing that leakage is negligible and not as a check on the bound itself.
The corruption is a single fixed family at one severity,
and the ensembles are public checkpoints that were not trained for this purpose, so public training recipes bound the disagreement available here.

\section{Deferred proofs}\label{app:proofs}
% ======================================================================

This appendix proves every result stated in the body and in Appendices~\ref{app:pooling-decomposition} and~\ref{app:curved-cts}, with one exception:
Proposition~\ref{prop:bdg-honest}, which is quoted from the literature at the citation given with it.
Proofs are grouped by the machinery they share, which in several places departs from the order in which the body states them; each proof still follows every result it draws on.

\begin{proof}[Proof of Theorem~\ref{thm:mci}]
Starting from the definition of relative entropy,
\begin{align*}
  \KL(p\|p^*_\alpha)
  &= \E_{X\sim p}\!\left[\log p(X) - \log p^*_\alpha(X)\right] \\
  &= \E_{X\sim p}\!\left[\log p(X)
     - \sum_i \alpha_i \log\pi_i(X) + \log Z(\alpha)\right]\\
  &= -H(p) + \sum_i \alpha_i\,H(p,\pi_i) + \log Z(\alpha).
\end{align*}
Rearranging and using $H(p,\pi_i) = \KL(p\|\pi_i) + H(p)$ yields \eqref{eq:mci}.
The variational form~\eqref{eq:mci-var} follows because $\KL(p\|p^*_\alpha) \geq 0$ with equality iff $p = p^*_\alpha$.
\end{proof}

\begin{proof}[Proof of Corollary~\ref{cor:DV}]
The Gibbs measure $Q^\star$ is a probability measure because $0 < \int g\,\dd\mu < \infty$, and is supported on $\{g>0\}$.
For any probability $Q \ll \mu$ with $Q(\{g=0\})=0$ (so $Q \ll Q^\star$),
\[
  \KL(Q\|Q^\star)
  = \E_Q\!\left[\log\frac{\dd Q/\dd\mu}{g/\int g\,\dd\mu}\right]
  = \KL(Q\|\mu) - \E_Q[\log g] + \log\!\int g\,\dd\mu.
\]
Rearranging yields~\eqref{eq:DV}.
Nonnegativity of $\KL(Q\|Q^\star)$ with equality iff $Q = Q^\star$ gives~\eqref{eq:DV-var};
a $Q$ with $Q(\{g=0\})>0$ has $\E_Q[\log g]=-\infty$ and so does not attain the supremum.
\end{proof}

\begin{proof}[Proof of Proposition~\ref{prop:multi-PAC}]
Set $\bar\alpha = 1$ in~\eqref{eq:mci}; the entropy term vanishes.
\end{proof}

\begin{proof}[Proof of Theorem~\ref{thm:OST}]
For each $n \geq 0$, the stopped process $(M_{\tau \wedge t})_t$ is a nonnegative supermartingale with $\E[M_{\tau \wedge n}] \leq M_0$
(discrete-time optional stopping at the bounded stopping time $\tau \wedge n$).
Since $M \geq 0$, Fatou's lemma gives $\E[M_\tau] = \E[\liminf_{n\to\infty} M_{\tau \wedge n}]
\leq \liminf_n \E[M_{\tau \wedge n}] \leq M_0$, the stated inequality.

For the equality cases assume $(M_t)$ is a martingale, so $\E[M_{\tau\wedge n}] = M_0$ for every $n$.  In case~(a), $\tau \wedge n
= \tau$ for $n$ large, so $\E[M_\tau] = M_0$ directly.  In case~(b),
uniform integrability of $\{M_{\tau\wedge n}\}_n$ upgrades the a.s.\
convergence $M_{\tau\wedge n} \to M_\tau$ to convergence in $L^1$, from which $\E[M_\tau] = \lim_n \E[M_{\tau\wedge n}] = M_0$.
Finiteness of $\E[M_\tau]$ does not by itself furnish this uniform integrability:
for the martingale $M_t = 2^t\,\ind\{X_1 = \dots = X_t = 1\}$ on i.i.d.\
fair bits, $\E[M_t] = 1$ for all $t$ but $M_t \to 0$ a.s., so the deterministic time $\tau = \infty$ has $\E[M_\tau] = 0 < 1 = M_0$ with $M_\tau$ integrable; here $\lim_n \E[M_{\tau\wedge n}] = 1 \neq 0 =
\E[M_\tau]$, exhibiting the failure of uniform integrability.  Downstream
arguments that need equality (not merely $\le$) invoke one of the sufficient conditions above; the equality is never used for a martingale that can leak mass to $\{\tau = \infty\}$.
\end{proof}

\begin{proof}[Proof of Theorem~\ref{thm:surv-factor}]
Write $G_t := P(\tau\geq t\mid\cF_t)$, so $H_t = P(\tau=t\mid\cF_t)/G_t$ and $P(\tau\geq t\mid\cF_{t-1}) = S_{t-1}$ (since $\{\tau\geq t\}=\{\tau>t-1\}$).
Parts~(i) and~(iv) are immediate from~\eqref{eq:clock-mart}.
Every ratio below is taken on $\{S_{t-1}>0\}$; on the complement $\{S_{t-1}=0\}$ the numerator $G_t$ vanishes almost surely as well, and \eqref{eq:clock-mart} reads the factor as $1$, so both sides of each displayed identity are unchanged there and the divisions are legitimate.
\emph{(ii)}~The one-step ratio $R_t^\tau := G_t/S_{t-1}$ satisfies $\E[R_t^\tau\mid\cF_{t-1}] = \E[G_t\mid\cF_{t-1}]/S_{t-1} = 1$ by the tower property, so $\Amart{\tau}_t = \Amart{\tau}_{t-1}R_t^\tau$ is a nonnegative $P$-martingale with $\Amart{\tau}_0 = 1$.
\emph{(iii)}~Since $1-H_t = P(\tau>t\mid\cF_t)/G_t = S_t/G_t$,
we have $1-\clock{\tau}_T = \prod_{t\leq T} S_t/G_t$, from which
\[
  (1-\clock{\tau}_T)\,\Amart{\tau}_T
  = \prod_{t=1}^T \frac{S_t}{G_t}\cdot\frac{G_t}{S_{t-1}}
  = \prod_{t=1}^T \frac{S_t}{S_{t-1}}
  = \frac{S_T}{S_0} = S_T,
\]
using $S_0 = P(\tau\geq 1\mid\cF_0) = 1$.
\emph{(v)}~As $\clock{\tau}_t - \clock{\tau}_{t-1} = (1-\clock{\tau}_{t-1})H_t$ and, by~(iii) at $t-1$,
$(1-\clock{\tau}_{t-1})\Amart{\tau}_t = S_{t-1}R_t^\tau = G_t$, the mass identity is $\Amart{\tau}_t(\clock{\tau}_t - \clock{\tau}_{t-1}) = G_t H_t = P(\tau = t\mid\cF_t)$.
\end{proof}

\begin{proof}[Proof of Theorem~\ref{thm:rep}]
By the tower property and the mass identity of Theorem~\ref{thm:surv-factor}(v),
\[
  \E[V_\tau]
  = \E\!\left[\sum_{t} V_t \ind\{\tau=t\}\right]
  = \E\!\left[\sum_t V_t P(\tau = t \mid \cF_t)\right]
  = \E\!\left[\sum_t V_t \Amart{\tau}_t (\clock{\tau}_t - \clock{\tau}_{t-1})\right].\qedhere
\]
\end{proof}

\begin{proof}[Proof of Theorem~\ref{thm:ville-tight}]
The first equality holds because $M_{\tau^\star} = \sup_{s\ge1} M_s$ by definition of $\tau^\star$, and $\sup_{t\ge0}M_t=\max(1,\sup_{s\ge1}M_s)$ exceeds a level $x>1$ exactly when $\sup_{s\ge1}M_s$ does.
For the second, optional stopping of the nonnegative martingale $M$ at $\sigma_x \wedge n$ with $\sigma_x := \inf\{t : M_t \geq x\}$ gives the upper bound $\PP(\sigma_x < \infty) \leq 1/x$ in general, the deficit being the mean overshoot $M_{\sigma_x} - x \geq 0$.
The no-overshoot hypothesis $M_{\sigma_x} = x$ upgrades this to the equality $\PP(\sigma_x < \infty) = 1/x$ (the classical Doob maximal identity, in which the running maximum is $M_0/U$ with $U \sim
\mathrm{Uniform}(0,1)$, Pareto$(1)$
\cite{nikeghbali2006doobs,quelques1972quelques}).
For the last statement, apply the no-overshoot hypothesis at rational levels only, and discard the countable union of the corresponding null sets.
Off that null set, suppose $\sup_t M_t>1$ and let $t^\ast:=\min\{t\ge1:M_t>1\}$; then $\max_{s<t^\ast}M_s=1$, so every rational $x\in(1,M_{t^\ast})$ has $\sigma_x=t^\ast$ and hence $M_{t^\ast}=x$, impossible for two distinct rationals.
Therefore $\sup_t M_t\le1$ a.s.; with $\E[M_t]=1$ and $M_t\ge0$ this forces $M\equiv1$, contradicting $M_t\to0$.
\end{proof}

\begin{proof}[Proof of Proposition~\ref{prop:maximal-conditional}]
\emph{(i)}  Fix $a>m$ and let $\sigma_a:=\inf\{u\ge0:M_u\ge a\}$.
Continuity of $M^\ast$ gives $M_{\sigma_a}=a$ on $\{\sigma_a<\infty\}$ and $M_u\le a$ for $u\le\sigma_a$, so the stopped process $M^{\sigma_a}$ is a bounded local martingale, hence a uniformly integrable martingale, and $\E[M_{\sigma_a\wedge n}]=m$ for every $n$.
Bounded convergence and $M_\infty=0$ send the left side to $a\,\PP(\sigma_a<\infty)+\E[M_\infty\ind\{\sigma_a=\infty\}]=a\,\PP(\sigma_a<\infty)$,
so $\PP(M^\ast_\infty\ge a)=\PP(\sigma_a<\infty)=m/a$.
The right side is continuous in $a$, so $\PP(M^\ast_\infty>a)=\lim_{b\downarrow
a}\PP(M^\ast_\infty\ge b)=m/a$.  For $a\le m$ the same limit along $b\downarrow
m$ gives $\PP(M^\ast_\infty>m)=1$, and $M^\ast_\infty\ge m\ge a$ then gives
$\PP(M^\ast_\infty>a)=1$, which is the first claim.  For $v\in(0,1)$,
$\PP(m/M^\ast_\infty<v)=\PP(M^\ast_\infty>m/v)=v$, so $m/M^\ast_\infty$ is uniform on $(0,1)$.

\emph{(ii)}  Set $\widetilde M_u:=M_{\tau+u}$ and $\widetilde\cF_u:=\cF_{\tau+u}$ for $u\ge0$.  Since $\tau<\infty$ a.s.,
$\widetilde M$ is a strictly positive local martingale for $(\widetilde\cF_u)_{u\ge0}$ with $\widetilde M_0=M_\tau>0$ and $\widetilde M_u\to0$ a.s.; since $M$ has no positive jumps, neither does $\widetilde M$, so its running supremum $\widetilde M^\ast_u=\sup_{v\le u}M_{\tau+v}$ is continuous, with ultimate value $M^\ast_{\ge\tau}$.
Continuity of the global $M^\ast$ alone would not suffice: a positive jump of $M$ below the standing record leaves $M^\ast$ continuous but makes $\widetilde M^\ast$ jump.
Running the argument of part~(i) on $\widetilde M$ with every expectation taken conditionally on $\widetilde\cF_0=\cF_\tau$ replaces $m$ by $M_\tau$ and gives~\eqref{eq:maximal-cond} for every $a>0$.
Hence the conditional law of $M^\ast_{\ge\tau}$ given $\cF_\tau$ is that of $M_\tau/U$ with $U$ uniform on $(0,1)$, so $M_\tau/M^\ast_{\ge\tau}=U$ has conditional law uniform on $(0,1)$; a conditional law free of the conditioning field is independence of it.
Taking logarithms, $\log(M^\ast_{\ge\tau}/M_\tau)=-\log U$ is standard exponential and independent of $\cF_\tau$, with mean $1$.
\end{proof}

\begin{proof}[Proof of Theorem~\ref{thm:firstpassage-exact}]
Optional stopping at the bounded time $\sigma_x\wedge n$ gives $\E[M_{\sigma_x\wedge n}]=1$, i.e.
$\E[M_{\sigma_x}\ind\{\sigma_x\le n\}]+\E[M_n\ind\{\sigma_x>n\}]=1$.
The first term increases to $\E[M_{\sigma_x}\ind\{\sigma_x<\infty\}]$ by monotone convergence.
For the second, $M_n<x$ on $\{\sigma_x>n\}$ and $M_n\ind\{\sigma_x=\infty\}\to M_\infty=0$ a.s.\ dominated by $x$, so it tends to $0$; \eqref{eq:fp-mass} follows.
Writing $M_{\sigma_x}=x+J_x$ in \eqref{eq:fp-mass} gives $x\,\PP(\sigma_x<\infty)+\E[J_x\ind\{\sigma_x<\infty\}]=1$,
which is~\eqref{eq:fp-exact}.
\end{proof}

\begin{proof}[Proof of Lemma~\ref{lem:chain-rule}]
$\KL(Q\|P) = \E_Q[\log(\dd Q / \dd P)] = \E_Q[\sum_t \log R_t]
= \sum_t \E_Q[\log R_t]$.  Conditioning on $X_{1:t-1}$,
$\E_Q[\log R_t \mid X_{1:t-1}] = \KL(Q_t \| P_t)$ evaluated at the realized history, from which the conditional-divergence form.
\end{proof}

\begin{proof}[Proof of Theorem~\ref{thm:pathspace}]
Apply Corollary~\ref{cor:DV} (finite-measure DV) with $\mu = P|_{\cF_T}$ and $g = \prod_i \Pi_{i,T}^{\alpha_i}$; $\log g = \sum_i \alpha_i \log
\Pi_{i,T}$, and its $\KL(Q\|Q^\star)$ term is $\KL(Q\|Q_T^\alpha)$, giving the
identity~\eqref{eq:pathspace} for every $Q$; the optimizer is the Gibbs law~\eqref{eq:Qtstar}.
\end{proof}

\begin{proof}[Proof of Theorem~\ref{thm:seqMCI}]
The substitution $\Pi_{t,T} = R_t$ in Theorem~\ref{thm:pathspace} is admissible, the $R_t$ being strictly positive and $\cF_T$-measurable, and $\log L_T^{(\alpha)} = \sum_t \alpha_t \log R_t$ identifies the linear term.
\end{proof}

\begin{proof}[Proof of Corollary~\ref{cor:seqMCI-cond}]
Substitute Lemma~\ref{lem:chain-rule} into the variational form of~\eqref{eq:seqMCI}; the conditional form of $R^*_\alpha$ follows from $\dd R^*_\alpha / \dd P \propto \prod_t (q_t/p_t)^{\alpha_t}$.
\end{proof}

\begin{proof}[Proof of Proposition~\ref{prop:grad}]
Differentiating $Z_T(\alpha) = \E_P[\exp(\sum_k \alpha_k \log
\Pi_{k,T})]$ gives the expectation identity under the tilted measure
$Q_T^\alpha$; a second differentiation gives the covariance.
Positive semidefiniteness of covariance implies convexity.
\end{proof}

\begin{proof}[Proof of Proposition~\ref{prop:coalescence}]
For each $i$, let $\{X_i^{(m)}\}_{m=1}^{\alpha_i}$ be independent samples from $P_t^{(i)}$, independent also across $i$.
The event $\bigcap_i \bigcap_m \{X_i^{(m)} = \gamma\}$ has probability $\prod_i P_t^{(i)}(\gamma)^{\alpha_i}$; summing over $\gamma$ gives the result.
\end{proof}

\begin{proof}[Proof of Theorem~\ref{thm:predictive}]
Extending a prefix $\gamma$ by one state $x$,
\[
  Z_{t+1}(\alpha)
  = \sum_\gamma \prod_i P_t^{(i)}(\gamma)^{\alpha_i}
     \sum_x \prod_i P(X_{t+1}=x \mid \gamma; P^{(i)})^{\alpha_i}
  = Z_t(\alpha) \sum_\gamma Q_t^\alpha(\gamma)\kappa_t^\alpha(\gamma),
\]
giving~\eqref{eq:predictive1}.  For nonnegativity of $\cC_\alpha(t+1) -
\cC_\alpha(t)$, observe that with $\bar\alpha = \sum_i \alpha_i \geq 1$ and
$\widetilde\alpha_i = \alpha_i/\bar\alpha$, and $p_i(x) \in [0,1]$,
$\prod_i p_i(x)^{\alpha_i} = \bigl(\prod_i p_i(x)^{\widetilde\alpha_i}
\bigr)^{\bar\alpha} \leq \prod_i p_i(x)^{\widetilde\alpha_i}$.
By H\"older on the simplex,
$\kappa_t^\alpha(\gamma) \leq \sum_x \prod_i p_i(x)^{\widetilde\alpha_i}
\leq \prod_i (\sum_x p_i(x))^{\widetilde\alpha_i} = 1$.
Taking logs yields~\eqref{eq:predictive2}.
\end{proof}

\begin{proof}[Proof of Proposition~\ref{prop:markov}]
Multiplying Markov densities yields the first identity.
For the asymptotic, Perron--Frobenius gives positive left/right eigenvectors $u,v$ with eigenvalue $r(T_\alpha) > 0$ and $T_\alpha^t = r(T_\alpha)^t (v u^\top + o(1))$.
Since $\mu_\alpha$ and $\mathbf{1}$ are nonnegative and nonzero, $\langle \mu_\alpha, v\rangle
> 0$ and $\langle u, \mathbf{1}\rangle > 0$, so
$Z_t(\alpha) = r(T_\alpha)^t(\langle\mu_\alpha,v\rangle\langle u,\mathbf{1}
\rangle + o(1))$; dividing by $t$ and taking logs gives the stated spectral-radius limit.
\end{proof}

\begin{proof}[Proof of Theorem~\ref{thm:master}]
Each one-step factor of~\eqref{eq:doleans} has conditional mean one: $\psi_s(\lambda)$ is $\cF_{s-1}$-measurable, so $\E[e^{\lambda d_s - \psi_s(\lambda)} \mid \cF_{s-1}] = e^{-\psi_s(\lambda)}\,\E[e^{\lambda d_s} \mid \cF_{s-1}] = 1$.
Hence $\E[\cE_t(\lambda) \mid \cF_{t-1}] = \cE_{t-1}(\lambda)$, and iterating from $\cE_0(\lambda) = 1$ gives $\E_P[\cE_t(\lambda)] = 1$, i.e.\ $\log\E_P[\cE_t(\lambda)] = 0$.
Apply Theorem~\ref{thm:pathspace} (equivalently Corollary~\ref{cor:DV}) on $(\Omega, \cF_t, P)$ with the single factor $\Pi_t = \cE_t(\lambda)$: for every $Q \ll P|_{\cF_t}$,
$\log \E_P[\cE_t(\lambda)] = \lambda\,\E_Q[Y_t] - \E_Q[\Psi_t(\lambda)] - \KL(Q\|P|_{\cF_t}) + \KL(Q\|Q^\star_t)$ with $\dd Q^\star_t/\dd P \propto \cE_t(\lambda)$.
The left-hand side is zero, which is~\eqref{eq:master-identity}.
\end{proof}

\begin{proof}[Proof of Proposition~\ref{prop:supermart-relax}]
(i) With $\Delta D_t(\lambda) := D_t(\lambda) - D_{t-1}(\lambda) \geq 0$ predictable,
$\E[E_t(\lambda) \mid \cF_{t-1}] = E_{t-1}(\lambda)\,e^{-\Delta D_t(\lambda)}\,\E[e^{\lambda d_t - \psi_t(\lambda)} \mid \cF_{t-1}] = E_{t-1}(\lambda)\,e^{-\Delta D_t(\lambda)} \leq E_{t-1}(\lambda)$,
so $E(\lambda)$ is a nonnegative supermartingale with $E_0(\lambda) = 1$.
Since $\dd Q^\star_t/\dd P = \cE_t(\lambda)$ (unit mean by Theorem~\ref{thm:master}),
$\E_P[E_t(\lambda)] = \E_P[\cE_t(\lambda)\,e^{-D_t(\lambda)}] = \E_{Q^\star_t}[e^{-D_t(\lambda)}] \leq 1$, which is~\eqref{eq:relax-deficit}.

(ii) Substituting $\E_Q[\overline\Psi_t(\lambda)] = \E_Q[\Psi_t(\lambda)] + \E_Q[D_t(\lambda)]$ into~\eqref{eq:master-identity} gives $\KL(Q\|P|_{\cF_t}) - \bigl(\lambda\,\E_Q[Y_t] - \E_Q[\overline\Psi_t(\lambda)]\bigr) = \KL(Q\|Q^\star_t) + \E_Q[D_t(\lambda)] \geq 0$, which is~\eqref{eq:master} with the stated slack.

(iii) Apply Corollary~\ref{cor:DV} on $(\Omega, \cF_t, P)$ with $g = E_t(\lambda)$:
$\log \E_P[E_t(\lambda)] = \lambda\,\E_Q[Y_t] - \E_Q[\overline\Psi_t(\lambda)] - \KL(Q\|P|_{\cF_t}) + \KL(Q\|\overline Q^\star_t)$ for every $Q \ll P|_{\cF_t}$, with $\dd \overline Q^\star_t/\dd P \propto E_t(\lambda)$.
The supermartingale property and $E_0(\lambda) \leq 1$ give $\E_P[E_t(\lambda)] \leq 1$, so the left-hand side is $\leq 0$ and~\eqref{eq:master} follows with slack $\KL(Q\|\overline Q^\star_t) - \log\E_P[E_t(\lambda)]$.
\end{proof}

\begin{proof}[Proof of Corollary~\ref{cor:event}]
Apply Corollary~\ref{cor:DV} on $(\Omega,\cF_t,P)$ with $g = E_t(\lambda)$, which is strictly positive, and $Q = P(\cdot \mid A)$.
Then $\E_Q[\log g] = \lambda\E[Y_t \mid A] - \E[\overline\Psi_t(\lambda) \mid A]$ and $\KL(Q\|P|_{\cF_t}) = -\log P(A)$, and~\eqref{eq:DV} rearranges to~\eqref{eq:event}.
The first residual is nonnegative because it is a relative entropy, the second because $\E_P[E_t(\lambda)] \leq 1$ by the supermartingale property with $E_0(\lambda) \leq 1$; discarding both gives~\eqref{eq:event-bound}.
\end{proof}

\begin{proof}[Proof of Corollary~\ref{cor:ville}]
Set $\tau_x := \inf\{t \geq 0 : E_t \geq e^x\}$.
Optional stopping at $\tau_x \wedge n$ gives $\E[E_{\tau_x \wedge n}] \leq E_0 \leq 1$.
On $\{\tau_x \leq n\}$, $E_{\tau_x \wedge n} \geq e^x$, so $e^x \PP(\tau_x \leq n) \leq 1$, and $n \to \infty$ gives $\PP(\tau_x < \infty) \leq e^{-x}$.
Since $\{\sup_t E_t \geq e^x\} \subseteq \{\tau_y < \infty\}$ for the same construction at every level $e^y$ with $y < x$, letting $y \uparrow x$ gives Ville's inequality.
The line-crossing consequence is immediate.
\end{proof}

\begin{proof}[Proof of the conditional Hoeffding lemma of Section~\ref{sec:cumulant-majorants}]
Since $\E[X\mid\cG] = 0$ and $a \leq X \leq b$ a.s., necessarily $a \leq 0 \leq b$.
By convexity of $x \mapsto e^{\lambda x}$,
\[
  e^{\lambda X}
  \leq \frac{b - X}{b - a}e^{\lambda a}
       + \frac{X - a}{b - a}e^{\lambda b}.
\]
Taking $\cG$-conditional expectations and using $\E[X\mid\cG] = 0$,
\[
  \E[e^{\lambda X}\mid\cG]
  \leq \frac{b}{b-a}e^{\lambda a} - \frac{a}{b-a}e^{\lambda b}.
\]
Set $h := \lambda(b-a)$ and $u := -a/(b-a) \in [0,1]$; the right-hand side becomes $(1-u)e^{-uh} + u e^{(1-u)h}$.
Define $f(h) := \log\bigl((1-u)e^{-uh} + u e^{(1-u)h}\bigr)$.
A direct computation gives $f(0) = f'(0) = 0$ and
\[
  f''(h) = \frac{u(1-u) e^h}{\bigl((1-u) + u e^h\bigr)^2} \leq \tfrac14,
\]
using $x/(1+x)^2 \leq 1/4$ for $x \geq 0$.
Taylor's theorem gives $f(h) \leq h^2/8$, and exponentiating yields $\E[e^{\lambda X}\mid\cG] \leq \exp(\lambda^2(b-a)^2/8)$.
\end{proof}

\begin{proof}[Proof of Theorem~\ref{thm:azuma}]
Write $U_t := \tfrac{1}{8}\sum_{s \leq t}(b_s - a_s)^2$.
The conditional Hoeffding lemma gives the cumulant majorant $\overline\Psi_t(\lambda) = \lambda^2 U_t$, so the bound branch~\eqref{eq:event-bound} of Corollary~\ref{cor:event} on $A = \{M_t - M_0 \geq x\}$ yields $\PP(A) \leq \exp(-\lambda x + \lambda^2 U_t)$; optimizing at $\lambda^* = x/(2U_t)$ gives~\eqref{eq:azuma1}, and the same tilt in Corollary~\ref{cor:ville} gives the line-crossing form~\eqref{eq:azuma2}.
The exact form~\eqref{eq:azuma-exact} is Corollary~\ref{cor:DV} with $g = \ind_A$ (so $\log P(A) = -\KL(P(\cdot|A)\|P)$) followed by the chain rule for relative entropy (Lemma~\ref{lem:chain-rule}).
\end{proof}

\begin{proof}[Proof of Lemma~\ref{lem:bernstein}]
We first prove the scalar inequality
\begin{equation}\label{eq:scalar-bernstein}
  e^u \leq 1 + u + \frac{u^2}{2(1 - u/3)},
  \qquad u < 3.
\end{equation}
Since $1 - u/3 > 0$ throughout $u < 3$, multiplying~\eqref{eq:scalar-bernstein} through by it leaves an equivalent inequality, $T(u) := 1 + \tfrac{2u}{3} + \tfrac{u^2}{6} - (1 - u/3)e^u \geq 0$.
Here $T(0) = T'(0) = 0$ and $T''(u) = \tfrac13\bigl(1 - (1-u)e^u\bigr) \geq 0$ at every real $u$,
since $(1-u)e^u \leq 1$; thus $T$ is convex with a global minimum of $0$ at the origin, and~\eqref{eq:scalar-bernstein} holds on the whole range.
Now apply~\eqref{eq:scalar-bernstein} with $u = \lambda X$; since $X \leq b$ and $0 \leq \lambda < 3/b$, $\lambda X < 3$ a.s., so
\[
  e^{\lambda X}
  \leq 1 + \lambda X + \frac{\lambda^2 X^2}{2(1 - \lambda X /3)}
  \leq 1 + \lambda X + \frac{\lambda^2 X^2}{2(1 - \lambda b/3)}.
\]
Conditional expectations and $\E[X\mid\cG] = 0$ give $\E[e^{\lambda X}\mid\cG] \leq 1 + \frac{\lambda^2}{2(1 - \lambda b/3)}
\E[X^2\mid\cG]$; finally $1 + z \leq e^z$
yields~\eqref{eq:bernstein-lemma}.
\end{proof}

\begin{proof}[Proof of Proposition~\ref{prop:freedman}]
Lemma~\ref{lem:bernstein} gives the conditional cumulant majorant $\overline\Psi_t(\lambda) = \frac{\lambda^2}{2(1-\lambda b/3)}\E[d_t^2\mid\cF_{t-1}]$,
so $E_t(\lambda)$ is a nonnegative supermartingale and, on $A_{x,v}$,
Corollary~\ref{cor:ville} bounds $\PP(A_{x,v})$ by $\exp(-\lambda x + \frac{\lambda^2}{2(1-\lambda b/3)}v)$; evaluating at the classical tilt $\lambda^\star = x/(v + bx/3)$, admissible since $\lambda^\star < 3/b$, yields~\eqref{eq:freedman}.
\end{proof}

\begin{proof}[Proof of Proposition~\ref{prop:mixture}]
The supermartingale property follows from Tonelli's theorem.
Applying the one-factor finite-measure DV identity (Corollary~\ref{cor:DV}) on $(\Lambda,\mu)$ to $\lambda \mapsto E_t(\lambda,\omega)$ pointwise in $\omega$ gives the identity~\eqref{eq:mixturepathwise} with $\dd\rho^*_t/\dd\mu \propto E_t(\cdot,\omega)$; dropping $\KL(\rho_t\|\rho^*_t) \ge 0$ gives the lower bound, and Ville's inequality applied to $\overline E_t$ yields~\eqref{eq:mixturePAC}.
\end{proof}

\begin{proof}[Proof of Proposition~\ref{prop:mixture-exact}]
Fix $t$.  If $\overline E_t = 0$ then $E_t(\lambda) = 0$ for $\mu$-almost every $\lambda$, hence for $\rho_t$-almost every $\lambda$ since $\rho_t\ll\mu$, so $\int_\Lambda\log E_t\,\dd\rho_t = -\infty$ and the defining inequality of $\cA_x$ fails at that $t$; the same $t$ is excluded from $\cV_x$ because $e^x>1>0$.
At every $t$ with $\overline E_t > 0$,
rearranging~\eqref{eq:mixturepathwise} gives
\[
  \int_\Lambda \log E_t(\lambda)\,\rho_t(\dd\lambda) - \KL(\rho_t\|\mu)
  = \log\overline E_t - \KL(\rho_t\|\rho^*_t),
\]
so the defining inequality of $\cA_x$ at $t$ reads $\log\overline E_t - \KL(\rho_t\|\rho^*_t) \geq x$.
Since $\KL(\rho_t\|\rho^*_t)\geq0$ that forces $\overline E_t \geq e^x$, from which $\cA_x\subseteq \cV_x$; negating the same inequality at every $t$ gives the displayed form of $\cV_x\setminus \cA_x$, and $\PP(\cA_x) = \PP(\cV_x) -
\PP(\cV_x\setminus \cA_x)$ is additivity on the disjoint union
$\cV_x = \cA_x \sqcup (\cV_x\setminus \cA_x)$.
When $\rho_t \equiv \rho^*_t$ the residual $\KL(\rho_t\|\rho^*_t)$ vanishes identically and the two events coincide.
Under the added hypotheses $\overline E$ is a nonnegative martingale with $\overline E_0 = 1$ and $\overline E_t\to0$ a.s., so $\cV_x = \{\sigma<\infty\}$ for the first passage $\sigma$ to the level $e^x > 1$, and Theorem~\ref{thm:firstpassage-exact} at that level gives $\PP(\cV_x) = e^{-x}\bigl(1 - \E[J\ind\{\sigma<\infty\}]\bigr)$, which is~\eqref{eq:mixture-exact-mart}.
Both residuals are nonnegative, and $\E[J\ind\{\sigma<\infty\}] = 0$ with $J\ge0$ forces $J = 0$ a.s.\ on $\{\sigma<\infty\}$, so equality in~\eqref{eq:mixturePAC} holds exactly when each vanishes.
\end{proof}

\begin{proof}[Proof of Lemma~\ref{lem:sigma-family}]
Take first $f\in C^1$.  Integration by parts for the product of the continuous finite-variation process $f(A_t)$ with the semimartingale $X=N+A$ gives
\[
  f(A_t)X_t
  =\int_0^t f(A_u)\,\dd N_u
   +\int_0^t f(A_u)\,\dd A_u
   +\int_0^t f'(A_u)\,X_u\,\dd A_u
\]
there being no bracket term because $f\circ A$ is continuous of finite variation.
The measure $\dd A$ is carried by $\{X=0\}$, so the third integral vanishes.
Change of variables along the continuous nondecreasing $A$ with $A_0=0$ evaluates the second as $\int_0^{A_t}f(z)\,\dd z$.
Rearranging gives~\eqref{eq:sigma-family}, whose right-hand side is a stochastic integral of a locally bounded progressively measurable integrand against a local martingale, hence a local martingale null at $0$.
Both sides of~\eqref{eq:sigma-family} are linear in $f$ and stable under bounded pointwise convergence with uniformly bounded primitives, and the $C^1$ functions generate the locally bounded Borel functions under such limits, so a monotone class argument extends the identity to every locally bounded Borel $f$~\cite{Nikeghbali2006essay}.

For the converse, take $f\equiv1$: then $C_t-X_t$ is a local martingale, so $C$ is the increasing process of the Doob--Meyer decomposition of $X$ and $C=A$.
Take next $f(z)=2z$, so that $A_t^2-2A_tX_t$ is a local martingale.
Integration by parts and $\dd(A^2)=2A\,\dd A$ give
\[
  A_t^2-2A_tX_t
  =2\int_0^t A_u\,(\dd A_u-\dd X_u)-2\int_0^t X_u\,\dd A_u
\]
whose first term is already a local martingale, since $\dd A-\dd X=-\dd N$.
Hence $\int_0^t X_u\,\dd A_u$ is a continuous nondecreasing local martingale null at $0$, so it vanishes identically, and $\dd A$ is carried by $\{X=0\}$~\cite{Nikeghbali2006essay}.
\end{proof}

\begin{proof}[Proof of Theorem~\ref{thm:curved-crossing}]
\emph{Reduction.}  For Borel $\psi\ge\varphi$ pointwise, the crossing event for $\psi$ is contained in the crossing event for $\varphi$.
For $\epsilon>0$ and $n\in\N$ let $\varphi_{\epsilon,n}:=\varphi\vee\epsilon$ on $[0,n)$ and $:=+\infty$ on $[n,\infty)$, so that $1/\varphi_{\epsilon,n}\le1/\epsilon$ is bounded and $\int_0^\infty\varphi_{\epsilon,n}(x)^{-1}\dd x\le n/\epsilon$ is finite.
These functions decrease to $\varphi$ as $\epsilon\downarrow0$ and $n\to\infty$, so the crossing events increase; their union is the crossing event for $\varphi$, because $X_t>\varphi(A_t)>0$ forces $\varphi_{\epsilon,n}(A_t)=\varphi(A_t)\vee\epsilon<X_t$ once $\epsilon<X_t$ and $n>A_t$.
Monotone convergence of probabilities on the left and of $1/\varphi_{\epsilon,n}\uparrow1/\varphi$ on the right, the latter sending the divergent case to the value $1$, reduce~\eqref{eq:curved-crossing} to the case in which $1/\varphi$ is bounded with $\int_0^\infty\varphi(x)^{-1}\dd x$ finite.  Assume that from here on.

\emph{The martingale.}  Set
\[
  \zeta(x):=1-\exp\Bigl(-\int_x^\infty\varphi(z)^{-1}\dd z\Bigr),
  \qquad
  f(x):=-\varphi(x)^{-1}\bigl(1-\zeta(x)\bigr)
\]
Then $\zeta$ is continuous and nonincreasing with values in $[0,1)$, has Lebesgue derivative $f$, satisfies $\zeta(x)\to0$ as $x\to\infty$, and obeys the pointwise relation $-f(x)\varphi(x)=1-\zeta(x)$.
The function $f$ is Borel with $|f|\le1/\varphi$ bounded, so Lemma~\ref{lem:sigma-family} applies with $\int_0^{A_t}f(z)\,\dd z=\zeta(A_t)-\zeta(0)$ and
\[
  M_t:=\zeta(A_t)-f(A_t)X_t
     =\zeta(A_t)+X_t\,\varphi(A_t)^{-1}\bigl(1-\zeta(A_t)\bigr)
\]
is a local martingale with $M_0=\zeta(0)=1-\exp\bigl(-\int_0^\infty\varphi(x)^{-1}\dd x\bigr)\in(0,1]$.
Every term on the right is nonnegative, so $M\ge0$.

\emph{The running supremum of $M$ is continuous.}  By~\eqref{eq:sigma-family},
$M_t=M_0-\int_0^t f(A_u)\,\dd N_u$, so $\Delta M_t=-f(A_t)\,\Delta N_t$ at every $t$.
The process $A$ is continuous, so $\Delta N_t=\Delta X_t\le0$ by hypothesis, and $-f\ge0$; hence $\Delta M_t\le0$.
A c\`adl\`ag process with no positive jumps has a continuous running supremum.
The argument uses only the sign of $f$ and the continuity of $A$, so no regularity of $\varphi$ beyond measurability enters.

\emph{$M_\infty=0$.}  Since $A_t\to\infty$, the inverse $A^{-1}_u:=\inf\{t:A_t>u\}$ is finite for every $u$ and increases to $\infty$.
Continuity of $A$ gives $A_{A^{-1}_u}=u$, and for $t_n\downarrow A^{-1}_u$ with $A_{t_n}>u$ the measure $\dd A$ charges $(A^{-1}_u,t_n]$, so that interval meets $\{X=0\}$; right-continuity of $X$ then gives $X_{A^{-1}_u}=0$ and $M_{A^{-1}_u}=\zeta(u)$.
A nonnegative local martingale is a supermartingale and converges a.s., and $\zeta(u)\to0$, so the limit is $M_\infty=0$.

\emph{The crossing event.}  Fix $t$.
If $X_t>\varphi(A_t)$ then $\varphi(A_t)<\infty$, hence $f(A_t)<0$, and $M_t>\zeta(A_t)-f(A_t)\varphi(A_t)=\zeta(A_t)+\bigl(1-\zeta(A_t)\bigr)=1$.
Conversely, if $M_t>1$ then $f(A_t)\ne0$, since $f(A_t)=0$ would give $M_t=\zeta(A_t)\le1$; dividing $-f(A_t)X_t>1-\zeta(A_t)=-f(A_t)\varphi(A_t)$ by $-f(A_t)>0$ gives $X_t>\varphi(A_t)$.
Therefore $\{\exists\,t\ge0:X_t>\varphi(A_t)\}=\{\sup_t M_t>1\}$.

\emph{Evaluation.}  The process $M$ is a nonnegative local martingale with $M_0>0$, $M_\infty=0$ and continuous running supremum, so Proposition~\ref{prop:maximal-conditional}(i) at the level $a=1\ge M_0$ gives $\PP(\sup_t M_t>1)=(M_0/1)\wedge1=M_0$, which is~\eqref{eq:curved-crossing}.

\emph{The clock-indexed form.}  Fix $u>0$ and put $\varphi_u:=\varphi$ on $[0,u)$ and $\varphi_u:=+\infty$ on $[u,\infty)$, a Borel function with $\int_0^\infty\varphi_u(x)^{-1}\dd x=\int_0^u\varphi(x)^{-1}\dd x$.
Since $X_t>\varphi_u(A_t)$ forces $A_t<u$, the crossing event for $\varphi_u$ is the event of~\eqref{eq:curved-crossing-u}, and applying~\eqref{eq:curved-crossing} to $\varphi_u$ gives~\eqref{eq:curved-crossing-u}.
\end{proof}

\begin{proof}[Proof of Theorem~\ref{thm:curved-crossing-discrete}]
\emph{The process and the crossing event.}  Since $I<\infty$ the function $w$ takes values in $(0,1]$, so $\zeta=1-w$ takes values in $[0,1)$ and $h=w/\varphi\ge0$; hence $M_t=\zeta(A_t)+h(A_t)X_t\ge0$, and $A_0=X_0=0$ gives $M_0=\zeta(0)=1-e^{-I}$.
Fix $t$ with $\varphi(A_t)<\infty$, so $h(A_t)>0$.
Then $M_t>1$ iff $h(A_t)X_t>1-\zeta(A_t)=w(A_t)$, and dividing by $h(A_t)$ gives $X_t>w(A_t)/h(A_t)=\varphi(A_t)$.
Where $\varphi(A_t)=\infty$ neither side can hold, since there $h(A_t)=0$ and $M_t=\zeta(A_t)\le1$.
This is~\eqref{eq:discrete-crossing-event}, pointwise: no hypothesis on the increments enters.

\emph{The decomposition.}  Write $a:=A_{s-1}$.
If $\Delta A_s=0$ then $A_s=a$ and $\Delta X_s=\Delta N_s$, so $\Delta M_s=h(a)\Delta X_s=h(a)\Delta N_s$,
while $R_s=h(a)\cdot0-(w(a)-w(a))=0$.
If $\Delta A_s>0$ then $X_s=0$ by the class-$(\Sigma)$ hypothesis, so, using $\zeta=1-w$,
\[
  \Delta M_s=\bigl(\zeta(A_s)-\zeta(a)\bigr)-h(a)X_{s-1}
            =-\bigl(w(A_s)-w(a)\bigr)-h(a)X_{s-1}.
\]
On the other side $\Delta N_s=\Delta X_s-\Delta A_s=-X_{s-1}-\Delta A_s$, so
\[
  h(a)\Delta N_s+R_s
  =-h(a)X_{s-1}-h(a)\Delta A_s
   +h(a)\Delta A_s-\bigl(w(A_s)-w(a)\bigr),
\]
the two $h(a)\Delta A_s$ terms canceling to leave the same quantity.
Summing gives~\eqref{eq:discrete-crossing-decomp}.
Each $h(A_{s-1})$ is $\cF_{s-1}$-measurable and bounded, and $N$ is a martingale, so each summand is integrable and the first sum is a martingale.  Since $w'=h$,
$R_s=h(a)\Delta A_s-\int_a^{A_s}h(z)\dd z=\int_a^{A_s}\bigl(h(a)-h(z)\bigr)\dd z$, from which $|R_s|\le\Delta A_s\sup_{z\in[a,A_s]}|h(a)-h(z)|$, bounded by $\Delta A_s$ times the variation of $h$ on $[a,A_s]$.
Those intervals are adjacent and disjoint, so their variations sum to at most $\mathrm{var}(h)$, giving $\sum_s|R_s|\le\|\Delta A\|_\infty\,\mathrm{var}(h)$.

\emph{The law.}  Optional stopping on the martingale part of~\eqref{eq:discrete-crossing-decomp} at the bounded time $T\wedge n$ gives $\E[M_{T\wedge n}]=M_0+\E\bigl[\sum_{s\le T\wedge n}R_s\bigr]$.
Before the crossing $X_t\le\varphi(A_t)$, so $h(A_t)X_t\le w(A_t)$ and $M_t\le\zeta(A_t)+w(A_t)=1$; with $C:=\|\Delta A\|_\infty\,\mathrm{var}(h)<\infty$ the stopped identity splits as $\E[(1+J)\ind\{T\le n\}]+\E[M_n\ind\{T>n\}]\le M_0+C$, and since the second term is nonnegative, monotone convergence makes $J\ind\{T<\infty\}$ integrable.
Let $n\to\infty$: the first term increases to $\PP(T<\infty)+\E[J\ind\{T<\infty\}]$, the second vanishes by dominated convergence ($M_n\le1$ on $\{T>n\}$ and $M_n\to0$ on $\{T=\infty\}$), and the right side converges to $M_0+\E\bigl[\sum_{s\le T}R_s\bigr]$ by dominated convergence with dominator $\sum_s|R_s|\le C$.
Hence $\PP(T<\infty)+\E[J\ind\{T<\infty\}]=M_0+\E\bigl[\sum_{s\le T}R_s\bigr]$, which rearranges to~\eqref{eq:discrete-crossing-law}.
\end{proof}

\begin{proof}[Proof of Theorem~\ref{thm:pacbayes}]
Apply Proposition~\ref{prop:mixture} with parameter space $\Theta$, prior $\mu = \pi$, and $E_t(\theta) = \exp(\lambda Y_t^\theta -
\overline\Psi_t^\theta(\lambda))$ at the fixed $\lambda$: each $E_t(\theta)$ is a
nonnegative supermartingale with $E_0(\theta)\le1$, so $\overline E_t = \int_\Theta E_t(\theta)\,\pi(\dd\theta)$ is one as well.
Dropping the nonnegative residual in~\eqref{eq:mixturepathwise} at the adapted posterior $\rho_t$ gives $\log\overline E_t \ge \lambda\,\E_{\rho_t}[Y_t^\theta] -
\E_{\rho_t}[\overline\Psi_t^\theta(\lambda)] - \KL(\rho_t\|\pi)$,
so the event in~\eqref{eq:pacbayes-bound} is contained in $\{\exists\,t : \overline E_t \ge e^x\}$, whose probability Ville's inequality (Corollary~\ref{cor:ville}) bounds by $e^{-x}$.
\end{proof}

\begin{proof}[Proof of Theorem~\ref{thm:ville-DV}]
Let $A := \{\sup_{t} M_t \geq \lambda\}$.
Applying Corollary~\ref{cor:DV} with $g = \ind_A$,
\[
  \log P(A)
  = \max_{Q : Q(A) = 1}\{-\KL(Q\|P)\}
  = -\min_{Q : Q(A) = 1}\KL(Q\|P),
\]
with minimizer $Q^* = P(\cdot \mid A)$; this is~\eqref{eq:ville-exact}.

For the inequality $P(A) \leq 1/\lambda$, the classical route proceeds independently of the DV variational identity.
Fix $\mu < \lambda$ and define the stopping time $\sigma_\mu := \inf\{t : M_t \geq \mu\}$; then $A \subseteq
\{\sigma_\mu < \infty\}$, a supremum of at least $\lambda$ exceeding $\mu$ at some finite time.  Optional stopping
(Theorem~\ref{thm:OST}) on the nonnegative supermartingale $M$ at $\sigma_\mu \wedge n$ gives $\E[M_{\sigma_\mu \wedge n}] \leq M_0 = 1$; since $M \geq 0$ and $M_{\sigma_\mu} \geq \mu$ on $\{\sigma_\mu \leq n\}$,
\[
  \mu\,\PP(\sigma_\mu \leq n)
  \leq \E[M_{\sigma_\mu}\ind\{\sigma_\mu \leq n\}]
  \leq \E[M_{\sigma_\mu \wedge n}] \leq 1.
\]
Sending $n \to \infty$ gives $\PP(A) \leq \PP(\sigma_\mu < \infty) \leq 1/\mu$, and letting $\mu \uparrow \lambda$ yields $\PP(A) \leq 1/\lambda$.
Combined with \eqref{eq:ville-exact}, this gives the quantitative information bound $\KL(P(\cdot|A)\|P) \geq \log\lambda$ (with equality iff the classical Ville bound is tight).
\end{proof}

\begin{proof}[Proof of Theorem~\ref{thm:doob}]
This is the classical Doob $L^p$ maximal inequality; we record the route.
With $M^* := \sup_{t \leq T} M_t$, Doob's maximal lemma $\PP(M^* \geq \lambda) \leq \lambda^{-1}\E[M_T\,\ind\{M^* \geq \lambda\}]$
(optional stopping at $\inf\{t : M_t \geq \lambda\}$) and the layer-cake formula give $\E[(M^*)^p] \leq \tfrac{p}{p-1}\,\E[M_T\,(M^*)^{p-1}]$;
H\"older with exponents $p$ and $p/(p-1)$ then yields the bound of Theorem~\ref{thm:doob}.
\end{proof}

\begin{proof}[Proof of Proposition~\ref{prop:ay-certificate}]
(i) Fix $t$ and write $y:=M^\ast_{t-1}$, so that $M^\ast_t=\max(y,M_t)$.
If $M_t\le y$ the running maximum does not move, $M^\ast_t=y$, and $D_{-\Phi}(y,y)=0$, while $A^\Phi_t-A^\Phi_{t-1}=\bigl[\Phi(y)-(y-M_t)\Phi'(y)\bigr]-\bigl[\Phi(y)-(y-M_{t-1})\Phi'(y)\bigr]=(M_t-M_{t-1})\Phi'(y)$,
which is~\eqref{eq:ay-increment} with the Bregman term absent.
If $M_t>y$ the maximum advances to $M_t$, so $M^\ast_t=M_t$ and $A^\Phi_t=\Phi(M_t)$; expanding the right side of~\eqref{eq:ay-increment} gives $(M_t-M_{t-1})\Phi'(y)-\Phi(y)+\Phi(M_t)-(M_t-y)\Phi'(y)=\Phi(M_t)-\Phi(y)+(y-M_{t-1})\Phi'(y)$,
which is $A^\Phi_t-A^\Phi_{t-1}$.
The two cases exhaust the step, so~\eqref{eq:ay-increment} holds pathwise; no continuity of the maximum is used, only that a discrete-time maximum advances by taking the current value.
For the supermartingale property take $\E[\,\cdot\mid\cF_{t-1}]$ in~\eqref{eq:ay-increment}: the process $M^\ast_{t-1}$ is $\cF_{t-1}$-measurable, so the first term has conditional expectation $\Phi'(M^\ast_{t-1})\,\E[M_t-M_{t-1}\mid\cF_{t-1}]$, which vanishes for a martingale, and the second is nonnegative because $-\Phi$ is convex.

(ii) Sum~\eqref{eq:ay-increment} over $t=1,\dots,T$ and take expectations.
The transform term drops by the display above, $A^\Phi_0=\Phi(M_0)$ since $M^\ast_0=M_0$, and $A^\Phi_T=\Phi(M^\ast_T)-(M^\ast_T-M_T)\Phi'(M^\ast_T)$; rearranging gives~\eqref{eq:ay-identity}.

(iii) At $\Phi(y)=-y^{p}/(p-1)$ one has $\Phi'(y)=-q\,y^{p-1}$, so $\Phi(y)-(y-x)\Phi'(y)=-y^{p}/(p-1)+q\,y^{p}-q\,y^{p-1}x=y^{p}-q\,y^{p-1}x$,
using $q-1/(p-1)=1$; this is $U(x,y)$ of~\eqref{eq:doob-bellman}, so $A^\Phi_t=U(M_t,M^\ast_t)$.
Rearranging~\eqref{eq:ay-identity}, the optional-stopping deficit $A^\Phi_0-\E[A^\Phi_T]$ is the Bregman sum.

For a submartingale with $\Phi$ nonincreasing, $\Phi'\le0$ and $\E[M_t-M_{t-1}\mid\cF_{t-1}]\ge0$ make the conditional expectation of the transform term nonpositive, so $A^\Phi$ remains a supermartingale and $A^\Phi_0-\E[A^\Phi_T]=\E\bigl[\sum_t D_{-\Phi}(M^\ast_t,M^\ast_{t-1})\bigr]-\E\bigl[\sum_t(M_t-M_{t-1})\Phi'(M^\ast_{t-1})\bigr]$, the subtracted term being nonpositive.
\end{proof}

\begin{proof}[Proof of Corollary~\ref{cor:doob-residual}]
(i) The function $\Phi(y)=-y^{p}/(p-1)$ is $C^1$, concave, and nonincreasing on $(0,\infty)$, so the submartingale form of Proposition~\ref{prop:ay-certificate} applies; every expectation involved is finite because $\E[(M^\ast)^p]\le q^p\,\E[M_T^p]<\infty$ by Theorem~\ref{thm:doob}.
By part (iii) the certificate process is $A^\Phi_t=U(M_t,M^\ast_t)$, a supermartingale, and its optional-stopping deficit is the Bregman record sum plus the nonnegative drift term $-\E\bigl[\sum_t(M_t-M_{t-1})\Phi'(M^\ast_{t-1})\bigr]$; in particular $\delta_{\mathrm B}=U(M_0,M_0)-\E[U(M_T,M^\ast_T)]=A^\Phi_0-\E[A^\Phi_T]\ge0$.
Evaluating $U$ at the endpoints---$U(M_0,M_0)=-M_0^p/(p-1)$ and $U(M_T,M^\ast_T)=(M^\ast)^{p}-q\,(M^\ast)^{p-1}M_T$---gives the displayed form of $\delta_{\mathrm B}$.
(ii) Solving~\eqref{eq:delta-B} for $\E[(M^\ast)^p]$, substituting $\E[(M^\ast)^{p-1}M_T]=\E[(M^\ast)^p]^{(p-1)/p}\E[M_T^p]^{1/p}-\delta_{\mathrm H}$,
and dividing by $\E[(M^\ast)^p]^{(p-1)/p}$ gives~\eqref{eq:doob-residual};
$\delta_{\mathrm H}\ge0$ is H\"older with exponents $q$ and $p$ applied to $(M^\ast)^{p-1}$ and $M_T$.
\end{proof}

\begin{proof}[Proof of Proposition~\ref{prop:power-certificate}]
Hypothesis (i) makes $\delta_{\mathrm M}$ the expectation of a nonnegative random variable, and (ii) with $U_0$ deterministic gives $\E[U_T]\le U_0$, so $\delta_{\mathrm S}\ge0$.  Adding the definitions of $\delta_{\mathrm M}$ and $\delta_{\mathrm S}$,
\[
  \delta_{\mathrm M}+\delta_{\mathrm S}
  = \bigl(\E[U_T]-\E[\Gamma]+C\,\E[\Xi]\bigr)
    + \bigl(U_0-\E[U_T]\bigr)
  = C\,\E[\Xi]-\E[\Gamma]+U_0,
\]
which rearranges to~\eqref{eq:power-certificate}.
Both deficits being nonnegative, the right side of~\eqref{eq:power-certificate} is nonnegative as soon as $U_0\le0$.
\end{proof}

\begin{proof}[Proof of Lemma~\ref{lem:geometric-mixture}]
By the conditional weighted AM--GM / H\"older inequality,
$\E[\prod_i (M_t^{(i)})^{\alpha_i}\mid \cF_{t-1}]
\leq \prod_i (\E[M_t^{(i)}\mid\cF_{t-1}])^{\alpha_i}
= \prod_i (M_{t-1}^{(i)})^{\alpha_i} = M_{t-1}^{(\alpha)}$.
\end{proof}

\begin{proof}[Proof of Corollary~\ref{cor:pooling-benefit}]
By Lemma~\ref{lem:geometric-mixture}, $M^{(\alpha)}$ is a nonnegative supermartingale with $M_0^{(\alpha)} = 1$.
Optional stopping at the bounded time $\tau^{(\alpha)}_x \wedge n$ gives $\E[M_{\tau^{(\alpha)}_x \wedge n}^{(\alpha)}]
\leq 1$.  On $\{\tau^{(\alpha)}_x \leq n\}$ one has $M_{\tau^{(\alpha)}_x}^{(\alpha)} \geq x$,
so
\[
  x\,\PP(\tau^{(\alpha)}_x \leq n)
  \leq \E\bigl[M_{\tau^{(\alpha)}_x}^{(\alpha)}\ind\{\tau^{(\alpha)}_x \leq n\}\bigr];
\]
letting $n \to \infty$ (monotone convergence on the right) yields $x\,\PP(\tau^{(\alpha)}_x < \infty) \leq g_x(\alpha)$, i.e.\ the first inequality of \eqref{eq:pooling}.
The second follows from $g_x(\alpha) \leq
\E[M_{\tau^{(\alpha)}_x \wedge n}^{(\alpha)}] \leq 1$.  When the $M^{(i)}$ never
disagree pathwise, $M^{(\alpha)}$ is an exact martingale and the bound reduces to Ville's $1/x$, with $g_x(\alpha) = 1$ when the stopped mixture is uniformly integrable.
Pathwise disagreement makes $M^{(\alpha)}$ a strict supermartingale that sheds mass before $\tau^{(\alpha)}_x$, so $g_x(\alpha) < 1$ and the bound is strictly sharper.
Optional-stopping leakage can leave $g_x(\alpha) < 1$ on its own.
\end{proof}

\begin{proof}[Proof of Proposition~\ref{prop:pooling-exact}]
Optional stopping on $N$ at $\tau^{(\alpha)}_x\wedge n$ gives $1=\E[M^{(\alpha)}_{\tau^{(\alpha)}_x\wedge n}]+\E[A_{\tau^{(\alpha)}_x\wedge n}]$. Let $n\to\infty$:
$\E[M^{(\alpha)}_{\tau^{(\alpha)}_x}\ind\{\tau^{(\alpha)}_x\le n\}]\uparrow g_x(\alpha)$ (monotone convergence); $\E[M^{(\alpha)}_n\ind\{\tau^{(\alpha)}_x>n\}]\to
\E[M^{(\alpha)}_\infty\ind\{\tau^{(\alpha)}_x=\infty\}]$ (there $M^{(\alpha)}_n<x$, dominated);
and $\E[A_{\tau^{(\alpha)}_x\wedge n}]\uparrow\E[A_{\tau^{(\alpha)}_x}]$.
Rearranging gives \eqref{eq:pooling-exact}.
\end{proof}

\begin{proof}[Proof of Corollary~\ref{cor:pooling-crossing-exact}]
On $\{\tau^{(\alpha)}_x<\infty\}$ one has $M^{(\alpha)}_{\tau^{(\alpha)}_x}\ge x$, so $M^{(\alpha)}_{\tau^{(\alpha)}_x}=x+J^{(\alpha)}_x$ with $J^{(\alpha)}_x\ge0$ there and
\[
  g_x(\alpha)
  = \E\bigl[M^{(\alpha)}_{\tau^{(\alpha)}_x}\ind\{\tau^{(\alpha)}_x<\infty\}\bigr]
  = x\,\PP(\tau^{(\alpha)}_x<\infty) + \E\bigl[J^{(\alpha)}_x\ind\{\tau^{(\alpha)}_x<\infty\}\bigr].
\]
Substituting $g_x(\alpha)=1-\E[A_{\tau^{(\alpha)}_x}]-\E[M^{(\alpha)}_\infty\ind\{\tau^{(\alpha)}_x=\infty\}]$ from~\eqref{eq:pooling-exact} and dividing by $x$ gives~\eqref{eq:pooling-crossing-exact}.
Each subtracted term is nonnegative,
so each vanishes exactly where the corresponding inequality of~\eqref{eq:pooling} is an equality.
At $W=1$ the geometric mixture is $M^{(1)}$ itself, a martingale, so its compensator $A$ vanishes identically;
adding $M^{(1)}_t\to0$ a.s.\ kills $M^{(1)}_\infty$ and leaves $\PP(\tau^{(\alpha)}_x<\infty)=\bigl(1-\E[J^{(\alpha)}_x\ind\{\tau^{(\alpha)}_x<\infty\}]\bigr)/x$,
which is~\eqref{eq:fp-exact}.
\end{proof}

\begin{proof}[Proof of Proposition~\ref{prop:renyi-anticipation}]
Take the single factor $\Pi_{1,T} = \Amart{\tau}_T$ in~\eqref{eq:pathZ},
so $Z_T(\alpha) = \E[(\Amart{\tau}_T)^\alpha]$ and $\Phi_T = \log Z_T$ is convex by Proposition~\ref{prop:grad}.  At $\alpha=1$,
$\Amart{\tau}$ is a nonnegative martingale with $\Amart{\tau}_0=1$
(Theorem~\ref{thm:surv-factor}(ii)), so $Z_T(1)=1$ and $\Phi_T(1)=0$.
At $\alpha=0$, with $0^0:=0$, $Z_T(0)=P(\Amart{\tau}_T>0)\le1$.
Differentiating by~\eqref{eq:grad} at $\alpha=1$, where the Gibbs law has density $\dd Q_A/\dd P = \Amart{\tau}_T/Z_T(1) = \Amart{\tau}_T$, gives $\Phi_T'(1) = \E_{Q_A}[\log \Amart{\tau}_T]
= \E[\Amart{\tau}_T\log \Amart{\tau}_T] = \KL(Q_A\|P|_{\cF_T})$.
If $\Phi_T\equiv0$ then $\Phi_T'(1)=0$, so $\KL(Q_A\|P|_{\cF_T})=0$ and $\Amart{\tau}_T=1$ almost surely; the converse is immediate.
Finally $\Phi_T''(\alpha) = \Var_{Q_T^\alpha}(\log \Amart{\tau}_T)$ by~\eqref{eq:grad}, so $\Phi_T$ is affine exactly when $\log \Amart{\tau}_T$ is $Q_T^\alpha$-a.s.\ constant for every $\alpha$.
Each $Q_T^\alpha$ is equivalent to $P$ restricted to $B := \{\Amart{\tau}_T>0\}$, so this says $\Amart{\tau}_T = c\,\ind_B$ for a constant $c$, and $\E[\Amart{\tau}_T]=1$ forces $c = 1/P(B)$; otherwise $\Phi_T''>0$ and $\Phi_T$ is strictly convex.
\end{proof}

\begin{proof}[Proof of Proposition~\ref{prop:renyi-wealth}]
Take the single factor $\Pi_{1,T}=M_T$ in~\eqref{eq:pathZ}, so $Z_T(s)=\E_P[M_T^{\,s}]$ and $\Phi_T=\log Z_T$ is convex by Proposition~\ref{prop:grad}.
At $s=1$ the martingale property with $M_0=1$ gives $Z_T(1)=1$ and $\Phi_T(1)=0$, so the Gibbs law at $s=1$ has density $\dd Q_M/\dd P=M_T/Z_T(1)=M_T$, and~\eqref{eq:grad} gives $\Phi_T'(1)=\E_{Q_M}[\log M_T]=\E_P[M_T\log M_T]=\KL(Q_M\|P)$.
At $s=0$ the Gibbs law is $P$ itself on $\{M_T>0\}$, so $\Phi_T'(0)=\E_P[\log M_T]$, which is $-\KL(P\|Q_M)$ because $\dd P/\dd Q_M=1/M_T$ there.
Subtracting the two gives $\Phi_T'(1)-\Phi_T'(0)=\KL(Q_M\|P)+\KL(P\|Q_M)$, the Jeffreys divergence between the two laws.
For a likelihood-ratio stream $M_T=\prod_{t\le T}R_t$, peel one step at a time: $\E_P[\prod_{t\le T}R_t^{\,s}]=\E_P\bigl[\prod_{t\le T-1}R_t^{\,s}\,\E_P[R_T^{\,s}\mid\cF_{T-1}]\bigr]$ by the tower rule, and when the conditional cumulant at $s$ is deterministic the inner factor is the constant $\E_P[R_T^{\,s}]$ and pulls out of the expectation; iterating over $t$ gives $Z_T(s)=\prod_{t\le T}\E_P[R_t^{\,s}]$, the additivity claimed.
\end{proof}

\begin{proof}[Proof of Lemma~\ref{lem:Rt}]
For $A \in \cF_t$,
$R(A \times \{t\}) = \E[\ind_A P(\tau=t\mid\cF_t)]
= \E[\ind_A \Amart{\tau}_t \Delta \clock{\tau}_t] = \int_{A \times \{t\}} \Amart{\tau}_t \,\dd\widehat P$,
using Theorem~\ref{thm:surv-factor}(v).
Integrating $V_t$ against $R$ recovers~\eqref{eq:rep}.
\end{proof}

\begin{proof}[Proof of Theorem~\ref{thm:pathtime}]
By Lemma~\ref{lem:Rt},
$Z_\tau(\alpha) = \int_{\widehat\Omega} \Amart{\tau}_t(\omega)\prod_i
\Pi_{i,t}(\omega)^{\alpha_i}\,\widehat P(\dd\omega,\dd t)$.
Apply Corollary~\ref{cor:DV} on $(\widehat\Omega,\widehat P)$ with $g(\omega,t) = \Amart{\tau}_t(\omega)\prod_i \Pi_{i,t}(\omega)^{\alpha_i}$;
$\log g = \log \Amart{\tau}_t + \sum_i \alpha_i \log \Pi_{i,t}$ yields~\eqref{eq:pathtime}.
\end{proof}

\begin{proof}[Proof of Theorem~\ref{thm:peeking}]
Equation~\eqref{eq:peeking} is Theorem~\ref{thm:pathtime} with one factor $\Pi_t = E_t(\lambda)$.
For the tail bound, Markov's inequality applied to $\exp(\lambda Y_\tau - \overline\Psi_\tau(\lambda))$ gives $P(A) \leq e^{-\lambda x + c}\,\E[e^{\lambda Y_\tau - \overline\Psi_\tau(\lambda)}]
= e^{-\lambda x + c + \cP_\tau(\lambda)}$.
\end{proof}

\begin{proof}[Proof of Proposition~\ref{prop:vanish}]
\emph{(a)} Optional stopping: $\E[E_\tau] \leq E_0 \leq 1$, so $\cP_\tau \leq 0$.
\emph{(b)} Pseudo-stopping time satisfies $\E[B_\tau] = \E[B_0]$ for every bounded martingale $B$; by uniform integrability the same holds for $E_t$, so $\E[E_\tau] = \E[E_0] = 1$ and $\cP_\tau = 0$.
\end{proof}

\begin{proof}[Proof of Lemma~\ref{lem:sup-nonintegrable}]
The martingale converges a.s.\ to $E_\infty = 0$, so $\E[E_\infty] = 0
\neq 1 = E_0$ and $(E_t)$ is not closed, hence not uniformly
integrable.  Were $\E[\sup_t E_t]$ finite, $(E_t)$ would be dominated by the integrable variable $\sup_t E_t$ and therefore uniformly integrable---a contradiction.
Hence $\E[\sup_t E_t] = \infty$.
No no-overshoot or continuity hypothesis is needed: the tail bound $\PP(\sup_t E_t \geq x) \leq 1/x$ of Theorem~\ref{thm:ville-tight} is an \emph{upper} bound and does not by itself force divergence, but the failure of uniform integrability does.
\end{proof}

\begin{proof}[Proof of Proposition~\ref{prop:infinite}]
Since $E_{\tau^\star} = \sup_{t\ge1} E_t \ge \sup_{t\ge0} E_t - 1$,
Lemma~\ref{lem:sup-nonintegrable} gives $\E[E_{\tau^\star}] = \infty$ and hence $\cP_{\tau^\star} =
\infty$.  In the continuous-path idealization, under the no-overshoot
hypothesis of Theorem~\ref{thm:ville-tight}, $E_{\tau^\star}$ is Pareto$(1)$ and $\int_1^\infty x^{-1}\,\dd x = \infty$ gives the divergence directly; the lemma shows it persists for every discrete-time anticipatory time, where only $\PP(E_{\tau^\star}\geq x)\leq 1/x$ is available.
\end{proof}

\begin{proof}[Proof of Theorem~\ref{thm:composite-peeking}]
For each $P\in\cP_0$, $\E_P[E_\tau]=\E_{\widehat P_P}[E_t \Amart{\tau,P}_t]$ (Theorem~\ref{thm:rep}), and Theorem~\ref{thm:peeking} rewrites $\log\E_P[E_\tau]=\cP^P_\tau$ as the identity~\eqref{eq:composite-peeking-simple}.
Taking $\sup_{P\in\cP_0}$ of both sides gives the upper envelope.
\end{proof}

\begin{proof}[Proof of Proposition~\ref{prop:seq-PAC}]
Proposition~\ref{prop:multi-PAC} applies verbatim with $\cX$ the path space $\Omega_T$.
Decomposition~\eqref{eq:seq-coin-decomp} follows from the chain rule (Lemma~\ref{lem:chain-rule}) applied to each relative-entropy term and swapping sums (the pointwise minimizer at step $t$ depends only on the step-$t$ conditionals).

For the equality case, factor $\prod_w\pi_{w,t}^{\alpha_w}=z_t(x_{1:t-1})\,q^*_t(x_t\mid x_{1:t-1})$, in which $q^*_t$ is a probability density by the definition of $z_t$.
Multiplying over $t$ and integrating gives $Z(\alpha)=\E_{Q^*}\bigl[\prod_{t=1}^T z_t\bigr]$, so the left side of~\eqref{eq:seq-coin-decomp} is $-\log\E_{Q^*}\bigl[\prod_{t=1}^T z_t\bigr]$ and the right side is $\E_{Q^*}\bigl[-\log\prod_{t=1}^T z_t\bigr]$.
The two are related by Jensen's inequality for the strictly convex $-\log$, which holds with equality exactly when $\prod_{t=1}^T z_t$ is $Q^*$-almost surely constant.
If each $\pi_w$ is a product measure over time then $z_t$ does not depend on $x_{1:t-1}$ and the product is constant, from which $\cC_\alpha(\pi_{1:W})=\sum_{t=1}^T(-\log z_t)=\sum_{t=1}^T\cC_\alpha(\pi_{1:W,t})$.
\end{proof}

\begin{proof}[Proof of Proposition~\ref{prop:budget-relent}]
By Lemma~\ref{lem:Rt}, $\dd R/\dd\widehat P = \Amart{\tau}_t$ on the path-time space and $\E[V_\tau] = \E_R[V_t]$ for every adapted $V$ for which either side is defined.
Taking $V_t = \log \Amart{\tau}_t$,
$\E[\log \Amart{\tau}_\tau] = \E_R[\log \Amart{\tau}_t]
= \E_R[\log(\dd R/\dd\widehat P)] = \KL(R\|\widehat P)$,
the last equality being~\eqref{eq:finite-entropy} applied to the finite measure $\widehat P$.
Hence $\cT_B$ is the set of random times whose path-time law lies in the $B$-ball around $\widehat P$, and $\mathfrak{C}_{\cT_B}(E) = \sup_{\tau\in\cT_B}\log\E[E_\tau]
= \log\sup_{\tau\in\cT_B}\E_R[E_t] = \log V(B)$,
the supremum being over exactly the path-time laws the constraint admits.
\end{proof}

\begin{proof}[Proof of Theorem~\ref{thm:peeking-radius}]
Write $\pi_t := P(\tau = t \mid \cF_t)$, so that $\E[E_\tau]=\sum_{t\ge1}\E[E_t\pi_t]$.
Let $E=M-C$ be the Doob decomposition, with $C$ predictable, increasing and $C_0=0$; then $M=E+C$ is a nonnegative martingale with $M_0=E_0\le1$ and $E\le M$ pointwise.
Each $\pi_t$ is $\cF_t$-measurable, so $\E[M_t\pi_t]=\E[M_T\pi_t]$ for $t\le T$, and therefore
\[
  \sum_{t\le T}\E[E_t\pi_t]
  \;\le\; \E\Bigl[M_T\sum_{t\le T}\pi_t\Bigr]
  \;\le\; \E[M_T]\,\esssup\mathfrak{A}^\tau
  \;\le\; \esssup\mathfrak{A}^\tau ,
\]
and letting $T\to\infty$ gives $\E[E_\tau]\le\esssup\mathfrak{A}^\tau$.
For the reverse, fix $\varepsilon>0$, let $S:=\{\mathfrak{A}^\tau>\esssup\mathfrak{A}^\tau-\varepsilon\}$, which has positive probability, and take $E_t:=P(S\mid\cF_t)/P(S)$, a nonnegative martingale with $E_0=1$.
Then $\E[E_\tau]=\sum_{t\ge1}\E[\ind_S\,\pi_t]/P(S)=\E[\mathfrak{A}^\tau\ind_S]/P(S)\ge\esssup\mathfrak{A}^\tau-\varepsilon$.
\end{proof}

\begin{proof}[Proof of Corollary~\ref{cor:peeking-ball}]
The first claim is~\eqref{eq:peeking-radius} read through the logarithm; the second is the same supremum evaluated at a time of the class with $\esssup\mathfrak{A}^\tau=\beta$.
For the third, $\Amart{\tau}\equiv1$ gives $\mathfrak{A}^\tau=\sum_{t\ge1}\Delta\clock{\tau}_t=1$, since $1-\clock{\tau}_T=S_T=P(\tau>T\mid\cF_T)$ decreases to $0$ for a finite time.
Conversely $\esssup\mathfrak{A}^\tau=1$ makes $\E[E_\tau]\le1$ for every nonnegative martingale with $E_0=1$; applied to $E=1+cN$ for a bounded martingale $N$ with $N_0=0$ and $|c|$ small enough that $1+cN\ge0$, at both signs of $c$, this forces $\E[N_\tau]=0$.
\end{proof}

\begin{proof}[Proof of Theorem~\ref{thm:pathtime-ct}]
By the Az\'ema multiplicative decomposition $S_t=(1-\clock{\tau}_t)\Amart{\tau}_t$ and $\dd R/\dd\widehat P=\Amart{\tau}_t$, so $\E[g_\tau]=\E_{\widehat P}[\Amart{\tau}_t\,g_t]$ for adapted $g$.
For a finite random time $\widehat P$ is a finite positive measure of total mass $\E[\clock{\tau}_\infty]\le1$,
so the finite-measure Donsker--Varadhan formula (Corollary~\ref{cor:DV}) applies to the nonnegative factor $g_t=\Amart{\tau}_t\prod_i\Pi_{i,t}^{\alpha_i}$, whose logarithm is $\log \Amart{\tau}_t+\sum_i\alpha_i\log\Pi_{i,t}$; the supremum runs over $Q\ll\widehat P$ with $Q(\{g_t=0\})=0$, and $\int g_t\,\dd\widehat P
=\E_{\widehat P}[\Amart{\tau}_t\prod_i\Pi_{i,t}^{\alpha_i}]
=\E[\prod_i\Pi_{i,\tau}^{\alpha_i}]=Z_\tau(\alpha)$ gives~\eqref{eq:pathtime-ct}.
The factor $\Amart{\tau}_t$ may vanish on a $\widehat P$-positive set, which is the case Corollary~\ref{cor:DV} is stated to cover.
\end{proof}

\begin{proof}[Proof of Proposition~\ref{prop:no-bdg-general}]
Fix $A\in\cF$ and take $\tau:=\ind_A$.  Since $M_0=0$ and $\langle M\rangle_0=0$,
$M_\tau=M_1\ind_A$ and $\langle M\rangle_\tau^{1/2}=\langle M\rangle_1^{1/2}\ind_A$, so the hypothesis reads $\E\bigl[(|M_1|-C\,\langle M\rangle_1^{1/2})\ind_A\bigr]\le0$.
Taking $A=\{|M_1|>C\,\langle M\rangle_1^{1/2}\}$ forces $\PP(A)=0$.
\end{proof}

\begin{proof}[Proof of Theorem~\ref{thm:bdg-pseudo}]
In the progressive enlargement $\cG_t=\cF_t\vee\sigma(\tau\wedge t)$ the pseudo-stopping property makes $(M_{t\wedge\tau})_{t\ge0}$ a $\cG$-local martingale, $\tau$ is a $\cG$-stopping time, and the quadratic variation is unchanged by the enlargement.
The classical inequalities applied in $\cG$ at the stopping time $\tau$ give~\eqref{eq:bdg-pseudo} with their own constants~\cite{Nikeghbali2006essay}.
\end{proof}

\begin{proof}[Proof of Corollary~\ref{cor:honest-inflation-law}]
The Brownian filtration satisfies~(C), so with avoidance the hypotheses are~(CA), under which $Z$ is continuous and $I_\tau=\inf_{u\le\tau}Z_u$.
For an honest time under~(CA) that infimum is uniform on $[0,1]$~\cite{Nikeghbali2006essay}, and $-\log U$ is standard exponential for $U$ uniform.
Hence $\Upsilon_\tau^2=1+\log(1/I_\tau)$ has mean $2$.
\end{proof}

%=============================================================================

% ======================================================================

\section*{Acknowledgments}

Large language models were used in producing this work, for which the author is solely responsible.

\bibliographystyle{plainnat}
\bibliography{pathwise_martingale}

\end{document}